\documentclass{siamart250211}
\usepackage{amsmath, amssymb}
\usepackage{subfigure}
\usepackage{enumitem}
\usepackage{geometry,xcolor,comment}
\newcommand{\eps}{{\displaystyle \varepsilon}}
\newcommand{\bsub}{\begin{subequations}}
\newcommand{\esub}{\end{subequations}$\!$}
\newcommand{\ds}[0]{\displaystyle}
\newcommand{\bs}[0]{\boldsymbol}
\usepackage{hyperref,url}
\usepackage[normalem]{ulem}
\newcommand{\al}[1]{\textcolor{black}{#1}}
\newcommand{\jgh}[1]{\textcolor{black}{#1}}

\newcommand{\ajb}[1]{\textcolor{black} {{#1}}}

\newcommand{\littleoh}{ \mbox{{\scriptsize $\mathcal{O}$}}}
\newcommand{\bigoh}{\mathcal{O}}

\newcommand{\braket}[1]   {\left<{#1}\right>}
\newcommand{\bx}{\mathbf{x}}
\newcommand{\Hc}{\mathcal{H}}
\newcommand{\Dc}{\mathcal{Q}}

\newcommand{\by}{\mathbf{y}}

\usepackage{subcaption,graphicx}

\theoremstyle{remark}
\newtheorem*{remark}{Remark}

\newcommand{\bn}{{\bf n}}
\newcommand{\br}{{\bf r}}

\newcommand{\bd}{{\bf d}}

\newcommand{\vx}{\mathbf{x}}

\newcommand{\Om}{\Omega}
\newcommand{\pOm} {\partial \Omega}

\newcommand{\cH}{{\mathcal{H}}}

\newcommand{\cS}{{\mathcal{S}}}

\newcommand\RR{{\mathbb R}}

\newcommand{\hn} {{\hat{\textbf{n}}}}
\newcommand{\htan} {{\hat{\textbf{t}}}}
\newcommand{\Q}{Q}
\newcommand{\cda}{c_{30}}
\newcommand{\ccb}{c_{21}}
\newcommand{\cbc}{c_{12}}
\newcommand{\cad}{c_{03}}
\newcommand{\rhoh}{{\hat{\rho}}} 
\newcommand{\etah}{{\hat{\eta}}} 
\newcommand{\rh}{{\hat{r}}} 

\title{The three-dimensional Neumann Green's function for general surfaces: singular asymptotics and boundary integral methods}

\author{Alan E. Lindsay \thanks{University of Notre Dame, Notre Dame, IN, 46556, USA.
  Email:
   {\tt a.lindsay@nd.edu } (corresponding author)}, \and Andrew J. Bernoff \thanks{Harvey Mudd College, Claremont, CA, USA Email: {\tt bernoff@g.hmc.edu}}\and Tristan Goodwill \thanks{Department of Statistics, University of Chicago, USA. Email: {\tt tgoodwill@uchicago.edu} } \and Jeremy G. Hoskins \thanks{Department of Statistics, University of Chicago, USA and NSF-Simons National Institute for Theory and Mathematics in
Biology, Chicago, IL. Email: {\tt jeremyhoskins@uchicago.edu}}}

\begin{document}

\label{firstpage}
\maketitle

\baselineskip=12pt

\begin{abstract}
    We present an asymptotic analysis and high-order boundary integral method for the three-dimensional Neumann Green's function in general \al{closed and smooth} geometries. The Neumann Green's function is a fundamental quantity which arises in numerous fields of science and engineering. In the application of singular perturbation methods to strongly localized reactions and diffusive transport, the Green's function plays the key role in mediating global dynamics. However, this essential quantity can only be determined in closed-form for a limited set of geometries. The Green's function for the Laplacian is an elliptic problem with a Dirac forcing term. Accurate resolution of the solution requires a careful decomposition into a singular and a regular part. The bulk scenario is where the source is placed off surface and the singularity is given by the free-space function. In the surface case, where the source is placed at a curved point on the boundary, we use asymptotic analysis to determine a three-term singularity structure. With explicit knowledge of these singularities, we develop a high-order boundary integral method for the determination of the remaining regular part. To resolve the singular boundary data, our integral method uses a custom discretization with Duffy patches near the source. We validate our method using several test cases in which closed-form solutions can be developed, including spheres, prolate spheroids and constructed domains. We demonstrate the applicability of our method to address some open problems in narrow capture theory.
\end{abstract}


\section{Introduction}

The Neumann Green's function for the Laplacian (or Neumann function) is a fundamental quantity with applications across science and engineering, with examples \jgh{including} electrostatics, acoustics, optics and diffusive processes \cite{Prolate1,Dassios2012,prolate2,Silbergleit2003,SurfaceGreen3D,SingerHolcman2006,cheviakov2010asymptotic,CHAILLAT2022,GrebenkovWard2026,Garabedian1964,BergmanSchiffer1953,carillo2026}. For a domain $\Omega\subset\mathbb{R}^3$ with \al{a smooth closed} boundary $\partial\Omega$, and a source point $\bx_0$, the problem for the Neumann function can be posed either interior or exterior to $\Omega$, with a bulk source ($\bx_0\notin\partial\Omega$) or a surface source ($\bx_0\in\partial\Omega$). These choices generate four related but distinct functions that have different normalization and singularity considerations. \al{The methodologies developed in this work provide a unified framework for the evaluation of all four functions, however,} our presentation will largely focus on the surface case which is the most significant challenge. In the case of a surface source $\bx_0\in\partial\Omega$, the exterior Neumann function $G^e(\bx;\bx_0)$ for the Laplacian satisfies the boundary value problem
\bsub\label{eq:eqnG_exterior}
\begin{equation}
  \label{eq:eqnG_exterior_a}  \Delta G^e = 0 \quad \text{in}\quad \mathbb{R}^3 \setminus \Omega, \qquad \partial_n G^e = -\delta(\vx - \vx_0) \quad \text{on} \quad \partial \Omega. 
\end{equation}
The formulation \eqref{eq:eqnG_exterior} assumes that the Dirac distribution is two-dimensional in the surface coordinates. The choice of the negative sign implies a unit flux into the surface and imposes the far field solvability condition 
\begin{equation}\label{eq:eqnG_exterior_c}
    G^e(\vx;\vx_0) \sim \frac{1} {4 \pi |\vx|}
\qquad \mathrm{as} \qquad {|\vx| \to \infty}.
\end{equation}
\esub
The \emph{interior} surface Neumann function satisfies
\bsub\label{eq:eqnG_interior}
\begin{equation}
 \label{eq:eqnG_interior_a}   \Delta G^i = \frac{1}{|\Omega|} \quad \text{in} \quad \Omega; \qquad 
   \partial_n G^i = -\delta(\vx - \vx_0) \quad \text{on} \quad \partial \Omega.
\end{equation}
The solution is unique up to a constant translation which can be fixed by the condition
\begin{equation}\label{eq:eqnG_interior_c}
\int_{\Omega}G^i(\bx;\bx_0)\textrm{d}V(\bx) = 0.
\end{equation}
\esub
In the above formulations $\partial_n \equiv \hn \cdot \nabla$ is the normal derivative to $\partial\Omega,$ with $\hn$ the outward pointing normal. The presence of the Dirac forcing term means that the Neumann function is singular as $\bx\to\bx_0$. A key step in resolving the Neumann function is to determine the correct singularity structure and establish the decomposition
\begin{equation}\label{eqn:Sing_G}
    G^{i,e}(\vx; \vx_0) = G^{i,e}_{\text{sing}}(\vx;\vx_0) + R^{i,e}(\vx; \vx_0).
\end{equation}
In the formulation \eqref{eqn:Sing_G} $G^{i,e}_{\text{sing}}(\bx;\bx_0)$ is singular as $\vx \to \vx_0$ and $R^{i,e}(\vx; \vx_0)$ is bounded as $\bx\to\bx_0$. \al{As mentioned previously, the \emph{bulk} case is relatively straightforward. For a given $\bx_0\notin\partial\Omega$}, we have that
\begin{equation}\label{eq:FreeSpace}
    G_{\text{sing}}^{i,e} (\bx;\bx_0) = G_F(\bx-\bx_0), \qquad G_F(\bx) = \frac{1}{4\pi|\bx|}.
\end{equation}
which is the standard free space Green's function for the Laplacian in three dimensions. The surface case ($\bx_0\in\partial\Omega$) dramatically alters the nature of the singular solution behavior. The precise behavior of the singular component $G^{i,e}_{\text{sing}}(\bx;\bx_0)$ for $\bx_0\in\partial\Omega$ has been the topic of several studies due \al{to} its centrality in the asymptotic solution of the narrow escape problem describing the escape rate of Brownian particles \al{through small boundary windows} \cite{cheviakov2010asymptotic,SingerHolcman2006,lindsay2017first,TargetSearch2024,GrebenkovWard2026,Bressloff2024,carillo2026} and in applications to superconductivity \cite{SurfaceGreen3D,Silbergleit2003}. The recent study of \cite{NURSULTANOV2021202} used a pseudodifferential operator approach to identify a triple singularity structure featuring a monopole term twice the strength of the free-space term \eqref{eq:FreeSpace}, a logarithmic singularity with strength proportional to the mean curvature at $\bx_0$ and a weaker path-dependent singularity whose strength is proportional to the difference of principal curvatures at $\bx_0$. \al{ In Sec.~\ref{sec:apps}, we highlight the particular importance of evaluating $R^{i,e}(\bx_0;\bx_0)$ in asymptotic analysis of narrow capture problems \cite{GrebenkovWard2026,lindsay2017first,NURSULTANOV2021202}.}

\paragraph{Singularity structure of the Green's function} We provide here a derivation of the singular behavior $G^{i,e}_{\text{sing}}(\bx;\bx_0)$ as $\bx\to\bx_0$ for $\bx_0\in\partial\Omega$ using a local expansion of the Laplacian in tangent plane coordinates near $\bx_0$. Our analysis re-derives the result obtained in \cite{NURSULTANOV2021202}, but with a formulation that is adaptable to boundary integral equation methods for accurate evaluation of $R^{i,e}(\bx;\bx_0)$. The result takes the form
\bsub\label{GSing_surf}
\begin{equation}
    G^{i,e}_{\text{sing}}(\bx;\bx_0) \sim \frac{1}{2\pi|\bx-\bx_0|} \mp \frac{\Hc(\bx_0)}{4\pi} \log \big[\, |\bx-\bx_0| + \al{q}\, \big] \pm \mathcal{Q}(\bx_0)\, e(\bx;\bx_0), \quad \mbox{as} \quad \bx\to\bx_0,
\end{equation}
where $\al{q} = (\bx-\bx_0)\cdot \hn$ is the signed distance from $\bx\in\Omega$ to the tangent plane at $\bx_0\in\partial\Omega$. The upper and lower signs reflect the exterior and interior singularities respectively. The local geometric quantities are defined as
\begin{equation}\label{eqn:curveH}
\Hc(\bx_0) =  \frac{\kappa_1(\bx_0) + \kappa_2(\bx_0)}{2}, \qquad \Dc (\bx_0) =  \frac{\kappa_1(\bx_0) - \kappa_2(\bx_0)}{2},
\end{equation}
\esub
where $\kappa_j(\bx_0)$ for $j=1,2$ denote the principal curvatures of the surface at $\bx_0$. The quantity $\Hc(\bx_0)$ is the mean curvature at $\bx_0\in\partial\Omega$ while $\Dc(\bx_0)$ is a measure of the anisotropy in the surface curvature. The function $e(\bx;\bx_0)$ is a weaker singularity that remains finite in $\Omega$ near $\bx_0$ but has a different limiting value depending on the path that $\bx$ approaches $\bx_0$, analogous to the behavior of $(x^2-y^2)/(x^2+y^2)$ as $(x,y)\to(0,0)$. In \eqref{eqn:Gsing}, we derive the explicit limiting behavior
\[
e(\bx;\bx_0) \sim \frac {1} {8\pi} \left[ \frac{
|\bx-\bx_0| - \al{q}}{|\bx-\bx_0| + \al{q}} \right ] \cos (2 \varphi), \quad \mbox{as} \quad\bx\to\bx_0.
\]
Here $\varphi$ is an azimuthal measure in the tangent plane to $\partial\Omega$ at $\bx_0\in\partial\Omega$ (see Fig.~\ref{fig:LocalSchematic}).

The main challenge addressed in this work is to reliably compute the solution of the regular part $R^{i,e}(\bx;\bx_0)$ while precisely implementing the surface singularity behavior \eqref{GSing_surf} and the normalization conditions \eqref{eq:eqnG_exterior_c} and \eqref{eq:eqnG_interior_c}. The Green's functions solving \eqref{eq:eqnG_exterior} and \eqref{eq:eqnG_interior} have so far only been accessible for simple geometries such as the sphere \cite{SurfaceGreen3D,cheviakov2010asymptotic,ChevWard2010}, prolate and oblate spheroids \cite{Prolate1} and ellipsoids \cite{Dassios2012}. These solutions often take the form of slowly convergent expansions in associated Legendre functions or Lam\'{e} functions whose reliable numerical evaluation has many associated issues \cite{prolate2}. However, for many important applications, it is essential to determine this function for general curved surfaces \cite{Silbergleit2003,SingerHolcman2006,GomezCheviakov2015,Camley2025}. In chemotaxis and chemical signaling, an open problem of broad interest is to understand the role that curvature plays in biological signaling processes involving diffusion \cite{Endres2025,Curvature2023,PlunkettLawley2024,LLM2020}. Matched asymptotic analysis is a powerful tool for approximating the solution of cell signaling problems in non-radially symmetric geometries, however, it expresses these solutions in terms of global quantities related to the Neumann function \cite{BL2025,lindsay2017first,ChevWard2010,GrebenkovWard2026,TXKTW2017}. 

A recent review article \cite{TargetSearch2024} raised the lack of reliable computation of the Neumann function as an obstacle for the application of matched asymptotic methods for capture problems. The contribution of this work is to fill this gap by providing a high-order boundary integral method that can evaluate solutions of (\ref{eq:eqnG_exterior}-\ref{eq:eqnG_interior}) for general three-dimensional geometries. Our method exactly captures the singularity behavior \eqref{GSing_surf} for both the bulk $(\bx_0\notin\partial\Omega)$ and surface ($\bx_0\in\partial\Omega$) cases while resolving for the regular part $R^{i,e}(\bx;\bx_0)$. For the interior Green's function, it exactly implements the global condition \eqref{eq:eqnG_interior_c}.

The paper is structured as follows. In Section \ref{sec:local_derivation}, we perform an asymptotic analysis in the vicinity of a source $\bx_0\in\partial\Omega$ located on a curved surface, leading to the singularity behavior \eqref{GSing_surf}. In Section \ref{Sec:Numerics} we describe our boundary integral method for the resolution of the regular part $R^{i,e}(\bx;\bx_0)$ of the Green's function while exactly implementing the singularity behavior and integral constraints. In Section \ref{sec:Results} we validate the accuracy of our methods on a range of known solutions for some simple geometries. We also demonstrate the applicability of the method to \jgh{address} the question of optimal arrangements of target sites in narrow capture problems. Finally in Section \ref{sec:Discussion} we conclude with some avenues for future research and potential applications of this method.

\section{Asymptotic analysis of the singular solution structure}\label{sec:local_derivation}

In this section, we derive the singular structure of the surface Green's function in the vicinity of the source point $\bx_0\in\partial\Omega$. In our derivation, we use a coordinate system near the source that will later facilitate numerical solution. For clarity we will proceed by deriving the relevant behavior for the exterior case.

\subsection{The tangent plane coordinate system}

In this section we use an orthogonal coordinate system that is a rotated and translated version of $\RR^3$. By orienting the coordinates along the principal axes at a point, we can reduce the number of terms in the local expansion of the surface. Consider a coordinate system centered at the source $\bx_0$:
\begin{equation}\label{eq:cordinates}
    \vx - \vx_0 = \al{q}\, \hn +  p_1 \htan_1 + p_2 \htan_2,
\end{equation}
where $\hn$ is the outward pointing unit normal to $\pOm$ at $\vx_0$, and $\htan_1, \htan_2$ are perpendicular tangent vectors aligned to the principal curvatures,  $\kappa_1$ and $\kappa_2$, which are taken as negative when $\Om$ is convex (which is consistent with an outward pointing normal). The mean curvature, $\Hc$, and the curvature difference, $\Dc$, at $\vx_0$ are defined by \eqref{eqn:curveH}. The surface can be locally defined by a function
$$\al{q}=f(p_1,p_2) = \frac{\kappa_1 p_1^2 +\kappa_2 p_2^2}{2} 
    +\al{\bigoh(p_1^3,p_1^2 p_2,p_1 p_2^2,p_2^3)},$$
or implicitly as 
$$\cS(p_1,p_2,\al{q}) := \al{q} - f(p_1,p_2) = 0. $$
It is convenient to also define a local cylindrical version of these coordinates by 
\bsub\label{eq:local_cylindrical}
\begin{equation}\label{eq:local_cylindrical_a}
    (p_1,p_2, \al{q} ) = (r \cos\varphi, r \sin \varphi, \al{q}),
\end{equation}  
where we have that
\begin{equation}\label{eq:local_cylindrical_b}
r =\sqrt{p_1^2+p_2^2}, \qquad \al{q} = \hat{\bf{n}}\cdot(\bx-\bx_0), \qquad \rho  \equiv |\bx-\bx_0| = \sqrt{\al{q}^2 +r^2} \ .
\end{equation}
\esub
A schematic of the local coordinate system on the tangent plane is shown in Fig.~\ref{fig:LocalSchematic}.
\begin{figure}[t]
    \centering
    \includegraphics[width=0.65\linewidth]{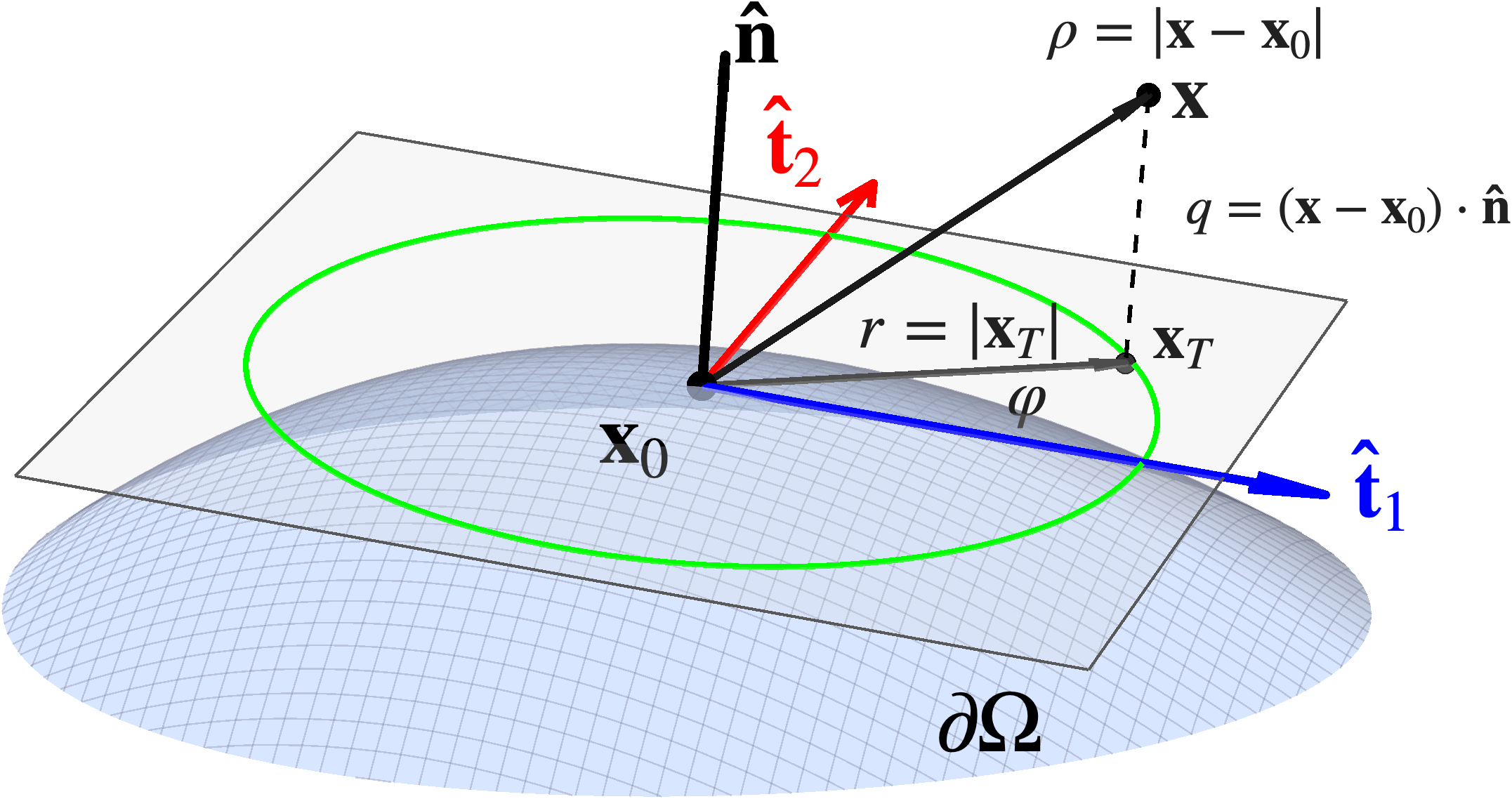}
    \caption{Local coordinate system near the point $\bx_0\in\partial\Omega$ and exterior to $\partial\Omega$. The unit vectors $\hat{\textbf{t}}_1$ ,$\hat{\textbf{t}}_2$ \eqref{eq:cordinates} lie along the principal directions. The local cylindrical coordinate system $(r,\varphi,\al{q})$ is defined in \eqref{eq:local_cylindrical}.}
    \label{fig:LocalSchematic}
\end{figure}

\subsection{The leading order Green's function}

Our aim is to write the Green's function in the form
\begin{equation}
    G(\vx ; \vx_0) = G_{\text{sing}}(\vx;\vx_0) +R(\vx;\vx_0),
\end{equation}
\ajb{where $G_{\text{sing}}(\vx ; \vx_0)$ is the singular portion of the Green's function and $R(\vx;\vx_0)$ is the regular part which is analytic at $\vx=\vx_0$. When the source is placed in the bulk, the singularity corresponds to the free space function $G_{\text{sing}}(\bx;\bx_0)= (4\pi|\bx-\bx_0|)^{-1}$; this choice completely specifies the regular part which 
satisfies an inhomogeneous Neumann problem on the domain boundary $\partial \Omega$, supplemented by appropriate decay conditions for the exterior problem. Significantly more work is required when the source is on the domain boundary; below we will expand $G_{\text{sing}}$ in an asymptotic expansion which captures the singular behavior near the source. }




In order to calculate an expansion for non-flat geometries, we first note that the unit normal on $\cS$ can be written as
$$\hn(\al{p_1,p_2}) = \frac{\nabla \cS}{|\nabla \cS|} = 
\frac{( -f_{p_1}, -f_{p_2}, 1 )}
{\sqrt{1+ (f_{p_1})^2 +(f_{p_2})^2}} \ ,
$$
which gives the following formula for $\partial_n G$ on $\pOm,$
\begin{equation}
\nonumber \left . \partial_n G \right|_{\pOm} 
=\left . \hn(p_1,p_2) \cdot \nabla G   \right |_{\pOm}  \\
=\left .
\frac{\partial_{\al{q}} G-f_{p_1}\partial_{p_1} G -f_{p_2}\partial_{p_2} G
} {\sqrt{1+ (f_{p_1})^2 +(f_{p_2})^2}}\right |_{\cS=0} \ .
\end{equation}
The transformed Dirac source takes the form
$$\delta(\vx - \vx_0)=  \sqrt{1+ (f_{p_1})^2 +(f_{p_2})^2 }\,\delta(p_1)\delta(p_2),$$
where the multiplicative factor is the area element, $\mathrm{d}A =\sqrt{1+ (f_{p_1})^2 +(f_{p_2})^2 } \, d p_1 \, d p_2 $.
However, if we specify $f_{p_1}(0,0)=f_{p_2}(0,0)=0$ then the area element at the source evaluates to unity. Consequently, we can reduce the boundary condition to 
\begin{equation}
\partial_{\al{q}} G-f_{p_1}\partial_{p_1} G -f_{p_2} \partial_{p_2} G = - \delta(p_1)\delta(p_2) 
\qquad (p_1,p_2) \in \RR^2
,
\end{equation}
which is valid in the region where the surface can be represented as a graph $\al{q}=f(p_1,p_2)$.

\subsection{The local perturbation expansion}
\ajb{Guided by the leading order singularity of the free space Green's function \eqref{eq:FreeSpace}, we introduce the variables
\begin{equation}\label{eqn:rescaling_coords}
(p_1,p_2,\al{q}) = (\eps \mu, \eps \nu, \eps \zeta),
\end{equation}
with associated cylindrical coordinates  $(r,\varphi,\zeta)$ 
$$\mu = r \cos \varphi, \quad \nu= r \sin \varphi,
\quad r=\sqrt{\mu^2+\nu^2} \ , \quad \rho=\sqrt{r^2+\zeta^2}  \ .$$
Here $\eps\ll 1$ is a small parameter that sets the scale of the local region around the source. In practice, we choose a fixed small distance, $d^*$, around the source and derive an asymptotic expansion for the solution in the region where $|\bx-\bx_0| < d^*$ as $\eps \to 0$. The Neumann boundary conditions are applied on the tangent plane (derived via a local expansion) and Laplace's equation is solved in the upper half-plane in a hemisphere where $\rho  \le d^*/\eps \to \infty$. }

\ajb{We now expand the surface boundary condition;} on the surface,
$$\zeta= \frac 1 {\eps}f(\eps \mu,\eps \nu) = \eps \frac{\kappa_1 \mu^2 +\kappa_2 \nu^2}{2} 
    +\bigoh(\eps^2).$$
Correspondingly, the boundary condition becomes
\bsub
\begin{equation}
\frac{1}{\eps} \left [ \partial_\zeta G-f_{p_1}\partial_{\mu} G -f_{p_2} \partial_{\nu} G \right ]= - 
\frac{1}{\eps^2}\delta(\mu)\delta(\nu) 
\qquad (p_1,p_2) \in \RR^2 , \qquad \zeta = f
,
\label{theBC}
\end{equation}
where
\begin{equation}\label{eqn:the}
    f_{p_1}(\eps \mu,\eps \nu) = \eps \kappa_1 \mu  +\bigoh(\eps^2) \ , \quad 
f_{p_2}(\eps \mu,\eps \nu) = \eps \kappa_2 \nu +\bigoh(\eps^2) \ .
\end{equation}
\esub
\al{Expansion of the normal derivative yields
\begin{equation}\label{eqn:boundary_expansion}
\partial_{\zeta}G|_{\zeta = f} = \partial_{\zeta}G|_{\zeta = 0} + \eps \Big(\frac{\kappa_1 \mu^2 + \kappa_2 \nu^2}{2}\Big)\partial^2_{\zeta\zeta}G|_{\zeta = 0} + \bigoh(\eps^2).
\end{equation}}
The Green's function is now expanded as
$$G(p_1,p_2,\al{q}) = \frac{1}{\eps} g_0(\mu,\nu, \zeta) 
+ g_1(\mu,\nu, \zeta) + \eps g_2(\mu,\nu, \zeta) + \ajb{R^*(\mu,\nu,\zeta)} +\bigoh(\eps^2) ,$$
\ajb{where we have written $R(\bx;\bx_0) =R^*(\mu,\nu,\zeta)$ which is a harmonic function satisfying a homogeneous Neumann condition on the surface. 
Our strategy at each order of the expansion is to specify boundary conditions for $g_n$ that incorporates all the singular portions of the expansion of the Neumann condition \eqref{eqn:boundary_expansion}.}
\ajb{
There is some arbitrariness in this expansion, largely due to the lack of far-field (matching) conditions for $G$. The regular part satisfies 
\bsub
\label{eq:halfplaneharm}
\begin{align}
    \Delta R^* &= 0, \qquad \zeta >0; \\[5pt]
    \partial_\zeta R^* &= 0, \qquad \zeta=0, \quad (\mu,\nu) \in \RR^2 \ . 
\end{align}
\esub
The general solution to \eqref{eq:halfplaneharm} admits a solution in terms of the standard spherical harmonics.
Ultimately, for this local expansion we are interested in a numerical strategy for determining $R(\vx_0;\vx_0) =R^*(0,0,0) =R_0$, 
that is the regular part of the Green's function at the source. 
}
We now expand the boundary condition \eqref{theBC} and collect terms at each order. \\

\noindent\uline{ $\mathcal{O}(\eps^{-2})$:}
At leading order, we obtain the upper half-plane problem 
\bsub
\begin{align}
    \Delta g_0 &= 0, & \zeta &>0, \quad (\mu,\nu) \in \RR^2; \\[5pt]
    \partial_\zeta g_0 &= -\delta(\mu)\delta(\nu),& \zeta&=0, \quad (\mu,\nu) \in \RR^2,   
\end{align}
\esub
where $\Delta \equiv \partial^2_\mu + \partial^2_\nu+  \partial^2_{\zeta} $.
We choose the half-plane Green's function solution, 
\begin{equation}
    g_0 = \frac{1}{2 \pi \rho} \ .
\end{equation}
\ajb{The solution at this order (and for that matter all successive orders) is arbitrary up to an additive harmonic function that would satisfy \eqref{eq:halfplaneharm}, which we take as zero. This is a gauge choice -- it only effects what is included in $R^*$ vs. $g_0$ . \\ }

\noindent\uline{ $\mathcal{O}(\eps^{-1})$:}
At this order, we expand the boundary condition noting that 
$$
 \partial_{\mu} g_0  = - \frac{\mu}{2 \pi\rho^3}  , \qquad
  \partial_{\nu} g_0  = - \frac{\nu}{2 \pi\rho^3} , \qquad 
\partial_{\zeta} g_0  = - \frac{\zeta} {2 \pi\rho^3} .
$$
Also, on the boundary $\rho(\mu,\nu,\zeta)=\rho(\mu,\nu,0)[1+\bigoh(\eps)] = r[1+\bigoh(\eps)]$. 
This allows us to expand the Neumann boundary condition to obtain the correction problem
\bsub\label{eqn:correction_g1}
\begin{align}
    \Delta g_1 &= 0, \qquad \zeta >0, \quad (\mu,\nu) \in \RR^2; \\[4pt]
    \partial_\zeta g_1 &= 
    - \frac{\kappa_1 \mu^2 + \kappa_2 \nu^2}{4 \pi r^3},
    \qquad \zeta=0, \quad  (\mu,\nu) \in \RR^2 
\end{align}
\esub
The Neumann boundary condition in cylindrical coordinates is
$$ \partial_\zeta g_1 
= 
     - \frac{1}{4\pi r} \left [ \cH +\Dc \cos(2\varphi) \right ]
    \qquad \zeta=0,
$$
\al{where $\Hc = \frac12(\kappa_1+\kappa_2)$ and $ \Dc= \frac12(\kappa_1-\kappa_2)$.} To solve \eqref{eqn:correction_g1}, we seek a solution of form 
\begin{equation}
g_1 = \Xi(r,\zeta) +\chi(r,\zeta) \cos (2 \varphi).
\label{eq:g1}
\end{equation}
Here the coefficient functions satisfy the problems
\bsub
\begin{equation}
\left\{ 
\begin{array}{rl}\Xi_{\zeta \zeta} + \Xi_{r r} + \ds\frac 1 {r} \Xi_r =0,& \quad  r, \zeta>0; \\[5pt]
\Xi_\zeta  = - \ds\frac {\cH} {4\pi r}, & \quad \zeta=0 \ .\\[10pt]
\end{array}
\right.
\quad
\left\{
\begin{array}{rl}
\chi_{\zeta \zeta} + \chi_{r r} + \ds\frac 1 {r}  \chi_{r}  - \ds\frac {4}{{r} ^2} \chi =0,&  \quad r, \zeta>0;\\[5pt]
\chi_\zeta  =  -\ds\frac {\Dc} {4\pi {r} }, & \quad \zeta=0 \ .\\[10pt]
\end{array}
\right.
\end{equation}
\ajb{There is also a regularity condition;  $g_1$ must be analytic as $r \to 0$ at any fixed $\zeta >0$.}
A solution to these problems can be found via an application of the Hankel transform,
\begin{equation}
    \Xi({r} ,\zeta) =
-  \ds\frac {\cH} {4\pi} \log \left [{ \rho + \zeta}\right ], \qquad \chi(r ,\zeta) =  \frac {\Dc} {8 \pi} 
\left [  \frac{\rho - \zeta}{\rho + \zeta} \right ] \ ,
\end{equation}
\esub
This yields the following expression for the monopole correction
\begin{equation}\label{eqn:g1}
   g_1(\mu,\nu,\zeta) = 
- \frac {\cH} {4\pi} \log \left [{ \rho + \zeta} \right ] 
+ \frac {\Dc} {8\pi} 
    \left [  \frac{\rho - \zeta}{\rho + \zeta} \right ] \cos (2 \varphi) \ ,
\end{equation}
\ajb{which again is arbitrary up to an additive harmonic function which satisfies \eqref{eq:halfplaneharm}.}
The second term in the expression \eqref{eqn:g1} can be rewritten as 
$$\left [  \frac{\rho - \zeta}{\rho + \zeta} \right ] \cos (2 \varphi)
= \left [  \frac{\rho - \zeta}{\rho + \zeta} \right ] \frac{\mu^2 -\nu^2}{r^2} = \frac{\mu^2 -\nu^2}{(\rho+\zeta)^2} \ ,$$
by using the relationship $\rho^2 = r^2+\zeta^2$. For fixed positive $\zeta$, we have the behavior 
$$g_1(\mu,\nu,\zeta) \sim  - \frac {\cH} {4\pi} \log  (2 \zeta)  +
\left [ - \frac {\cH} {16\pi \zeta^2}+ \frac {\Dc} {32 \pi \zeta^2} \cos (2 \varphi) \right ] r^2
+\bigoh(r^4),$$
as $r\to0$ which demonstrates that the solution is differentiable (even analytic) in this limit. However, we observe that the solution is singular at the origin; along rays from the origin the mean curvature term blows up like $\log \rho$ and depends on the polar angle while the curvature difference term depends on both the polar and azimuthal angle. \\

\noindent\uline{$\mathcal{O}(\eps^{0})$:} To better understand the error structure of the expansion, we proceed to the next order which requires a higher order expansion of the boundary. This is given by
\begin{equation}\label{eqn:bdy_higher}
    f( p_1, p_2) =  \frac{\kappa_1 p_1^2 +\kappa_2 p_2^2}{2} 
    +
     \left ( \cda p_1^3 
    +\ccb p_1^2p_2 
    +\cbc p_1 p_2^2 
    +\cad p_2^3
   \right) + 
    \bigoh(r^4) \ .
\end{equation}
\al{In the rescaled coordinates \eqref{eqn:rescaling_coords}, this higher order boundary representation \eqref{eqn:bdy_higher} takes the form
\begin{gather*}
    \zeta = \eps f_1 + \eps^2 f_2 + \bigoh(\eps^3),\\[4pt] f_1 = \frac{\kappa_1\mu^2+ \kappa_2 \nu^2}{2}, \qquad f_2 = \cda \mu^3 
    +\ccb \mu^2\nu 
    +\cbc \mu \nu^2 
    +\cad \nu^3
\end{gather*}
The high order expansion of the normal derivative \eqref{eqn:boundary_expansion} is given by 
\begin{equation}
\partial_{\zeta}G|_{\zeta = f} = \partial_{\zeta}G|_{\zeta = 0} 
+ \eps f_1\partial^2_{\zeta\zeta}G|_{\zeta = 0} 
+ \eps^2\Big[ f_2 \partial^2_{\zeta\zeta}G|_{\zeta = 0} + 
\frac{f_1^2}{2} \partial^3_{\zeta\zeta\zeta}G|_{\zeta = 0} \Big ]
+ \bigoh(\eps^3),
\end{equation}
Additionally, we have that
\begin{align}
    f_{p_1}(\eps\mu,\eps\nu) &= \eps\kappa_1\mu + \eps^2(3\cda \mu^2 
    +2\ccb \mu\nu 
    +\cbc \nu^2) + \bigoh(\eps^3)\\[4pt]
    f_{p_2}(\eps\mu,\eps\nu) &= \eps\kappa_2\nu + \eps^2 (\ccb \mu^2 + 2\cbc\mu\nu + 3 \cad \nu^2) + \bigoh(\eps^3)
\end{align}
}
\al{After expanding and simplifying algebra,} the high order correction  $g_2(\mu,\nu,\zeta)$ satisfies
\bsub
\begin{align}
    \Delta g_2 &=  0 \qquad \zeta >0 \\
    \partial_\zeta g_2& = - \left [ 
    \frac{\cH^2}{4 \pi} 
     + \Dc \cH \frac{\mu^4-\nu^4}{2 \pi r^4}
     + \Dc^2 \frac{\mu^4-6\mu^2\nu^2+\nu^4}{4 \pi r^4} \right ]  \\
 \nonumber & \phantom{=}   - \left [ \frac{\cda \mu^3 
    +\ccb \mu^2 \nu 
    +\cbc \mu \nu^2 
    +\cad \nu^3}{\pi r^3} \right ] 
   \qquad  \qquad \zeta=0, \quad  (\mu,\nu) \in \RR^2 \ ,
\end{align}
\esub
The boundary condition can be rewritten in cylindrical coordinates as
\bsub\label{eqn:g_2}
\begin{align}
\label{eqn:g_2a}    \Delta g_2 &= 0, \qquad \zeta >0 ;\\
    \label{eqn:g_2b}  \partial_\zeta g_2 &= - \left [
    \frac{\cH^2}{4 \pi} 
    +   \frac{ \Dc \cH }{2 \pi }  \cos(2\varphi) 
    +  \frac{ \Dc^2 }{4 \pi }  \cos(4\varphi) \right ]  \qquad \qquad \qquad    \zeta=0, \quad  (\mu,\nu) \in \RR^2.
    \\
    \nonumber & \phantom{=}  - \left [ \frac{\cbc+3\cda}{4 \pi} \cos(\varphi)
     + \frac{\ccb +3\cad}{4 \pi} \sin(\varphi)
     + \frac{\cda-\cbc}{4 \pi} \cos(3\varphi)
     + \frac{ \ccb- \cad} {4 \pi} \sin(3\varphi) \right ] 
     \ ,
\end{align}
\esub

By linearity, $g_2$ will be a sum of terms of form 
$$g_2^c(\zeta,r, \varphi)= \Phi_n(\zeta,r)\cos(n \varphi),  \qquad
g_2^s(\zeta,r, \phi)= \Phi_n(\zeta,r)\sin(n \varphi),
$$
where $\Phi_n(\zeta,r)$ satisfies 
\bsub\label{eq:g2_sol}
\begin{align}
\partial^2_{\zeta \zeta}\Phi_{n} + \partial^2_{r r}\Phi_{n} + \frac 1 r \partial_{r }\Phi_{n} - \frac {n^2}{r^2} \Phi_n &=0  \qquad \zeta>0; \\[4pt]
 \partial_{\zeta}\,\Phi_{n}  &= 1  \qquad \zeta=0 .
\end{align}
\esub
\ajb{There is again a regularity condition;  $g_2$ must be analytic as $r \to 0$ at any fixed $\zeta >0$.} For $n=0$,
$$\Phi_0(\zeta,r)=\zeta \, ,$$
and using Green's function and similarity methods, we find that
\bsub
\begin{align}
\Phi_1(\zeta,r)&=\frac r 2 \log (\rho+\zeta) + \frac{\zeta(\rho-\zeta)}{2r},& 
\Phi_2(\zeta,r)&= \zeta-\frac {2}{3} \frac{(\rho^3-\zeta^3)}{r^2 },
 \\
\Phi_3(\zeta,r)&= \frac{\zeta(\rho^3-\zeta^3)}{r^3 }
-\frac{3\zeta^2}{2r} - \frac{3r}{8}
,& 
\Phi_4(\zeta,r)&= -\frac {4 \rho^3(r^2+6\zeta^2)}{15r^4} +\zeta+\frac{8\zeta^3}{3r^2 } +\frac{8\zeta^5}{5r^4} ,
\end{align}
\esub
where $\rho=\sqrt{r^2+\zeta^2}$. 
\ajb{These terms are not infinitely differentiable at the origin, however,  
$\lim_{\rho\to 0} \Phi_n =0 $ for $n=0,1,2,3,4$. From this we conclude that $g_2$ does not contribute to the regular part at $\bx_0$.}

\subsection{Reconstituting the Green's function and the regular part}
At this point, we bring together the final expression for the singular part of the Green's function, $G_{\text{sing}}$. In terms of the coordinates $(p_1,p_2,\al{q}) = (\eps \mu, \eps \nu, \eps \zeta)$, we have that
\begin{align*} 
\frac{1}{\eps}g_0(\mu,\nu,\zeta)&=  g_0(p_1,p_2,\al{q}) \ ,  \\[4pt]
g_1(\mu,\nu,\zeta) & =  g_1(p_1,p_2,\al{q})  + \frac{\cH}{4 \pi}\log \eps
\ ,  \\[4pt]
\eps g_2(\mu,\nu,\zeta) &=  g_2(p_1,p_2,\al{q})
+ \left [C_1p_1 +C_2 p_2  \right ] \log \eps
,
\end{align*}
for appropriate constants $C_1$ and $C_2$. This allows us to identify 
\begin{align*}
G_{\text{sing}}(p_1,p_2,\al{q}) &\sim  g_0(p_1,p_2,\al{q}) + g_1(p_1,p_2,\al{q}) +g_2(p_1,p_2,\al{q}) \\
&\sim  g_0(p_1,p_2,\al{q}) + g_1(p_1,p_2,\al{q}) + 
\ajb{\bigoh (|\vx-\vx_0| \log |\vx-\vx_0|)} \ .
\end{align*}
Recalling the definition $\al{q} = (\bx-\bx_0)\cdot \hn$, we arrive at the form of the singularity behavior
\begin{align}
\nonumber G_{\text{sing}}(p_1,p_2,\al{q}) &\sim g_0(p_1,p_2,\al{q}) + g_1(p_1,p_2,\al{q}) \\
&\sim \frac{1}{2 \pi |\bx-\bx_0|} - \frac {\cH} {4\pi} \log \left [{ |\bx-\bx_0| + \al{q}} \right ] 
+ \frac {\Dc} {8\pi} \left[ \frac{
|\bx-\bx_0| - \al{q}}{|\bx-\bx_0| + \al{q}} \right ] \cos (2 \varphi), 
\label{eqn:Gsing}
\end{align}
\al{as $\bx\to\bx_0$.} 
\al{The interior problem with a boundary source has an identical singularity structure to the exterior problem, but with sign changes in the coefficients of the curvature dependent terms. The $1/|\Omega|$ term on the right hand side of \eqref{eq:eqnG_interior} generates a smooth particular solution and is included in the regular part.}

\al{In Section~\ref{Sec:Numerics}, we will introduce a regular-singular decomposition of the Green's function which matches the precise local behavior \eqref{eqn:Gsing} as $\bx\to\bx_0$, but contains certain modifications to ensure stable numerical evaluation on general domains.}

\section{Boundary integral solution of the regular part of the Green's function}\label{Sec:Numerics}

In order to obtain accurate solutions of \eqref{eq:eqnG_exterior} and \eqref{eq:eqnG_interior} which effectively capture the singularities in the Neumann Green's function as $\bx \to \bx_0,$ we first perform a  standard reduction of the problem to a boundary integral equation (BIE) \cite{kress} for an unknown density $\sigma$ on the boundary. A point of departure from existing BIE methods is the singular nature of the boundary data, which requires a custom discretization of the portion of the geometry in the vicinity of the source, and careful evaluation of the right hand side in order to mitigate effects of catastrophic cancellations and retain high-order convergence. 

\jgh{Intuitively, it is tempting to subtract off the three singular terms in the asymptotic expansion $G^{i,e}_{\rm sing}$ from $G^{i,e}$ as in \eqref{eqn:Gsing}, reducing the problem to a boundary value problem with weaker singularities near the source $\bx_0.$ Unfortunately, for interior problems, or certain exterior problems on non-convex domains, this generates additional singularities away from $\bx_0$ due to the singularities in $G^{i,e}_{\rm sing}$ lying along the ray emanating from $\bx_0$ in the negative normal direction. In order to circumvent this issue, we modify $G^{i,e}_{\rm sing},$ constructing new functions which have the same singularities near $\bx_0,$ but with singularities only along finite length line segments rather than rays. In practice, this amounts to adding additional `image sources' outside the domain.}

In \jgh{the remainder of} this section, we outline our \jgh{numerical} approach. We begin by introducing several auxiliary functions and their analytic properties which are useful in the construction of suitable boundary data for our integral equations. Following this, we review the reduction of \eqref{eq:eqnG_exterior} and \eqref{eq:eqnG_interior} to BIEs and the classical approach for enforcing integral constraints in the interior problem. Finally, we sketch a common approach to discretize BIEs and detail the modifications required to treat our boundary data.

\subsection{Analytic properties of the corrections}\label{eq:analytical_corrections}

Consider the functions $\phi,\psi$ defined by
\begin{equation}
    \phi(\mu,\nu,\al{q}) = \log( {\rho}+\al{q}), \qquad
\psi(\mu,\nu,\al{q}) = \frac{(\mu^2-\nu^2)}{(\rho+\al{q})^2},
\end{equation}
with $\rho = \sqrt{\mu^2+\nu^2+\al{q}^2}.$ A direct calculation shows that $\phi$ and $\psi$ are harmonic away from the half-line $\{(0,0,\al{q})\,|\, \al{q} \le 0 \}.$ In this section we describe modifications of the functions $\phi$ and $\psi$ which have the same behavior near the origin, but have singularities lying only on a finite line segment. The key observation is given in the following lemma.

\begin{lemma}
    Given $\phi$ and $\psi$ as defined above, for any $(\mu,\nu,\al{q})$ not on the ray $\{(0,0,\al{q})\,|\, \al{q}\le 0\},$ we have
    $$\frac{\partial}{\partial \al{q}} \phi(\mu,\nu,\al{q}) = \frac{1}{\rho}, \qquad\frac{\partial^3}{\partial \al{q}^3}\psi(\mu,\nu,\al{q}) = -2\left( \frac{\partial^2}{\partial \mu^2} \frac{1}{\rho}- \frac{\partial^2}{\partial \nu^2} \frac{1}{\rho}\right).$$
\end{lemma}
The proof is straightforward, and omitted. Applying the fundamental theorem of calculus immediately gives the following corollary.

\begin{corollary}
 For any $\eta_2<0$ and any $(\mu,\nu,\al{q})$ not on the ray $\{(0,0,\al{q})\,|\, \al{q}\le 0\},$
 \begin{align}\label{eqn:phi_int}
 \phi(\mu,\nu,\al{q}) - \phi(\mu,\nu,\al{q}-\eta_2) = -\int_{\eta_2}^0 \frac{1}{\sqrt{\mu^2+\nu^2+(\al{q}-t)^2}}\,{\rm d}t,
 \end{align}
 and
  \begin{align}\label{eqn:psi_int}
  &\psi(\mu,\nu,\al{q}) - \psi(\mu,\nu,\al{q}-\eta_2) \jgh{-}\eta_2 \psi'(\mu,\nu,\al{q}-\eta_2) -\frac{\eta_2^2}{2} \psi''(\mu,\nu,\al{q}-\eta_2)\\
  &\quad \quad \quad= 2\int_{\eta_2}^0 \int_{\eta_2}^t \int_{\eta_2}^s\left( \frac{\partial^2}{\partial \mu^2} \frac{1}{\sqrt{\mu^2+\nu^2+(\al{q}-u)^2}}- \frac{\partial^2}{\partial \nu^2} \frac{1}{\sqrt{\mu^2+\nu^2+(\al{q}-u)^2}}\right)\,{\rm d}u\,{\rm d}s\,{\rm d}t, \nonumber
  \end{align}
  with $\psi'(\mu,\nu,\al{q}) = \frac{\partial}{\partial \al{q}} \psi(\mu,\nu,\al{q})$ and $\psi''(\mu,\nu,\al{q}) = \frac{\partial^2}{\partial \al{q}^2} \psi(\mu,\nu,\al{q}).$ 
\end{corollary}
By interchanging the order of derivatives and integration, it is easily seen that the right-hand sides of (\ref{eqn:phi_int}) and (\ref{eqn:psi_int}) are harmonic away from the line segment $\{(0,0,\al{q})\,|\, \eta_2 \le \al{q}\le 0\}$ and are $O(1/\rho)$ and $O(1/\rho^3),$ respectively, as $\rho \to \infty$ uniformly in angle. It follows that the left-hand sides of the corresponding equations have removable singularities on the ray $\{(0,0,\al{q})\,|\, \al{q} < \eta_2\}.$

The following corollary, which amounts to a change of the order of integration in (\ref{eqn:psi_int}), gives a more convenient expression for the integral in (\ref{eqn:psi_int}).

\begin{corollary}\label{cor:psi}
Let $\psi$ and $\eta_2$ be as above. Then
\begin{align}
&\psi(\mu,\nu,\al{q}) - \psi(\mu,\nu,\al{q}-\eta_2) \jgh{-}\eta_2 \psi'(\mu,\nu,\al{q}-\eta_2) -\frac{\eta_2^2}{2} \psi''(\mu,\nu,\al{q}-\eta_2)\\
  &\quad \quad \quad= \int_{\eta_2}^0 t^2\left( \frac{\partial^2}{\partial \mu^2} \frac{1}{\sqrt{\mu^2+\nu^2+(\al{q}-t)^2}}- \frac{\partial^2}{\partial \nu^2} \frac{1}{\sqrt{\mu^2+\nu^2+(\al{q}-t)^2}}\right)\,{\rm d}t. \nonumber
\end{align}
\end{corollary}
Finally, the following lemma establishes properties of a related function that will be convenient in numerical simulations.
\begin{lemma}
    For any $\eta_2< \eta_1 <\eta_0 <0,$ \al{let}
    \[c_0 = -\frac{\eta_1 \eta_2}{(\eta_1-\eta_0)(\eta_0-\eta_2)},\qquad c_1 = -\frac{\eta_0 \eta_2}{(\eta_0-\eta_1)(\eta_1-\eta_2)},\qquad c_2 = -\frac{\eta_0 \eta_1}{(\eta_0-\eta_2)(\eta_2-\eta_1)},
    \]
    \al{and define}
  \bsub
  \begin{align}
      \psi_{\eta_0,\eta_1,\eta_2}(\mu,\nu,\al{q})&=\psi(\mu,\nu,\al{q})-c_0\psi(\mu,\nu,\al{q}-\eta_0)-c_1\psi(\mu,\nu,\al{q}-\eta_1)-c_2\psi(\mu,\nu,\al{q}-\eta_2),\label{eqn:psidef}\\[5pt]
       \al{\phi_{\eta_0,\eta_1,\eta_2}(\mu,\nu,\al{q})} &\jgh{=\phi(\mu,\nu,\al{q})-c_0\phi(\mu,\nu,\al{q}-\eta_0)-c_1\phi(\mu,\nu,\al{q}-\eta_1)-c_2\phi(\mu,\nu,\al{q}-\eta_2)\label{eqn:phidef}.}
  \end{align}
  \esub
    Then $\psi_{\eta_0,\eta_1,\eta_2}$ is harmonic, except on $\{(0,0,\al{q})\,|\, \eta_2 \le \al{q} \le 0\},$ and is $O(1/\rho^3)$ uniformly in angle as $\rho \to \infty.$
\end{lemma}
\begin{proof}
Using Corollary \ref{cor:psi} we observe that for $\psi_j(\mu,\nu,\al{q}) := \psi(\mu,\nu,\al{q}-\eta_j),$ then 
$$\psi_j(\mu,\nu,\al{q}) = \psi(\mu,\nu,\al{q}-\eta_2) \jgh{+} (\eta_2-\eta_j) \psi'(\mu,\nu,\al{q}-\eta_2) +\frac{(\eta_2-\eta_j)^2}{2}\psi''(\mu,\nu,\al{q}-\eta_2) + I_j,$$
where 
$$I_j =\int_{\eta_2}^{\eta_j} (t-\eta_j)^2\left( \frac{\partial^2}{\partial \mu^2} \frac{1}{\sqrt{\mu^2+\nu^2+(\al{q}-t)^2}}- \frac{\partial^2}{\partial \nu^2} \frac{1}{\sqrt{\mu^2+\nu^2+(\al{q}-t)^2}}\right)\,{\rm d}t,$$
for $j=0,1,2.$ It is easily seen that the coefficients $c_0,c_1,c_2$ satisfy $\sum_{j=0}^2\eta_j^k c_j =\delta_{k,0},$ with $k=0,1,2.$ It follows that
$$\psi_{\eta_0,\eta_1,\eta_2}(\mu,\nu,\al{q}) = I-c_0 I_0 - c_1 I_1,$$
with 
$$I =\int_{\eta_2}^{0} t^2\left( \frac{\partial^2}{\partial \mu^2} \frac{1}{\sqrt{\mu^2+\nu^2+(\al{q}-t)^2}}- \frac{\partial^2}{\partial \nu^2} \frac{1}{\sqrt{\mu^2+\nu^2+(\al{q}-t)^2}}\right)\,{\rm d}t.$$
 The rest of the proof follows immediately from the properties of the integrands of $I,I_0,$ and $I_1.$ \jgh{A similar construction can be applied to $\phi$ to obtain $\phi_{\eta_0,\eta_1,\eta_2}.$}
\end{proof}

\jgh{Finally, considering $\psi_{\eta_2,\eta_1,\eta_0}$ and $\phi_{\eta_2,\eta_1,\eta_0}$  defined by (\ref{eqn:psidef}) and (\ref{eqn:phidef}), respectively, for $\eta_2<\eta_1<\eta_0<0$ it follows straightforwardly from the definitions that $\psi_{-\eta_2,-\eta_1,-\eta_0}(\mu,\nu,-\al{q})$ and $\phi_{-\eta_2,-\eta_1,-\eta_0}(\mu,\nu,-\al{q})$ are harmonic functions of $(\mu,\nu,\al{q})$ away from the line segment $\{(0,0,s)\,|\, 0 \le s  \le -\eta_2\}.$}

\subsection{Derivation of the integral equation}

\subsubsection{General notation}
In this subsection we introduce general notation and operators which will be used in the construction, discretization, and solution of our boundary integral equations.

Given a point $\bx_0 \in \partial \Omega,$ let $\hn(\bx_0)$ denote the outward pointing normal of $\partial \Omega$ at $\bx_0$. Similarly, relative to that normal, let $\kappa_{1}(\bx_0)$ and $\kappa_{2}(\bx_0)$ denote the principal curvatures of $\partial \Omega$ at $\bx_0,$ the quantities $\Hc(\bx_0)$, $\Dc(\bx_0)$ denote their sums and differences and let $\htan_1(\bx_0)$, $\htan_2({\bx_0})$ be the principal directions. In the following, it will be convenient to work in the rotated and shifted coordinate frame $\bx -\bx_0= \al{p_1} \, \htan_1({\bx_0})+ \al{p_2} \, \htan_2(\bx_0)+q \, \hn ( {\bx_0)}.$

For $\eta_2<\eta_1<\eta_0 <0,$ we define the functions
\begin{align}
Q^{\pm}_{\eta_0,\eta_1,\eta_2}(\al{p_1},\al{p_2},q;\bx_0) &= \frac{1}{2\pi | \bx -\bx_0|}\mp\frac{\Hc(\bx_0)}{4 \pi} \phi_{\pm \eta_0,\pm \eta_1,\pm \eta_2}(\al{p_1},\al{p_2},\pm q)\\
& \pm\frac{\Dc(\bx_0)}{8\pi} \psi_{\pm \eta_0,\pm \eta_1,\pm \eta_2}(\al{p_1},\al{p_2},\pm q). \nonumber
\end{align}
For the exterior function, \jgh{in order for $Q^{+}$ to be harmonic in the exterior of $\Omega,$ the line segment on which its singularities lie ($\bx_0+s\hn (\bx_0),$ $s\in [\eta_2,0]$) must lie entirely within $\Omega.$ Analogously, for interior problems, $\bx_0-s\hn (\bx_0),$ $s\in [\eta_2,0]$ must lie outside of the interior of $\Omega.$ Additionally, if $|\eta_0|,|\eta_1|,$ and $|\eta_2|$ are too small, then the `image points' will induce extra nearly-singular terms in $\partial_n Q^\pm$ near $\bx_0.$ If they are too large then, depending on the nature of $\partial \Omega,$ they can generate additional nearly singular behavior at other points on $\partial \Omega$ which would require additional refinement. In practice, we observe empirically that our numerical method is fairly insensitive to the precise choice of $\eta_0,\eta_1,$ and $\eta_2.$} We note that for many geometry/source configurations, these `image points' can be omitted entirely for the interior problem. \jgh{The following lemma establishes the regularity of the normal derivative of $Q^\pm$ on $\partial \Omega.$
\begin{lemma}
Let $Q^\pm_{\eta_0,\eta_1,\eta_2}$ be as above, with $\eta_2<\eta_1<\eta_0<0$ chosen so that the line segment $\ell_\pm =\{\bx_0 \pm s \bn(\bx_0)\,|\,\eta_2 \le s <0\}$ does not intersect $\partial \Omega.$ Then $\partial_n Q^{\pm}$ is smooth on $\partial \Omega$ away from $\bx_0,$ and bounded near $\bx_0.$
\end{lemma}
\begin{proof}
By construction, the functions $\phi_{\pm \eta_0,\pm \eta_1,\pm \eta_2}$ and $\psi_{\pm \eta_0,\pm \eta_1,\pm \eta_2}$ are smooth on $\partial \Omega\setminus \{\bx_0\}.$ Similarly, the normal derivative of $1/(2\pi|\bx-\bx_0|)$ is smooth away from $\bx_0,$ and hence $\partial_n Q^{\pm}_{\eta_0,\eta_1,\eta_2}$ is smooth on $\partial \Omega$ away from $\bx_0.$ Near $\bx_0,$ $\phi-\phi_{\pm \eta_0,\pm\eta_1,\pm\eta_2}$ and $\psi-\psi_{\pm \eta_0,\pm\eta_1,\pm\eta_2}$ are smooth and hence it suffices to consider the Neumann trace of
$$\widetilde{Q}^\pm(\al{p_1},\al{p_2},q;\bx_0) := \frac{1}{2\pi |\bx-\bx_0|} \mp \frac{\mathcal{H}(\bx_0)}{4\pi} \phi(\al{p_1},\al{p_2},\pm q) \pm \frac{\mathcal{Q}(\bx_0)}{8\pi} \psi(\al{p_1},\al{p_2},\pm q).$$
Taking the normal derivative of $\widetilde{Q}^\pm$ yields
\begin{align*}
    \partial_n \widetilde{Q}^\pm(\al{p_1},\al{p_2},q;\bx_0) &= -\frac{\bn(\bx)\cdot \br}{2\pi \rho^3}\mp \frac{\mathcal{H}(\bx_0)}{4\pi (\rho\pm q)}W^\pm \\[4pt]
    &\qquad\pm \frac{\mathcal{Q}(\bx_0)}{8\pi}\left[-2\frac{W^\pm (\al{p_1}^2-\al{p_2}^2)}{(\rho\pm q)^3} +2\frac{\al{p_1} \,\hat{\bf t}_1(\bx_0)\cdot \bn(\bx)-\al{p_2}\, \hat{\bf t}_2(\bx_0)\cdot \bn(\bx)}{(\rho\pm q)^2}\right],
\end{align*}
with $\br = \bx-\bx_0,$ $\rho = |\br|,$ and $W^\pm = \bn(\bx)\cdot \br / \rho\, \pm\, \bn(\bx_0)\cdot \bn(\bx).$ Note that for $\bx$ sufficiently close to $\bx_0,$  $\bn(\bx_0) \cdot \br = \frac{1}{2}(\kappa_1 \al{p_1}^2+\kappa_2 \al{p_2}^2)+O(\rho^3) = O(\rho^2)$ as $\bx \to \bx_0$ from $\partial \Omega.$ It follows immediately that $\rho \pm q = \rho \pm \br \cdot \bn(\bx_0) \approx \rho$ in the limit $\bx\to \bx_0$ from $\partial \Omega.$ Similarly, 
$$\bn(\bx) = \frac{-\kappa_1 \al{p_1}\, \hat{\bf t}_1(\bx_0)-\kappa_2 \al{p_2}\, \hat{\bf t}_2(\bx_0)+\bn(\bx_0)}{\sqrt{1 + \al{p_1}^2   \kappa_1^2+\al{p_2}^2\kappa_2^2}} +O(\rho^2),$$
and hence $\bn(\bx)\cdot \br = -\frac{1}{2}(\kappa_1 \al{p_1}^2+\kappa_2 \al{p_2}^2)+O(\rho^3).$ 
Additionally, this expression gives us that $\hat{\bf t}_{1,2}(\bx_0)\cdot \bn(\bx)=O(\rho)$ and $\bn(\bx)-\bn(\bx_0) = O(\rho).$ Using the above estimates, we find
$$W^\pm =\pm1 +O(\rho^2),$$
and so
$$\partial_n \widetilde{Q}^\pm (\al{p_1},\al{p_2},q;\bx_0) = -\frac{\bn(\bx)\cdot\br}{2\pi \rho^3} - \frac{\mathcal{H}(\bx_0)}{4\pi \rho}- \mathcal{Q}(\bx_0)\frac{\al{p_1}^2-\al{p_2}^2}{4\pi \rho^3} +O(1).$$
Simplifying, this gives
\begin{align}\label{eqn:rhsasym}
\partial_n \widetilde{Q}^\pm (\al{p_1},\al{p_2},q;\bx_0) = -\frac{\bn(\bx)\cdot\br + \kappa_1 \al{p_1}^2/2 +\kappa_2 \al{p_2}^2/2}{2\pi \rho^3} +O(1).
\end{align}
Now, for $\bx$ near $\bx_0,$ the asymptotic form of $\bn(\bx)\cdot \br$ given above tell us that the leading order terms in $\partial_n \tilde{Q}^\pm$ cancel, and $\partial_n \tilde{Q}^\pm = O(1)$ as $\bx\to \bx_0$ from $\partial \Omega.$
\end{proof}
}

Given a density $\sigma \in L^2(\partial \Omega),$ we define the standard single layer operator $S$ (see \cite{kress}, \jgh{Chapter 6} for example) by
$$S[\sigma](\bx) = \int_{\partial \Omega} G_F(\bx -\bx')\,\sigma(\bx')\,{\rm d}A(\bx'), \qquad G_F(\bx) = \frac{1}{4\pi|\bx|}.$$
When thinking of it as a map from $L^2(\partial \Omega) \to L^2(\partial \Omega)$ we denote it by $\mathcal{S}.$ It is well-known for example (see \cite{kress} \jgh{Chapter 6 and the references therein}), that for $\sigma \in L^2(\partial \Omega),$ the normal limit of the normal derivative of $S$ exists, in an $L^2$ sense, and
$$\lim_{\epsilon \to 0} \hn({\bx}) \cdot \nabla S[\sigma] (\bx\pm\epsilon \, \hn(\bx)) = \mp \frac{1}{2}\sigma(\bx) + {\rm p.v.} \int_{\partial \Omega} \hn({\bx}) \cdot \nabla_{\bx} G_F(\bx-\bx')\, \sigma(\bx ')\,{\rm d}A(\bx').$$
The integral on the right-hand side is weakly-singular and so the principal value can be dropped. For $\sigma \in L^2(\partial \Omega)$ we let $\mathcal{S}':L^2(\partial \Omega) \to L^2(\partial \Omega)$ denote the operator defined by
$$\mathcal{S}'[\sigma](\bx) = \int_{\partial \Omega} \hn({\bx}) \cdot \nabla_{\bx} G_F(\bx-\bx')\, \sigma(\bx ')\,{\rm d}A(\bx').$$

\subsubsection{The exterior problem}

When solving the exterior problem we consider the ansatz
$$G^e(\bx;\bx_0) = Q^+_{\eta_0,\eta_1,\eta_2}(\bx;\bx_0) + S[\sigma_{\bx_0}](\bx)$$
where $\sigma_{\bx_0}$ is some yet to be determined density. Note that by construction $G^e$ is harmonic in the exterior region $\mathbb{R}^3\setminus\Omega.$ Taking the normal limit of the normal derivative of $G^e,$ and inserting it into our boundary condition, we obtain the integral equation
\begin{align}\label{eqn:ext_bie}
-\frac{1}{2} \sigma_{\bx_0}(\bx)+ \mathcal{S}'[\sigma_{\bx_0}](\bx) = - \partial_n Q^+_{\eta_0,\eta_1,\eta_2}(\bx;\bx_0).   
\end{align}
Here, $\partial_n Q^+$ denotes the normal derivative of $Q^+$ on surface, i.e. excluding the delta function. Standard results give existence and uniqueness for (\ref{eqn:ext_bie}), see \cite{kress} for example.

\subsubsection{The interior problem}

Turning to the interior problem, we use the ansatz
$$G^i(\bx;\bx_0) = Q^-_{\eta_0,\eta_1,\eta_2}(\bx;\bx_0) + S[\sigma_{\bx_0}](\bx) + \frac{\al{p_1}^2+\al{p_2}^2}{4 |\Omega|}+\frac{\beta}{|\Omega|},$$
where again $\sigma_{\bx_0}$ is some yet to be determined density and $\beta$ is  constant. The first term contributes the required delta function to $\partial_n G^i(\bx;\bx_0)$ near $\bx =\bx_0$ and the third term enforces the constraint $\Delta G^i = 1/|\Omega|$ in $\Omega.$ If we again take the normal limit of the normal derivative of $G^i,$ and insert it into our boundary condition, we obtain the integral equation
\begin{align}\label{eq:int_BC}
\frac{1}{2} \sigma_{\bx_0}(\bx)+ \mathcal{S}'[\sigma_{\bx_0}](\bx) = - \partial_n Q^-_{\eta_0,\eta_1,\eta_2}(\bx;\bx_0) - \frac{1}{|\Omega|}\partial_n p(\bx),
\end{align}
with $p(\bx) = (\al{p_1}^2+\al{p_2}^2)/4.$ The operator on the left-hand side of the above equation has a one-dimensional nullspace. Indeed, it is easily shown \cite{kress} that for any $\mu \in L^2(\partial \Omega),$ 
$$\int_{\partial \Omega}\mathcal{S}'[\mu](\bx)\,{\rm d}A(\bx) = -\frac{1}{2}\int_{\partial \Omega} \al{\mu(\bx)}\,{\rm d}A(\bx),$$
i.e. constant functions are in the left nullspace. Consistency of the integral equation then requires
$$ \int_{\partial \Omega} \left[\partial_n Q^-_{\eta_0,\eta_1,\eta_2}(\bx;\bx_0) + \frac{1}{|\Omega|}\partial_n p(\bx) \right]\, {\rm d}A(\bx) = 0,$$
which follows from Green's theorem. To remove the nullspace, we impose the additional condition that $\int_{\partial \Omega} {\sigma}_{\bx_0}(\bx)\,{\rm d}A(\bx) = 0.$ We then arrive at the integral equation 
\begin{align}\label{eqn:int_bie}
\frac{1}{2} {\sigma}_{\bx_0}(\bx)+ \mathcal{S}'[{\sigma}_{\bx_0}](\bx) + \frac{1}{|\jgh{\partial \Omega}|}\int_{\partial \Omega} {\sigma}_{\bx_0}(\bx')\,{\rm d}A(\bx') = - \partial_n Q^-_{\eta_0,\eta_1,\eta_2}(\bx;\bx_0) - \frac{1}{|\Omega|}\partial_n p(\bx).
\end{align}
To satisfy the volume integral constraint (\ref{eq:eqnG_interior_c}) on $G^i$, we require that
\begin{align*}
\beta = -\int_{\Omega} \left[ Q^-(\bx) + S[{\sigma}_{\bx_0}](\bx) + \frac{\al{p_1}^2+\al{p_2}^2}{4 |\Omega|}\right]\,{\rm d}V(\bx).    
\end{align*}
 To simplify the right-hand side, we note that $\Delta p = 1,$ and by integrating by parts,
\begin{align*}
 \beta &=  -\int_{\Omega} \, \underbrace{\Delta p}_{=1} \, \left[ Q^-(\bx) + S[{\sigma}_{\bx_0}](\bx) + \frac{\al{p_1}^2+\al{p_2}^2}{4 |\Omega|}\right]\,{\rm d}V(\bx) \\
  &= - \int_{\partial \Omega} \partial_n p \left[ Q^-(\bx) + S[{\sigma}_{\bx_0}](\bx) + \frac{\al{p_1}^2+\al{p_2}^2}{4 |\Omega|}\right]\,{\rm d}A(\bx)\\
   &\quad + \int_{\partial \Omega} p(\bx) \partial_n \al{G^i(\bx;\bx_0)}\,{\rm d}A(\bx) -\frac{1}{|\Omega|}\int_{\Omega} p(\bx)\,{\rm d}V(\bx).
\end{align*}
Using the fact that $\partial_nG^i(\bx;\bx_0) = -\delta(\bx -\bx_0)$ we see that since $p(\bx_0)=0,$ the second term on the right-hand side of the above equation vanishes. Additionally, $p(\bx) = \nabla \cdot\frac{1}{12}(\htan_1(\bx_0) \al{p_1}^3+\htan_2(\bx_0) \al{p_2}^3),$ so the last integral in the above expression can be computed via a boundary integral. Putting this all together, the condition for $\beta$ may be written in the form
$$\beta = -\gamma - \int_{\partial \Omega} \sigma_{\bx_0} S[\partial_np](\bx)\,{\rm d}A(\bx),$$
where we have applied Fubini's theorem to the integral and 
$$\gamma =\int_{\partial \Omega} \left[\partial_n p \Big[ Q^-(\bx;\bx_0) + \frac{\al{p_1}^2+\al{p_2}^2}{4 |\Omega|}\Big]+ \frac{1}{|\Omega|} \hn({\bx}) \cdot \frac{1}{12}\Big(\htan_1({\bx_0}) \al{p_1}^3+\htan_2({\bx_0}) \al{p_2}^3\Big)\right]\, {\rm d}A(\bx).$$
Note that despite the singularities in $Q^-,$ the integrand vanishes at $\bx_0.$

\begin{remark}
In \cite{chakraborty2025fastintegralmethodsneumann}, a slightly different formulation is given. In essence, in that paper the constant term $\beta$ is determined at the same time as the density $\sigma_{\bx_0}.$ In this work we first compute $\sigma_{\bx_0}$ and then obtain $\beta$ afterwards in the postprocessing step.
\end{remark}

\begin{remark}
    For off-surface sources, either exterior or interior, the singular part of the Neumann Green's function is simply the free-space Laplace Green's function, and the regular part has smooth normal derivatives on $\partial \Omega.$ For the exterior problem, the regular part can be obtained by using the ansatz $G(\bx;\bx_0) = G_F(\bx-\bx_0)+S[\sigma_{\bx_0}](\bx)$ and solving 
    (\ref{eqn:ext_bie}) with a different right hand side (see~\cite{kress} for more details). For the interior problem, we use the ansatz $G(\bx;\bx_0) = G_F(\bx-\bx_0)+S[\sigma_{\bx_0}](\bx)+\frac{|\bx|^2}{6|\Omega|}+\frac{\beta}{|\Omega|}$, solve (\ref{eqn:int_bie}) with the corresponding right hand side, and compute $\beta$ as before (see~\cite{chakraborty2025fastintegralmethodsneumann} for more details).
\end{remark}

\subsection{Discretization of the integral equations}
In this paper we use a standard collocation method for the efficient discretization and solution of the integral equations (\ref{eqn:ext_bie}) and (\ref{eqn:int_bie}), which are both of the general form
\begin{align}\label{eqn:ref_int_eq}
\sigma + K[\sigma] = f.
\end{align}
To make the presentation as self-contained as possible we briefly sketch the main idea of the approach. We refer the interested reader to \cite{loc_cor_quad} for more details and references, and to the software package \emph{fmm3dbie} \cite{fmm3dbie} which we have used in our numerical illustrations.

\subsubsection{Discretization of the equations}\label{subseq:disc}
In the following we let $T_0$ denote the standard right-triangle $T_0= \{(x,y) \in \mathbb{R}_{> 0}^2\,|\, x+y < 1\}$ and assume that the boundary of $\Omega$ is specified via a union of smooth mutually-disjoint maps $\phi_j:T_0 \to \partial \Omega,$ with $j=1,\ldots, N_{\rm patches}.$ Let $T_j$ denote the image of $T_0$ under $\phi_j.$

The integral equation (\ref{eqn:ref_int_eq}) can then be written as
\begin{align}\label{eqn:bie_sum_patches}
\sigma(\bx) + \sum_{j=1}^{N_{\rm patches}}\int_{T_0} K(\bx, \phi_j(u,v))\,  \sigma(\phi_j(u,v))\,|\,g_j(u,v)|\,{\rm d} A(u,v)  = f (\bx),
\end{align}
where $g_j$ denotes the determinant of the metric tensor of the $j^{\text{th}}$ patch which we assume to be boundedly invertible on the closure of $T_0.$ 

Next, we fix a basis $\tau_k:T_0 \to \mathbb{R},$ $k=1,\ldots, N_{\rm basis},$ and let $\{(u_k,v_k)\}_{k=1}^{N_{\rm basis}} \subset T_0$ be an associated set of discretization nodes. Typically, we assume that the nodes are constructed so that the coefficients-to-values matrix $\al{V}:=\{\tau_\ell (u_k,v_k)\}_{k,\ell=1}^{N_{\rm basis}}$ is as well-conditioned as possible. We then relax (\ref{eqn:bie_sum_patches}), instead only requiring consistency at our discretization nodes,
\begin{align}\label{eqn:bie_sum_patch_red}
\sigma_i(u_k,v_k) + \sum_{j=1}^{N_{\rm patches}}\int_{T_0} K(\phi_i(u_k,v_k), \phi_j(u,v))\,  \sigma_j(u,v)\,{|g_j|}\,{\rm d} A  = f_i(u_k,v_k),
\end{align}
where $\sigma_i(u,v) = \sigma(\phi_i(u,v))$ and $f_i(u,v) = f(\phi_i(u,v)).$ Assuming that $\sigma_i$ is approximately in the span of the basis functions $\tau_k,$ then
$$\sigma_i(u,v) \approx \sum_{j=1}^{N_{\rm basis}} \tau_j(u,v) (V^{-1})_{j,k} \sigma_i(u_k,v_k),$$
which we can insert into (\ref{eqn:bie_sum_patch_red}) to obtain
\begin{align}\label{disc_syst}
\sigma_{i,k} + \sum_{j=1}^{N_{\rm patches}} \sum_{\ell=1}^{N_{\rm basis}}K_{i,k,j,\ell} \sigma_{j,\ell} \approx f_{i,k}, \qquad i=1,\ldots, N_{\rm patches};\, \quad k=1,\ldots,N_{\rm basis},
\end{align}
with $\sigma_{i,k} = \sigma_i(u_k,v_k),$ $f_{i,k} = f_i(u_k,v_k)$ and
\begin{align}\label{eqn:k_def}
K_{i,k,j,\ell} = \sum_{m=1}^{N_{\rm basis}}(V^{-1})_{m,\ell}\int_{T_0}K(\phi_i(u_k,v_k), \phi_j(u,v))\,  \tau_m(u,v) \,{|g_j|}\,{\rm d} A.
\end{align}

In this paper, except for the patches touching $\bx_0,$ we choose the basis on $T_0$ to be Koornwinder polynomials of degree at most $N_{\rm order},$ which give an orthonormal basis of the monomials $u^j v^k,$ $j,k \ge 0,$ $j+k \le N_{\rm order}$ on $T_0$ (we defer discussion of the near-$\bx_0$ patches to Section \ref{sec:rhs}). With this choice, $N_{\rm basis} = (N_{\rm order}+1)(N_{\rm order}+2)/2.$ For the corresponding discretization nodes, we use the $N_{\rm order}$ Vioreanu-Rokhlin (VR) nodes \cite{vioreanu2014spectra}, which are chosen so as to give a well-conditioned $V$ matrix. Associated with the VR nodes $(u_j,v_j)$ is a set of quadrature weights $\{w_j\}_{j=1}^{N_{\rm basis}}$ such that the quadrature rule with nodes $(u_j,v_j)$ and weights $w_j$ integrate all Koornwinder polynomials of degree at most $N_{\rm order}$ exactly.

\subsubsection{Quadrature for weakly-singular integral operators}
The remaining piece is the numerical calculation of the integrals $K_{i,k,j,\ell}$ in (\ref{eqn:k_def}). Here we give a brief sketch of the approach taken in \emph{fmm3dbie}, more details of which can be found in \cite{loc_cor_quad}. Broadly speaking, for a given {\it target} patch number $i$ we split the {\it source} patches $j = 1,\ldots,N_{\rm patches}$ into three sets: self ($i=j$), near patches ($i \neq j$ but the $i^{\text{th}}$ patch is `close' to the $j^{\text{th}}$ patch), and far (patches which are neither self nor near). In our numerical experiments we use the default settings of \emph{fmm3dbie} \cite{fmm3dbie, bremer, xiaogimbutas}.

\subsubsection{Far quadrature}
When the source patch $T_j$ is a sufficient distance from the target patch $T_k,$ the kernel $K(\phi_i(u_k,v_k),\phi_j(u,v))$ is in $C^\infty(\overline{T_0}),$ as is $\sqrt{|g_j|}.$ For these integrals we use the \emph{native quadrature rule}, 
\bsub
\begin{align}
    K_{i,k,j,\ell}&\approx \sum_{n=1}^{N_{\rm basis}}\sum_{m=1}^{N_{\rm basis}}\al{\left[K(\phi_i(u_k,v_k),\phi_j(u_n,v_n)) \tau_m(u_n,v_n) \left| g_j(u_n,v_n)\right|\,w_{n}\right] (V^{-1})_{m,\ell}},\\
    &=\al{K(\phi_i(u_k,v_k),\phi_j(u_l,v_l)) \left| g_j(u_l,v_l)\right|\,w_{l}},
\end{align}
\esub
where we have used the definition of $V$ in going from the first line to the second.

\begin{remark}\label{rem:oversampling}
In practice the above integrals are frequently computed using oversampling: integrals are computed using a Vioreanu-Rokhlin quadrature rule that is higher order than the order of the polynomial expansions. In most settings, it is infeasible to assume that one has an analytic description of the surface available at the time when matrix entries are being computed. We describe how to address this in the next section.
\end{remark}

\subsubsection{Near quadrature}
For source patches $T_j$ lying close to the target patch $T_i,$ the kernel $K$ will be nearly-singular, and so its coefficients when expressed in the Koornwinder polynomial basis will decay extremely slowly, if at all. In this case, the native quadrature rule, or the oversampling procedure alluded to in Remark \ref{rem:oversampling}, will not be accurate enough without an impractical number of nodes. Instead, we use adaptive integration, subdividing the patch $T_j$ into smaller patches recursively until the error is within a prescribed tolerance. In practice, when splitting patches one rarely has access to an analytic description of the surface to query. Instead, the maps  $\phi_j$ are approximated by maps $\hat{\phi}_j$ given by Koornwinder polynomial expansions. In particular, we represent our surface $\partial \Omega$ by storing the values of $\phi_j,$ $\partial_u \phi_j,$ $\partial_v \phi_j,$ and other surface properties (normals, shape operators, etc.) at Vioreanu-Rokhlin nodes $(u_k,v_k)$ for each patch. If, when the surface is being constructed, the derivatives of the maps $\phi_j$ are not given analytically, then approximations are computed via numerical differentiation. With these approximate maps $\hat{\phi}_j,$ we can interpolate relevant geometric quantities to any point in $T_0.$ We refer the interested reader to \cite{loc_cor_quad} for a more thorough discussion of this procedure and its attendant considerations.

\subsubsection{Self quadrature}\label{subseq:self}
For the self quadrature, when the source patch and the target patch are the same, once again the singularities of $K$ preclude using the native quadrature rule, though adaptive integration, as for the near quadrature, can still be used. In practice, since many levels of refinement are required for adaptive integration on self patches, this can become prohibitively slow. Instead, a standard approach is to use {\it generalized Gaussian quadratures} \cite{bremer}, which \al{use} preconstructed tables of specially constructed quadrature rules which are designed to integrate functions with singularities like those found in our kernels. We refer the reader to \cite{bremer, loc_cor_quad} for a detailed description of these quadratures and their use.

\subsubsection{Singularities in the right-hand sides}\label{sec:rhs}
For smooth domains and smooth right-hand sides, the above discretization and quadrature procedure yields high-order convergence, and can be easily coupled to standard fast algorithms such as the fast multipole method (FMM)~\cite{greengard1987fast} and fast direct solvers~\cite{martinsson2025fast,ho2012fast}. When sources are present on the boundary, however, the right-hand side contains $\partial_n Q^{\pm}_{\eta_0,\eta_1,\eta_2}$ which is not smooth in the vicinity of $\bx_0$ (see Fig.~\ref{fig:examp_Q}). More specifically, it is easily seen that the right-hand side is smooth (except for a removable singularity) on any smooth curve passing through the source $\bx_0.$ The limiting behavior as $\bx \to \bx_0$, however, depends on the direction of approach. To see this, we note that the images at $\bx_0+\eta_j \hat\bn(\bx_0)$ for $j=0,1,2$ contribute a smooth term to $\partial_n Q^{\pm}_{\eta_0,\eta_1,\eta_2}$. The remaining piece can be expressed as a ratio of smooth functions of $\mu$ and $\nu$, whose limiting behavior is given by the right hand side of \eqref{eqn:g_2b} as $\bx\to\bx_0$.

\begin{figure}
\centering
\subfigure[Right hand side of \eqref{eq:int_BC}: $-\partial_n Q^{-}(\bx;\bx_0)- \frac{1}{|\Omega|}\partial_n p(\bx)$.\label{fig:examp_Q}]{\includegraphics[width = 0.475\textwidth]{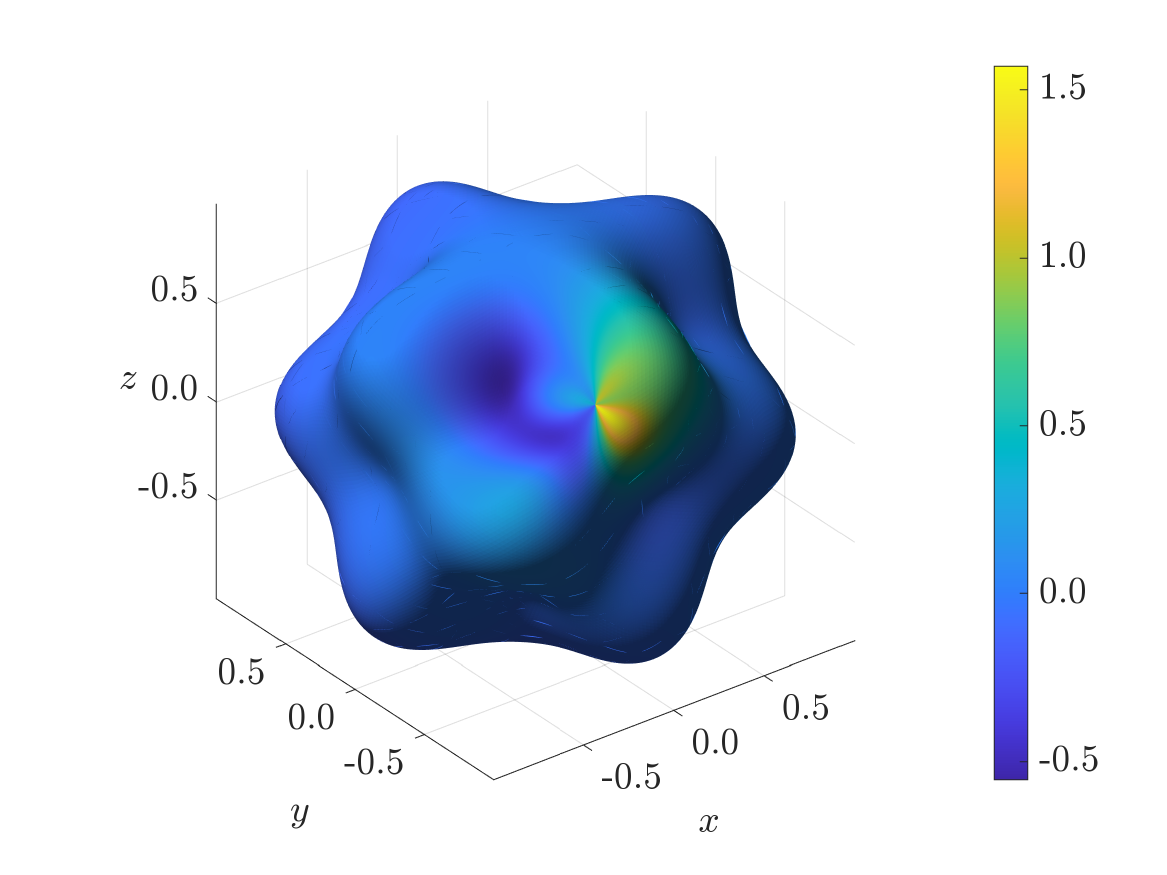}}
\qquad
\subfigure[Second coordinate $v\in(0,1)$ or $t\in (-1,1)$ on each patch. \label{fig:blob_coords}]
{\includegraphics[width = 0.475\textwidth]{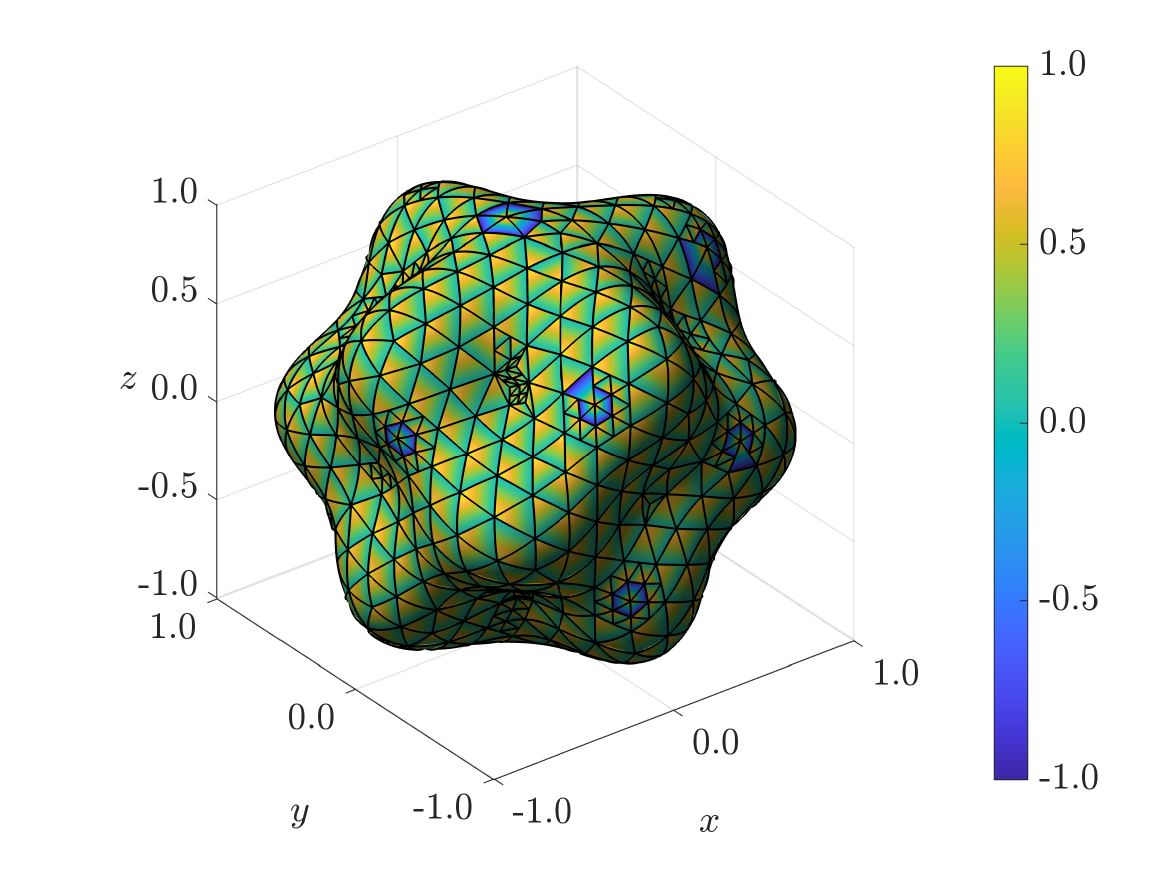}}
\caption{Example right hand side and discretization for a sphere with a radius that has been perturbed by a spherical harmonic, $0.3\cdot  \Re Y_6^4$. Left: An example right hand side for (\ref{eq:int_BC}), which has a singularity at the source $\bx_0=( -0.2910,-0.8131,0.5674)\in\partial\Omega$. Right: The orientation of our mesh elements. In this discretization, we have introduced a few collections of Duffy patches, so that the same quadratures can be used to compute $G(\bx;\bx_0)$ for multiple $\bx_0$.}
\end{figure}

This singularity presents two difficulties. Firstly, neither the density nor the right-hand side are well-represented in a piecewise polynomial basis, leading to slow convergence. Secondly, the formula for $\partial_n Q^{\pm}$ involves cancellations between singular terms which, if computed naively, give rise to catastrophic cancellations. Here we summarize our approach for overcoming these issues.

The first step is to ensure that the source $\bx_0$ coincides exactly with a vertex of the triangulation of $\{T_i\}$. This can be easily done with standard meshing tools such as Gmsh~\cite{gmsh} or distmesh~\cite{distmesh}. Let $\mathcal{I}$ be the set of indices $i$ for which the corresponding triangle $T_i$ in the mesh has a vertex at $\bx_0$. For convenience, we assume that these patches are oriented so that $\phi_i(0,1)=\bx_0$ for $i\in \mathcal{I}$. For each $\theta$, $\bs r_\theta(r):=\phi_i(r \cos\theta, 1-r\sin\theta)$ is a curve passing through $\bx_0$ that smoothly depends on $\theta$.
The above discussion then implies that
$\partial_nQ^\pm_{\eta_0,\eta_1,\eta_2}\left(\phi_i(r \cos\theta, 1-r\sin\theta)\right)$ will be a smooth function of $r$ and $\theta$ for each $i\in\mathcal{I}$. As mentioned above, the polar singularity prevents $\partial_nQ^\pm_{\eta_0,\eta_1,\eta_2}\circ \phi_i$ from being well represented by the basis of Koornwinder polynomials. Instead, we introduce the Duffy transform~\cite{duffy1982quadrature}
\begin{align*}
    \psi(s,t) = \left(\frac{s+1}{2} \frac{1-t}{2},\; \frac{t+1}{2}\right), 
\end{align*}
which is a bijection from the square $(-1,1)^2$ into $T_0$. This parameterization is chosen so that the lines of constant $s$ are mapped to different straight lines passing through the point $(0,1)$. Thus 
\begin{align*}
    \partial_nQ^\pm_{\eta_0,\eta_1,\eta_2}\circ \phi_i \circ \psi(s,t),
\end{align*}
will be a smooth function of $(s,t)\in (-1,1)^2$. It can therefore be well represented in the basis $\tilde \tau_{k,k'}(s,t) := P_k(s) P_{k'}(t)$ with $k,k'=1,\ldots ,N_{\rm order}+1$ of tensor product Gauss-Legendre polynomials. With this in mind, on the square $(-1,1)^2$ we choose our discretization points to be tensor products of Gauss-Legendre nodes: $\{(s_k,s_{k'})\}_{k,k'=1}^{N_{\rm order}+1}$. An example of the patch orientation is shown in Fig.~\ref{fig:blob_coords}. With this choice of basis function and discretization nodes, the discretization and quadrature scheme summarized in Sections \ref{subseq:disc}-\ref{subseq:self} can be extended \emph{mutatis mutandis} to accommodate these modified patches, though we choose the order of these `Duffy' patches to be four higher than the order of triangular patches. Extensive numerical evidence suggests that the density $\sigma$ is also well-resolved on the Duffy patches.

The addition of Duffy patches enables high-order convergence, even near the source, and so eliminates the need for excessive refinement. \al{This, in turn, reduces the effect of catastrophic cancellations, since quadrature nodes within the larger Duffy patches on the right-hand side are typically farther from the source than they would be if many levels of refinement were used.} Still, for high accuracy and robustness, care is required when evaluating $\partial_nQ^\pm_{\eta_0,\eta_1,\eta_2}$ at nodes close to $\bx_0,$ as the expression contains a difference of singular terms. Specifically, it can be expressed as a sum of bounded terms plus \al{the fraction in 
(\ref{eqn:rhsasym}).}
While the other terms can be easily evaluated without catastrophic cancellations, the cancellation of the \al{three terms in the numerator of (\ref{eqn:rhsasym}) is fairly delicate. Let
$$F(s,t) = \hn(s,t) \cdot (\bx(s,t) - \bx_0)+\frac12\kappa_1 p_1(s,t)^2 +\frac12 \kappa_2 p_2(s,t)^2,$$
be the numerator from that expression.}
To stably evaluate $F,$ we note \al{that} $F(s,1) = \partial_t F(s,1) = \partial_t^2 F(s,1) = 0$~\cite{hsiao2021boundary}, and hence
\begin{align}
    F(s,t) = -\int_t^1 \frac12 (t-t')^2\partial_t^3 F(s,t')\,{\rm d}t'.\label{eq:F_dttt}
\end{align}
This formula explicitly uses the fact that $F = O(|\bx-\bx_0|^3)$ and does not involve cancellations, and so gives a stable way for computing the most singular part of $\partial_nQ^\pm_{\eta_0,\eta_1,\eta_2}$. We discretize the integral in (\ref{eq:F_dttt}) using a 16th order Gauss-Legendre rule. In order to evaluate the integrand, we use the parameterization or implicit function defining our surface to compute $\partial_t^3F$ at the nodes $(s_k,s_{k'})$. We then interpolate it to the required positions. 
\al{For surfaces where the required derivatives are not known, we instead approximate $\partial_t^3F$ by differentiating the tensor product Legendre interpolant of $F$ on the patch.} While the accuracy of these numerically computed derivatives is limited, this approach still avoids the catastrophic cancellation inherent in $F$ and enables a sufficiently accurate evaluation of  $\partial_nQ^\pm_{\eta_0,\eta_1,\eta_2}$.

\subsection{Derivatives with respect to source location}
In many applications the computation of the Neumann Green's function, and its regular part, appear in an objective function in which the location of the source appears as a parameter to be optimized over. If gradients with respect to source locations are computed using finite difference, catastrophic cancellation can lead to loss of accuracy. Instead, one can obtain these derivatives directly from the integral equation. For sources which are off-surface (either in the interior or exterior), the right hand side $f$ in the integral equation (\ref{eqn:ref_int_eq}) is smooth, and can be differentiated in $\bx_0.$ The integral equation can then be discretized and solved `as-is' with this modified right-hand side in order to compute the derivative of the solution with respect to $\bx_0.$

For on-surface sources, differentiating $Q^{\pm}$ with respect to $\bx_0$ leads to a more singular right-hand side. In principle this can be handled by changing the basis used to represent the density in the vicinity of the source. Computing the right-hand side would still require care to avoid catastrophic cancellations in the evaluation of the more singular data. An alternative approach is to use reciprocity. In particular, if $\bx,\bx_0 \in \partial \Omega,$ $\bx \neq \bx_0,$ Green's identities give
$$G(\bx;\bx_0) = G(\bx_0;\bx).$$
Hence, $\nabla_{\bx_0}G(\bx;\bx_0) = \nabla_{\bx_0}G(\bx_0;\bx).$ Derivatives with respect to the target do not change the right-hand side or the solution of (\ref{eqn:ref_int_eq}), and appear only in the post-processing to evaluate the field.

\section{Results}\label{sec:Results}
In this section we demonstrate our method on a variety of examples and benchmark against several scenarios where closed-form or highly accurate solutions of \eqref{eq:eqnG_exterior} and \eqref{eq:eqnG_interior} are available.

\subsection{Sphere}\label{eq:sphere_case}
For the unit sphere, closed-form solutions for both the interior and exterior Green's functions are available. 

\noindent\underline{Interior bulk:} For $\bx\in\Omega$, $\bx' = \bx/|\bx|^2$ and $\bx_0\in \Omega$, we have that (see \cite[Appendix A]{ChevWard2010})
\bsub\label{eq:bulk_greens_int_sphere}
\begin{align}
\label{eq:bulk_greens_int_sphere_a} G^i(\bx;\bx_0) &= \frac{1}{4\pi|\bx-\bx_0|} + R^i(\bx;\bx_0),\\
\label{eq:bulk_greens_int_sphere_b} R^i(\bx;\bx_0)& = \frac{1}{4\pi|\bx||\bx'-\bx_0|} + \frac{1}{4\pi} \log\left[\frac{2}{1-\bx\cdot\bx_0 + |\bx||\bx'-\bx_0|}\right] + \frac{1}{8\pi} \left( |\bx|^2 + |\bx_0|^2\right) - \frac{7}{10\pi}.
\end{align}
\esub
\noindent\underline{Interior surface:} For $\bx_0\in\partial\Omega$ and $\bx\in\Omega$, we have that \cite[Lemma 2.1]{cheviakov2010asymptotic} 
\bsub
\begin{equation}\label{eq:surface_greens_int}
G^i(\bx;\bx_0) = \frac{1}{2\pi|\bx-\bx_0|} + \frac{1}{4\pi} \log \left(\frac{2}{1- \bx\cdot \bx_0 + |\bx-\bx_0| }\right)  + \frac{1}{8\pi}\big( |\bx|^2 + 1\big) -\frac{7}{10\pi}.
\end{equation}
In the limit as $\bx\to\bx_0$, the regular part  can be identified to be
\begin{equation}\label{eq:surface_Reg_int}
R^i(\bx_0;\bx_0) = \frac{1}{4\pi} \log 2 - \frac{9}{20\pi}.
\end{equation}
\esub
\noindent\underline{Exterior surface:} For $\bx_0\in\partial\Omega$ and $\bx\in\mathbb{R}^3\setminus\Omega$, we have that (cf.~\cite{SurfaceGreen3D})
\bsub\label{eq:surface_greens_ext}
\begin{equation}\label{eq:surface_greens_ext_a}
    G^e(\bx;\bx_0) =  \frac{1}{2\pi|\bx-\bx_0|} - \frac{1}{4\pi} \log \left(\frac{1- \bx\cdot \bx_0 + |\bx-\bx_0|}{ |\bx| - \bx\cdot\bx_0 }\right).
\end{equation}
In the limit as $\bx\to\bx_0$, we identify from \eqref{eq:surface_greens_ext_a} that the exterior regular part is
\begin{equation}\label{eq:surface_Reg_ext_b}
    R^e(\bx_0;\bx_0) = -\frac{1}{4\pi} \log 2.
\end{equation}
\esub
In Fig.~\ref{eq:sphere_error} we display relative errors between these known exact solutions and our numerical method. In Fig.~\ref{eq:sphere_error_a} we plot the interior solution $R^i(\bx;\bx_0)$ for $\bx_0 = (-\frac14,\frac13,-\frac15)$ where $\bx$ are points lying on a plane through the origin. The relative error is approximately $10^{-12}$. In Fig.~\ref{eq:sphere_error_b}, we plot the relative errors in the approximation of $R^i(\bx_0;\bx_0)$ and $R^e(\bx_0;\bx_0)$, given by \eqref{eq:surface_Reg_int} and \eqref{eq:surface_Reg_ext_b} respectively, along an arc of points $\bx_0 = (\sqrt{1-\al{q}_0^2},0,\al{q}_0)$ for $\al{q}_0\in[-1,1]$. The numerical method provides an approximation of these values to a relative error of around $10^{-9}$.

\begin{figure}[htbp]
\centering
\subfigure[$R^i(\bx;\bx_0)$ for $\bx_0\in\Omega$.]{\includegraphics[width = 0.425\textwidth]{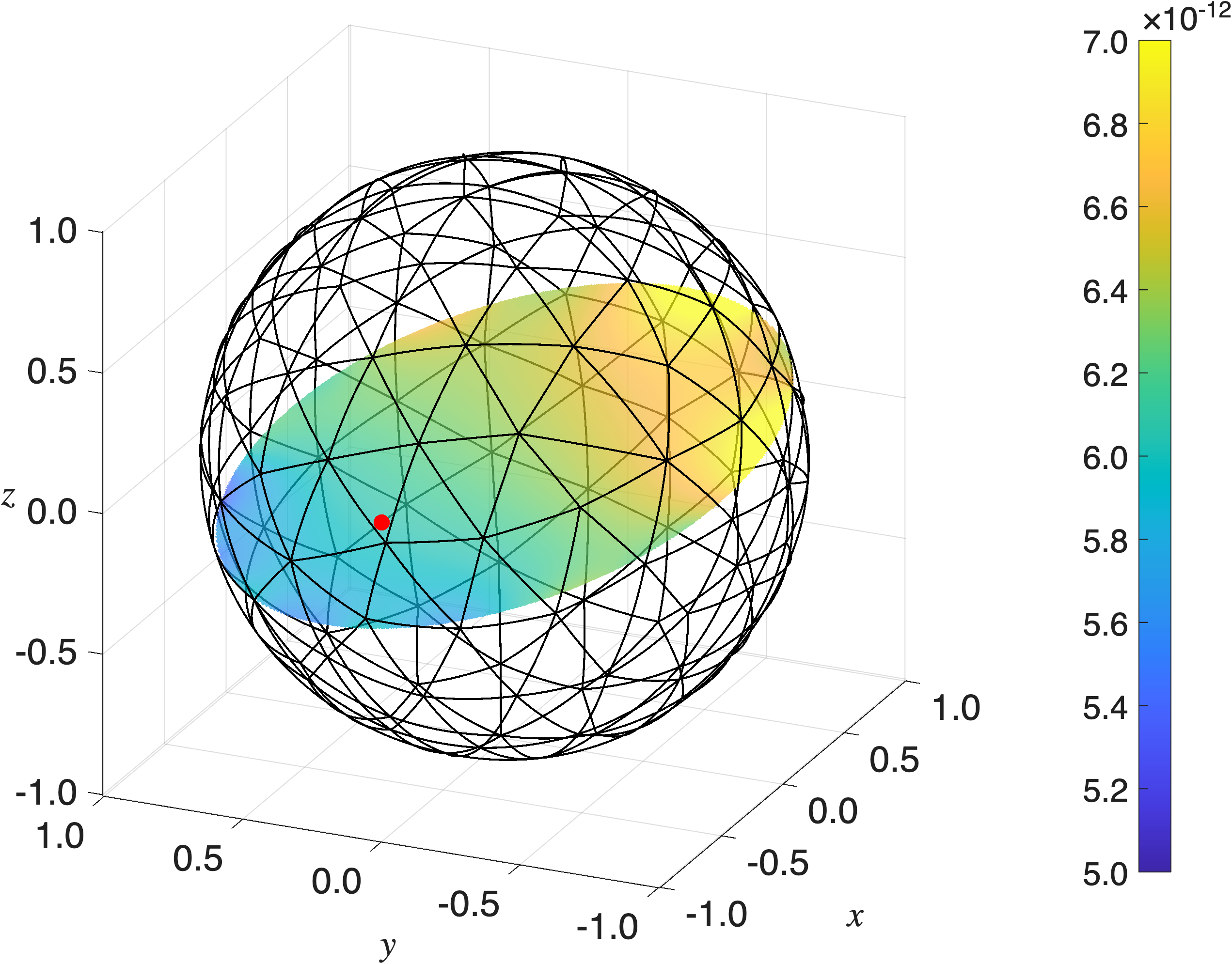}\label{eq:sphere_error_a}}\qquad
\subfigure[Relative errors in $R^e(\bx_0;\bx_0)$ and $R^i(\bx_0;\bx_0)$ for $\bx_0\in\partial\Omega$.]{\includegraphics[width = 0.475\textwidth]{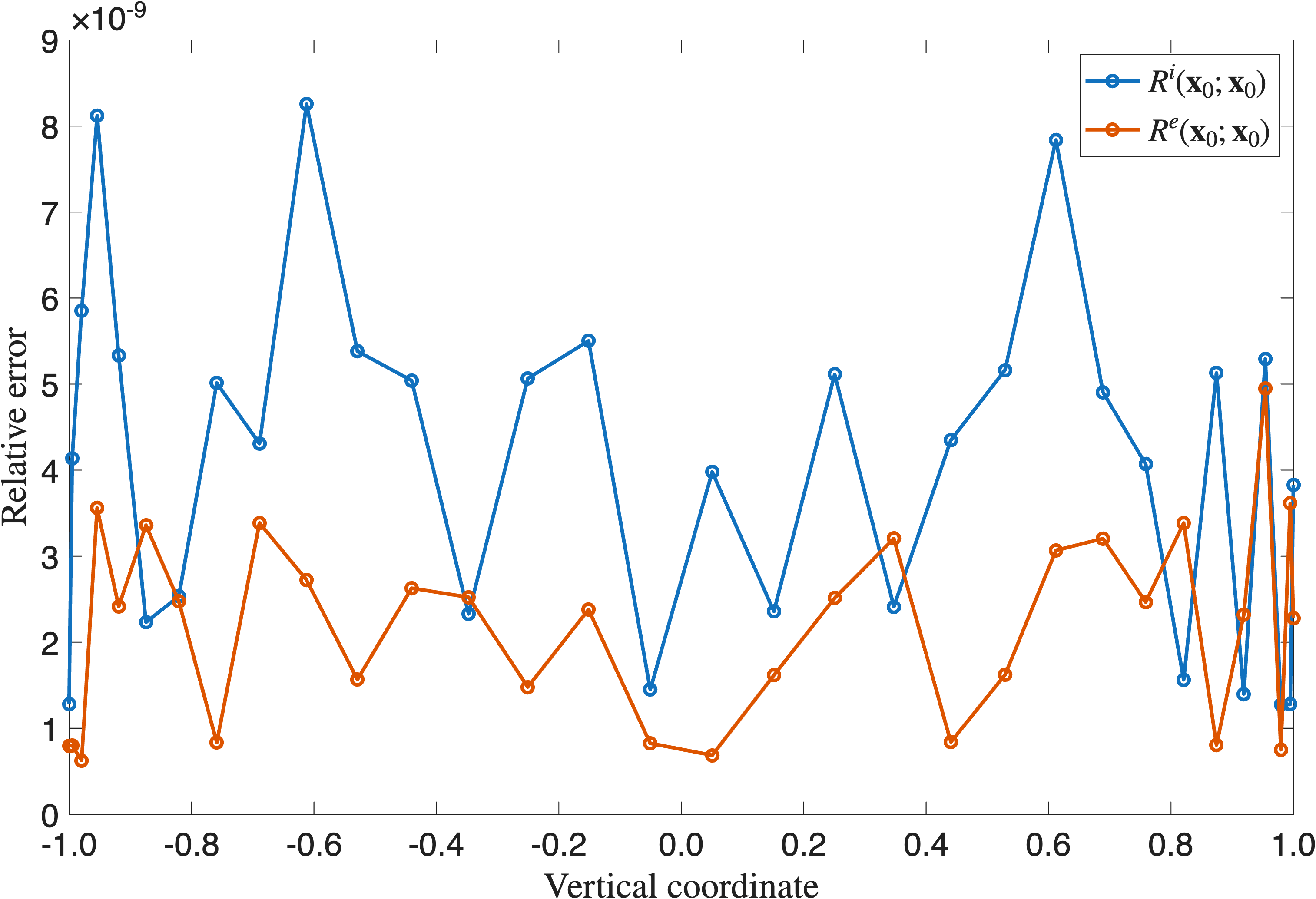}\label{eq:sphere_error_b}}

\caption{Error in numerical approximation in the sphere case. Left: Relative error in calculation of $R^i(\bx;\bx_0)$ for $\bx_0\in\Omega$ (red dot) and evaluated at points $\bx$ on an interior plane through the origin. The relative error is around $5\times10^{-12}$. Right: Relative errors for numerical evaluation of $ R^e{(\bx_0;\bx_0)}=-\frac{1}{4\pi}\log2$ and $R^i{(\bx_0;\bx_0)}=\frac{1}{4\pi}\log2 - \frac{9}{20\pi}$ for points $\bx
_0= (\sqrt{1-\al{q}_0^2},0,\al{q}_0)$ and $\al{q}_0 \in [-1,1].$ The agreement is around $10^{-9}$. \label{eq:sphere_error}}
\end{figure}

\subsection{A family of axi-symmetric solutions for the exterior Green's function}

In this section, we construct a family of domains and solutions for the exterior Green's function \eqref{eq:eqnG_exterior} and compare with our numerical method. Our strategy will be to first specify a harmonic function $G^e(\bx;\bx_0)$ with the correct singularity structure at the source and then determine a domain $\Omega$ that satisfies the Neumann boundary condition \eqref{eq:eqnG_exterior_a}. Throughout the construction, the surface source is located at $\vx_0=(0,0,1)\in\partial\Omega$ which will generate an axi-symmetric solution.
 
We will work in axi-symmetric spherical coordinates $(d, \theta)$ centered at the origin for which
\begin{equation}\label{eq:spherical_coords}
d= \sqrt{r^2+z^2}  \ ,  \qquad \theta = \tan^{-1} \left [\frac{r}{z} \right] \ ,
\end{equation}
where $r=\sqrt{x^2+y^2}$ is the distance to the polar axis. 
We define an orthonormal basis in the $(r,z)$ plane, 
$$\hat{\bd} = ( 
 \sin \theta ,
 \cos \theta) \ , \qquad 
 \hat{\bs{\theta}} = ( \cos \theta ,
 - \sin \theta ) \ , \quad 
$$
corresponding to radial and polar vectors
and note that $\frac{\partial \hat{\bd}}{\partial \theta} = \hat{\bs{\theta}}$. 
We also define cylindrical coordinates $(r,\al{q})$ centered at the source  $\vx_0= (0,0,1)$ for convenience,
$$r= \sqrt{x^2+y^2} = d \sin \theta \ , \quad
\al{q}=z-1 \ , $$
for which the distance to the source is given by
$$\rho = \sqrt{r^2+(z-1)^2} = \sqrt{d^2 - 2d \cos \theta +1 } \ 
.$$
Motivated by the exact solution for the sphere case in \al{Section} \ref{eq:sphere_case}, we propose axi-symmetric solutions for $G^e(\bx;\bx_0)$ and seek to find the domain 
$$\Omega = \{ \vx \ | \ d \le P(\theta) \quad \textrm{for} \quad 0 \le \theta \le \pi \}, $$
where $P(\theta)$ is to be determined. The conditions
\[
P(0)=1, \qquad P'(0)=P'(\pi)=0,
\]
place the source at the north pole of the surface and preclude a cusp at the north or south pole.
The ({un-normalized}) tangent vector $\mathbf{t}(P)$ and  normal vector $\mathbf{n}(P)$ to this curve are 
$$\mathbf{t}(P) = P'(\theta) \hat{\bd} +P(\theta) \hat{\bs\theta}, \qquad \mathbf{n}(P) = P(\theta) \hat {\bd} -P'(\theta) \hat{\bs\theta} \ .$$
For a known Green's function $G^e(d,\theta)$, the boundary condition $\partial_n G^e = 0$ then becomes
\begin{equation}\label{eqn:normalAxi}
    \partial_n G^e = \big(\al{\partial_d} G^e \, \hat{\bd} + \frac {1}{d} \al{\partial_{\theta}}G^e\, \hat{\bs\theta} \big)\cdot \big(P(\theta) \, \hat {\bd} -P'(\theta) \,\hat{\bs\theta}\big) = 0, \qquad 0\leq\theta\leq\pi.
\end{equation}
If $G^e(d,\theta)$ is known, then \eqref{eqn:normalAxi} yields the following ODE for $P(\theta)$,
\begin{equation}
    P'(\theta) =  P^2 \frac{\al{\partial_d}G^e(P,\theta)}{\al{\partial_{\theta}}G^e(P,\theta) } \ .
    \label{Pode}
\end{equation}
 At the source the surface has a local parameterization $\al{q}=\frac{\mathcal{H}}{2} r^2 +\bigoh(r^4) $ reflecting a smooth boundary with the prescribed mean curvature. As $\al{q} =z-1 =P(\theta) \cos \theta -1$ and $r = P(\theta) \sin \theta $, substituting and expanding yields
\begin{equation}
    P(\theta) \sim 1 + \frac{1+\Hc}{2}\theta^2 + \bigoh ( \theta^4),
    \label{BCnorth}
\end{equation} 
At the south pole ($\theta=\pi$) we want the surface to be smooth which implies 
$\lim_{\theta\to\pi} P'(\theta) = 0$. Moreover, the smoothness and axi-symmetry of the Green's function implies $\partial_{\theta}G^e(P,\pi)=0$.  From \eqref{Pode} we deduce that  
$$
\partial_d G^e(P,\pi) =0 \ ,
$$
In practice, we find some $d_*$ such that $\partial_{d}G(d_*,\pi)=0$ 
which determines the location of the south pole
and then enforce
$$\lim_{\theta \uparrow \pi} P(\theta) =d_*.$$

\subsubsection{The proposed Green's function}
Motivated by the spherical solution, we look for a surface with a specified $G^e$ of the form 
\begin{equation}
G^e(\vx;\bx_0) =G^+(d, \theta) = \frac{1}{2 \pi \rho} -
\frac{\mathcal{H}}{4 \pi } \log [ \rho+\al{q}]
- \frac{1}{4 \pi d} \left [1 +\mathcal{H}\right ] +\frac{\mathcal{H}}{4 \pi} \log [d+z] \ .
\end{equation}
However, this form is singular on the polar axis for $z \le 1$ where $\rho+\al{q}=0$. By noticing that $r^2=\rho^2-\al{q}^2=d^2-z^2$,
one can see that
$$\frac{\rho-\al{q}}{d-z} = \frac{d+z}{\rho+\al{q}},$$
which allows us to derive an alternative form of $G^e(\vx;\bx_0)$
\begin{equation}
G^e(\vx;\bx_0) =G^-(d, \theta) = \frac{1}{2 \pi \rho} +
\frac{\mathcal{H}}{4 \pi } \log [ \rho-\al{q}]
- \frac{1}{4 \pi d} \left [1 +\mathcal{H}\right ] -\frac{\mathcal{H}}{4 \pi} \log [d-z] \ ,
\end{equation}
which is valid everywhere except on the polar axis for $z \ge 0$.
Patching these two solutions together yields
\begin{equation}\label{eqn:GreenFull}
    G^e(\vx;\bx_0) = 
    \begin{cases}
    G^+(d,\theta), &  0 \le \theta < \pi/2; \\[4pt]
    G^-(d,\theta), & \pi/2 \le \theta \le \pi.
    \end{cases}
\end{equation}
which is $C^\infty $ everywhere except for along the polar axis (where $r=0$) when $0\le z \le 1$. Along this line segment there is a singularity in the log terms and there is a source/sink at each endpoint of the line segment. These singularities are in the interior of $\Omega$ except for the expected singularity at $\vx=\vx_0$. Due to axisymmetry, we can write
$$z = d \cos \theta \ , \qquad
\al{q}= d \cos \theta-1 \ , \qquad 
\rho = \sqrt{d^2 - 2d \cos \theta +1 } \ .$$
In the far-field 
$$G^e(\vx;\bx_0) \sim
\frac{1}{4 \pi d} +\frac{4+\mathcal{H}}{8 \pi} \,
\frac{\cos \theta}{ d^2} + \bigoh(1/d^3) \ ,
$$
as $d \to \infty$ 
satisfying the far-field condition \eqref{eq:eqnG_exterior_c}. At the source $d=z=1$, 
$$
G^e(\vx;\bx_0) =  \frac{1}{2 \pi \rho} -
\frac{\mathcal{H}}{4 \pi } \log [ \rho+\al{q}]
- \frac{1}{4 \pi } \left [1 +\mathcal{H}\right ] +\frac{\mathcal{H}}{4 \pi} \log 2 + \bigoh(\rho) \ .
$$
which allows one to identify the regular part of the Green's function. We identify that the first two terms are the singular portion at the source,
$$G^e_{\text{sing}}(\vx;\vx_0) =  \frac{1}{2 \pi \rho} -
\frac{\mathcal{H}}{4 \pi } \log [ \rho+\al{q}], $$
and define $G^e(\vx;\bx_0)\equiv G^e_{\text{sing}}(\vx;\vx_0)+R^e(\vx;\vx_0)$,
then the last two terms yield the regular part,
\begin{equation}\label{eq:RegFull}
    R^e(\vx_0;\vx_0) = - \frac{1}{4 \pi } \left [1 +\mathcal{H}\right ] +\frac{\mathcal{H}}{4 \pi} \log 2 \ .
\end{equation}
On the positive polar axis $G^e(d,0)=G_+(d,0)$ is singular for $0\le d \le 1$. For $d>1$
$$G^e(d,0)=G_+(d,0) = -\frac{\mathcal{H}}{4\pi} \left [ \log(1-1/d) 
+\frac{1}{d}
\right ] + \frac{d+1}{4 \pi d(d-1)},\qquad d>1,$$
 and as $d\downarrow1$,  
$$\al{\partial_d}G^e(d,0) \sim -\frac 1 {2 \pi} (d-1)^{-2} \ .$$ 
which is negative for all $\mathcal{H}$. 
At $\theta=\pi$, 
$$G^e(d,\pi)=G_-(d,\pi) = \frac{\mathcal{\Hc}}{4\pi} \left [ \log(1+1/d) 
-\frac{1}{d}
\right ] + \frac{d-1}{4 \pi d(d+1)},\qquad d \ge 0,$$
and
$$\al{\partial_{d}}G^e(d,\pi) = \frac{(\Hc+1)+(\Hc+2)d-d^2}{4 \pi d^2(d+1)}. $$
The south pole of $\partial \Omega$ must be at a value $d_*$ where the normal derivative vanishes, that is where $\al{\partial_{d}}G^e(d_*,\pi)=0$. This yields 
\begin{equation}\label{eqn:dstar}
    d_* = \frac{(\Hc+2) \pm \sqrt{\Hc^2+8\Hc+8} } {2} \ ,
\end{equation}
which has two real roots for $\Hc>\Hc_* \equiv -4 + 2 \sqrt{2} \approx -1.17157$, a double root at this value and no solutions for smaller values of $\Hc$.

\subsubsection{An implicit solution for the bounding domain}
In this section we will use the fact that $G^e(d,\theta)$ is harmonic to derive an exact implicit solution of the ODE system \eqref{Pode} for $P(\theta)$.  First note that for a harmonic function $\Delta G^e=0$ (except for $0 \le d \le 1$ on the north polar axis) that
$$\Delta  G^e \equiv \frac{1}{d^2} \frac {\partial}{\partial d } \Big(d^2 \frac{\partial }{\partial d} G^e\Big)  +  
\frac{1}{d^2 \sin \theta}  \frac {\partial}{\partial \theta } \Big(\sin \theta \frac{\partial}{\partial\theta} \,G^ e\Big) =0 .
$$
This can be rewritten as 
$$ \frac {\partial}{\partial d } \Big(d^2 \sin \theta \, \frac{\partial }{\partial d}G^e\Big) +  \frac {\partial}{\partial \theta } \Big(\sin \theta \,\frac{\partial }{\partial \theta}G^e\Big) =0 \ . $$
Now, rewriting the equation \eqref{Pode} for $P(\theta)$, we have that
$$  \sin\theta \, \frac{\partial}{\partial {\theta}}G^e(P,\theta) \, \textrm{d}P - P^2 \sin\theta\, \frac{\partial}{\partial d} G^e(P,\theta) \, \textrm{d} \theta  =0. $$
The differential equation  $M(P,\theta) \, \textrm{d}P +N(P,\theta) \, \textrm{d} \theta =0$ is exact if
$\al{\partial_{\theta} M = \partial_PN}$ which is exactly the condition derived above. Integrating the equation yields the implicit solution
$$ \frac{ \Hc }{4 \pi} \rho  
+\frac{1- P \cos \theta}{2 \pi \rho}
- \Hc \frac{P}{4 \pi} + (1+\Hc) \frac{\cos \theta}{4 \pi} = C .$$
where
$$\rho = \sqrt{P^2 - 2P \cos \theta +1}  = \sqrt{(1-P\cos \theta)^2 +(P \sin \theta)^2}, $$
is the distance to the source at the north pole.  If we approach the north pole  (i.e.~ $\rho \to 0$) with a behavior prescribed by \eqref{BCnorth}  (that is $P \sim \frac{1+\Hc}{2} \theta^2$ as $\theta$ decreases to zero) we find that $C=\frac{1}{4 \pi}$. The implicit definition of the surface can now be written as 
\begin{equation}\label{eqn:Implicit}
    F(P,\theta) =
\Hc \rho +2 \frac{1- P \cos \theta}{\rho}
- \Hc {P} + (1+\Hc) {\cos \theta} -1 =0 \ .
\end{equation}
To find the south pole, expand $F(d_*,\pi-\eps)=0$ as $\eps \to 0$,
$$ F(d_*,\pi-\eps) \sim 
\frac{(\Hc+1)+(\Hc+2) d_* -d_*^2 }{2(d_*+1)} \eps^2 + \bigoh(\eps^4)  \ .$$
This yields the same condition on $d_*$ stated in \eqref{eqn:dstar}.

We now summarize our process for generating a solution $G^e(\bx;\bx_0)$ for $\bx_0= (0,0,1)$ of \eqref{eq:eqnG_exterior}. First we select a mean curvature $\Hc>\Hc_{*} \equiv-4+ 2\sqrt{2}$ and calculate the value $d_*$ from \eqref{eqn:dstar}. The exact Green's function $G^e(\bx;\bx_0)$ is given by \eqref{eqn:GreenFull} and its regular part $R^e(\bx_0;\bx_0)$ by \eqref{eq:RegFull}.
    For each $\theta\in[0,\pi]$, we solve the implicit equation $F(P,\theta) =0$ given in \eqref{eqn:Implicit} for $P(\theta)$. The surface is then reconstructed from \eqref{eq:spherical_coords} for $d = P(\theta)$.

In Fig.~\ref{fig:constructed} we demonstrate our process by constructing domains for a range of $\Hc$ values. The solution of the implicit equation \eqref{eqn:Implicit} is shown in Fig.~\ref{fig:constructed_b} which describes a family of \lq\lq tear-drop\rq\rq\ shaped domains. For each case, we use the constructed domain in our boundary integral method and evaluate the regular part $R^e(\bx_0;\bx_0)$. In Fig.~\ref{fig:constructed_a} we show the constructed domain and a favorable comparison with the known value \eqref{eq:RegFull}.

We remark that the choice of source placement $\bx_0= (0,0,1)$ renders the solution axi-symmetric and hence $\Dc(\bx_0)=0$. Therefore the weaker singularity is not present in this example. This method can potentially be expanded to generate a larger class of domains by constructing more complex Green's functions from additional point sources. Despite some effort, we were not able to expand this methodology to generate domains without an axi-symmetric profile.

\begin{figure}[htbp]
    \centering
    \subfigure[The constructed domains and $R^e(\bx_0;\bx_0)$.]{\includegraphics[width = 0.45\textwidth]{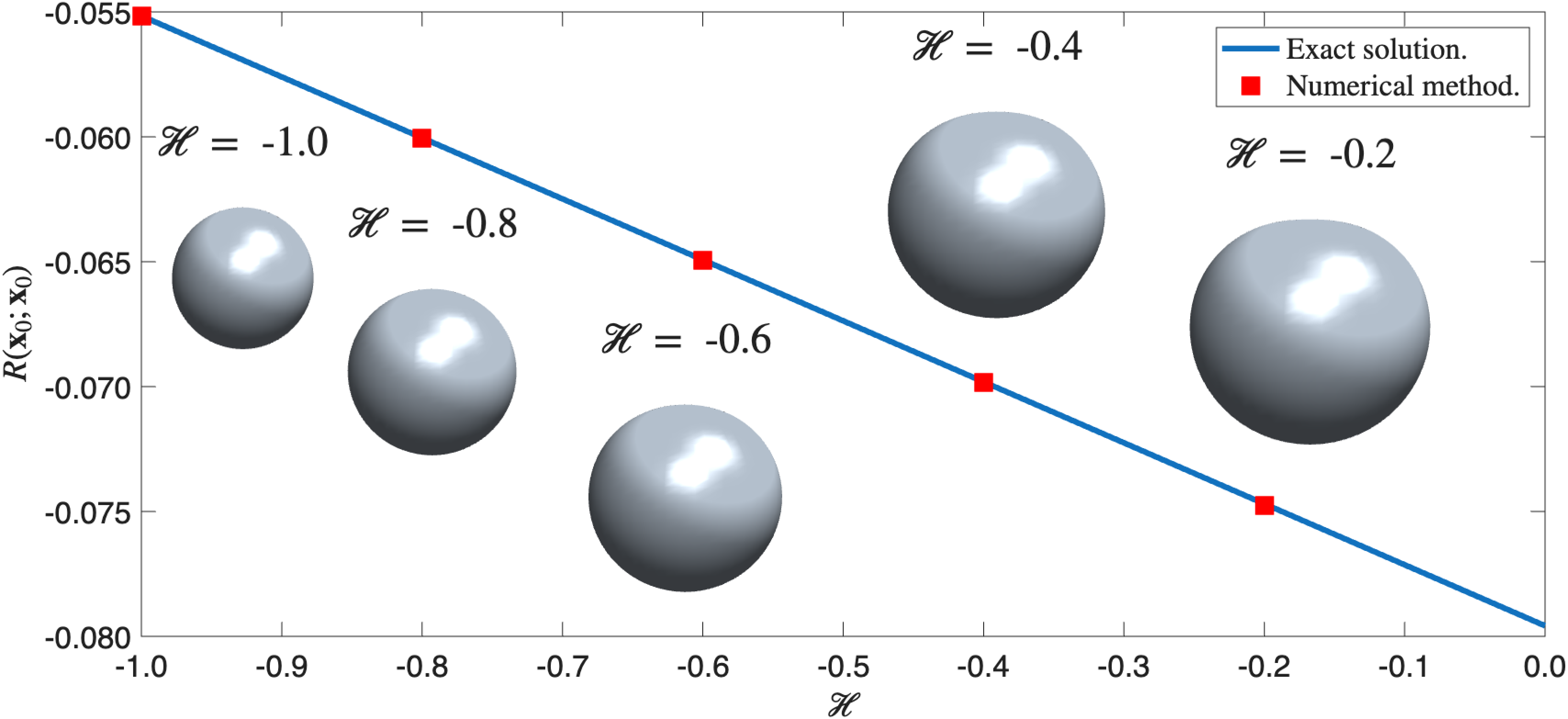}\label{fig:constructed_a}}\qquad
    \subfigure[The function $P(\theta)$.]{\includegraphics[width = 0.44\textwidth]{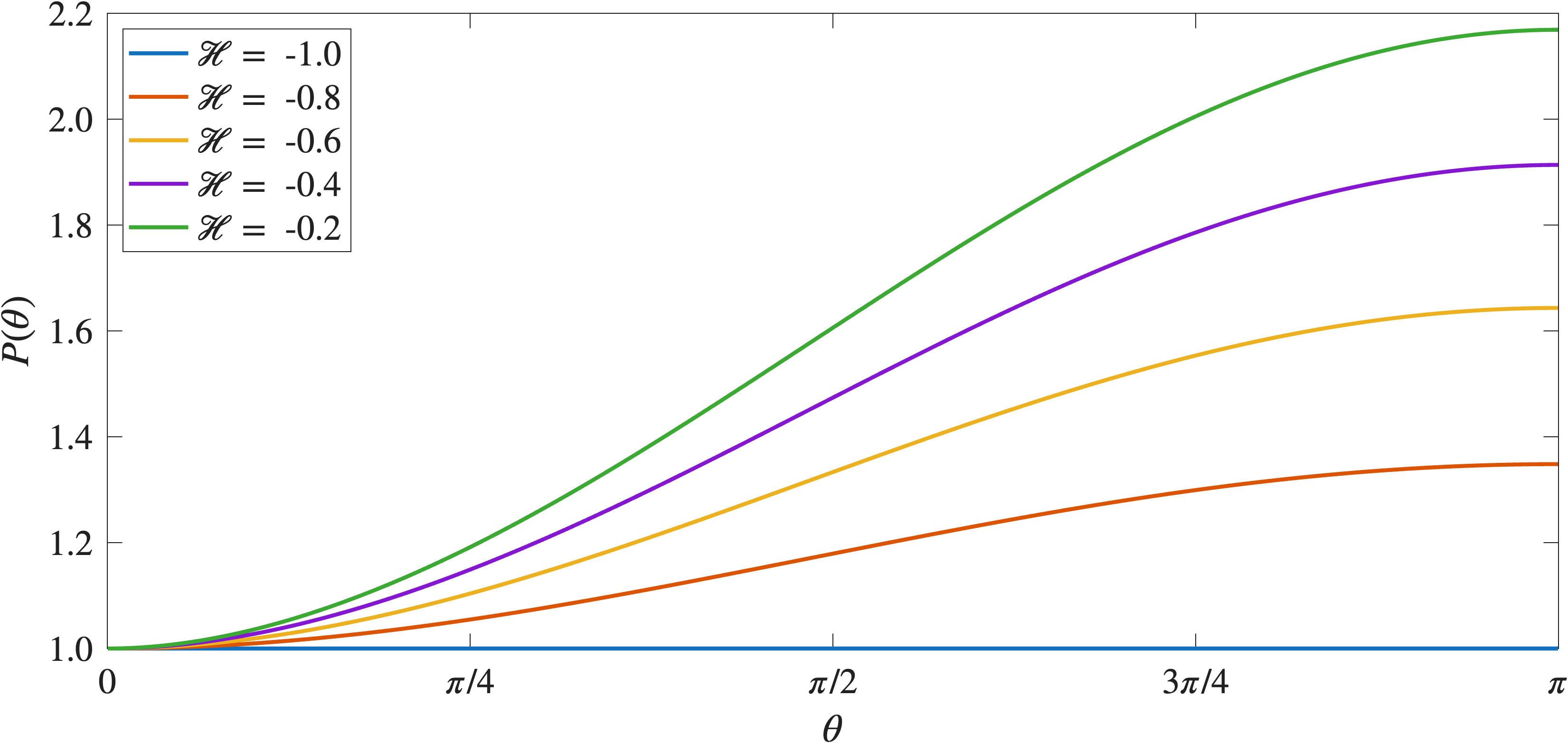}\label{fig:constructed_b}}
    \caption{Comparison of the constructed solutions with the numerical method. For values $\Hc\in\{-1.0,-0.8,-0.6,-0.4,-0.2\}$ we construct the domain $P(\theta)$ from numerical solution of \eqref{Pode}. The exact regular part $R^e(\bx_0;\bx_0)$ from \eqref{eq:RegFull} agrees closely with our numerical method (red squares). \label{fig:constructed}}
\end{figure}

\subsection{Series solution in prolate spheroid harmonics}

In this section, we develop a solution to the exterior surface Green's function \eqref{eq:eqnG_exterior} for the prolate spherical geometry defined by
\begin{equation}
\partial\Omega = \Big\{ (x,y,z)\in\mathbb{R}^3 \ \Big| \ \frac{x^2 + y^2}{a^2} + \frac{z^2}{b^2} = 1 \ \Big\} \, ,
\end{equation}
where $b>a>0$. The solution of \eqref{eq:eqnG_exterior} is separable in the prolate spheroid coordinate system
\begin{equation}\label{eqn:Prolate_coords}
    x = f \sqrt{(\xi^2-1) (1 - \al{q}^2)} \cos\phi, \qquad
    y = f \sqrt{(\xi^2-1) (1 - \al{q}^2)} \sin\phi, \qquad
    z = f \, \xi \, \al{q}, 
\end{equation}
where $f = \sqrt{b^2-a^2}$ is half the interfocal distance. Here $\xi\in[1,\infty)$ is the radial variable, $\al{q} \in[-1,1]$ is the elevation coordinate and $\phi \in [0,2\pi)$ is the azimuthal variable. The surface of the prolate spheroid is defined by the surface $\xi_0 = b/f$ and $\xi = 1$ is the line connecting the foci. The task is to solve for the Green's function and then peel off the triple singularity structure \eqref{GSing_surf} to identify the regular part $R^e(\bx_0;\bx_0)$. We begin by recalling that the harmonic functions in prolate coordinates are
\[
\{Q_n^m(\xi),P_n^m(\xi) \}\times \{P_n^m(\al{q})\cos m \phi,P_n^m(\al{q})\sin m\phi \},
\]
for $n=0,1,\ldots $ and $m = 0,1,\ldots,n$. Here $P_n^m$ and $Q_n^m$ are the associated Legendre functions of the first and second kind respectively \cite{Prolate1,prolate2}. The functions $P_n^m(\xi)$ are unbounded as $\xi\to\infty$ while $Q_n^m(\xi)$ are bounded as $\xi\to\infty$. When restricted to $\xi= \xi_0$, the surface prolate spherical harmonics are the functions
\[
C_n^m(\al{q},\phi) = P_n^m(\al{q}) \cos m \phi, \qquad S_n^m(\al{q},\phi) = P_n^m(\al{q}) \sin m \phi.
\]
 To establish the orthogonality structure, we introduce the weighted surface inner product
\begin{equation}\label{eqn:orthogo}
\braket{u,v}_w = \int_{\partial\Omega} u(\al{q},\phi)v(\al{q},\phi)\, w\, \mathrm{d}A,  \qquad w = \frac{1}{f^2\sqrt{(\xi^2-\al{q}^2)(\xi^2-1)}},
\end{equation}
for $\mathrm{d}A = w^{-1}\mathrm{d}\al{q} \,  \mathrm{d}\phi$. The orthogonality relationships for the surface prolate harmonics are 
\bsub\label{eqn:ortho_main}
\begin{equation}
\braket{S_n^m,S_N^M}_w = \gamma_{m,n} \delta_{n,N} \delta_{m,M}, \qquad \braket{S_n^m,C_N^M}_w = 0, \qquad \braket{C_n^m,C_N^M}_w = \gamma_{m,n} \delta_{n,N} \delta_{m,M},
\end{equation}
where the normalization constants are given by
\begin{equation}\label{eqn:gamma_mn}
\gamma_{m,n} = \frac{2\pi(n+m)!}{(2n+1)(n-m)!} \times\left\{ \begin{array}{rl} 2, & m=0; \\[4pt]1, & m\neq0.\end{array} \right.
\end{equation}
\esub

For the exterior problem requiring a finite solution as $\xi\to\infty$, the general solution of Laplace's equation with a source at coordinates $(\xi_0,\al{q}_0,\phi_0)$, is
\begin{equation}\label{eqn:G_surf_full}
G^e(\bx;\bx_0) = \sum_{n=0}^{\infty}\sum_{m=0}^{n} B_{m,n} Q_n^m(\xi) P_n^m(\al{q})\cos m(\phi-\phi_0),
\end{equation}
with constants $B_{m,n}$. 

\al{In the prolate coordinate system, the Dirac source has the representation
\begin{equation}
    \delta(\bx-\bx_0) = \frac{1}{|J|}\delta(q-q_0)\delta(\phi-\phi_0) = w(\xi_0,q)\delta(q-q_0)\delta(\phi-\phi_0),
\end{equation}
where  $|J|$ is the determinant of the coordinate transformation given in \eqref{eqn:Prolate_coords} which can be written in terms of
the scale factor $w$ is given in \eqref{eqn:orthogo}. To  apply the boundary condition, we note that the normal derivative on $\xi = \xi_0$ transforms to
\begin{equation}
\partial_nG^e = \frac{a^2}{f} w(\xi_0,\al{q}) \partial_{\xi} \, G^e  
\end{equation}
so the Neumann condition, $\partial_nG^e = -\delta(\bx-\bx_0) $ becomes
\begin{equation}
\frac{a^2}{f} w(\xi_0, q) \partial_{\xi} \, G^e 
= -  w(\xi_0,q)\delta(q-q_0)\delta(\phi-\phi_0),
\end{equation}
applied on the surface where $\xi=\xi_0$.
Substituting the expansion \eqref{eqn:G_surf_full} and
applying the orthogonality relationships  \eqref{eqn:ortho_main} yields
\[
B_{m,n} = - \frac{f}{a^2}\frac{P_n^{m}(\al{q}_0)}{\gamma_{m,n} Q_n^{m'}(\xi_0)} = \frac{1}{f(1-\xi_0^2)} \frac{P_n^{m}(\al{q}_0)}{\gamma_{m,n} Q_n^{m'}(\xi_0)}.
\]
}
\al{
Substituting the constants $B_{m,n}$ into the general solution \eqref{eqn:G_surf_full}, we obtain the series solution
\begin{equation}\label{eq:mainGreensExp}
    G^e(\bx;\bx_0) = \frac{1}{f(1-\xi_0^2)}\sum_{n=0}^{\infty}\sum_{m=0}^{n} \frac{P_n^m(\al{q}_0)}{\gamma_{m,n} Q_n^{m'}(\xi_0)} Q_n^m(\xi) P_n^m(\al{q})\cos m(\phi-\phi_0).
\end{equation}}
We now turn to the issue of extracting the singularity structure by considering the solution of $G^e(\bx;\bx_0)$. \al{As is common with spectral expansions of Dirac sources, the convergence properties of \eqref{eq:mainGreensExp} are terrible for any $\bx\in\partial\Omega$ and deteriorate further as $\bx$ approaches $\bx_0$. To ameliorate these issues, we evaluate \eqref{eq:mainGreensExp} off the surface on a ring of points $r_c = |\bx-\bx_0|$ in the tangent plane at $\bx_0$ (see Fig.~\ref{fig:LocalSchematic})}. On this ring, we have that
\begin{equation}\label{eqn:GreensLocal}
    G^e(\bx;\bx_0) = \frac{1}{2\pi r_c} - \frac{\mathcal{H}(\bx_0)}{4\pi} \log r_c
    + \frac {\Dc(\bx_0)} {8\pi} \cos (2 \varphi) +
    R^e(\bx;\bx_0)+
    o(1), \quad \mbox{as} \quad \bx\to\bx_0,
\end{equation}
where $\varphi$ is an azimuthal measure in the tangent plane. To isolate $R^e(\bx_0;\bx_0)$, we calculate the average value of the function
\[
W(\bx;\bx_0) = G^e(\bx;\bx_0)-\frac{1}{2\pi r_c} + \frac{\Hc(\bx_0)}{4\pi} \log r_c,
\]
along a ring of points in the tangent plane at $\bx_0$. To improve the convergence properties of the series, the series representation of the leading monopole singularity is subtracted from the expansion. The relevant series for the monopole expansion \cite{Prolate1} around a source $\bx_0 = (\xi_0,\al{q}_0,\phi_0)$ is
\bsub
\begin{equation}\label{eq:key_identities_b} 
   \frac{1}{|\bx-\bx_0|} = \frac{1}{f} \sum_{n=0}^{\infty}\sum_{m=0}^n   H_{m,n} P_n^m(\al{q}_0)P_n^m(\xi_0)Q_n^m(\xi) P_n^m(\al{q})\cos m(\phi-\phi_0),
\end{equation}
where the constants $H_{m,n}$ are given by
\begin{equation}
    H_{m,n} = \eps_m (-1)^m(2n+1) \left[ \frac{(n-m)!}{(n+m)!}\right]^2, \qquad \eps_m = \left\{ \begin{array}{rl} 1, & m=0; \\ 2, & m>0.\end{array} \right.
\end{equation}
\esub
Adding and subtracting a monopole from the solution \eqref{eq:mainGreensExp} yields 
\begin{align*}
G^e(\bx;\bx_0) &= \frac{1}{2\pi|\bx-\bx_0|} - \frac{1}{2\pi|\bx-\bx_0|}+\sum_{n=0}^{\infty}\sum_{m=0}^{n} B_{m,n} Q_n^m(\xi) P_n^m(\al{q})\cos m(\phi-\phi_0)\\[5pt]
&= \frac{1}{2\pi|\bx-\bx_0|} + \sum_{n=0}^{\infty}\sum_{m=0}^{n} \left[B_{m,n} - \frac{H_{m,n} P_n^{m}(\al{q}_0)P_n^m(\xi_0)}{2\pi f} \right] Q_n^m(\xi) P_n^m(\al{q})\cos m(\phi-\phi_0),
\end{align*}
The series for the function $W(\bx;\bx_0)$ is now given by
\begin{equation}\label{eqn:seriesS}
    W(\bx;\bx_0) = \frac{\Hc(\bx_0)}{4\pi}\log r_c +  \sum_{n=0}^{\infty}\sum_{m=0}^{n} \left[B_{m,n} - \frac{H_{m,n} P_n^{m}(\al{q}_0)P_n^m(\xi_0)}{2\pi f} \right] Q_n^m(\xi) P_n^m(\al{q})\cos m(\phi-\phi_0).
\end{equation}
We evaluate $W(\bx;\bx_0)$ on $N_{\text{ring}}=24$ equally spaced points $\{\bx_i\}_{i=1}^{N_{\text{ring}}}$ on a ring of radius $r_c$ in the tangent plane. As $r_c\to0$, we recover the predicted biquadratic (quadratic in $r_c^2$) behavior of $\bar{W} = \frac{1}{N_{\text{ring}}}\sum_{i=1}^{N_{\text{ring}}}W(\bx_i;\bx_0)$ and use fitting to extrapolate the value of $R^e(\bx_0;\bx_0)$ - see Fig.~\ref{fig:RingLocal_a}. Along the ring $|\bx-\bx_0| = r_c$, we directly verify in Fig.~\ref{fig:RingLocal_b} the presence of the third weaker singularity predicted in \eqref{eqn:GreensLocal}. For the prolate spheroid, the mean and differences of the surface curvature at the point $\bx_0$ are given by
\[
\Hc(\bx_0) = \frac{\xi_0(1-2\xi_0^2 + \al{q}_0^2)}{2f(\xi_0^2-1)^{\frac12}(\xi_0^2 - \al{q}_0^2)^{\frac32}}, \qquad \Dc(\bx_0) = \frac{\xi_0 \big( 1 - \al{q}_0^2 \big)}{2f (\xi_0^2 - 1)^{\frac12} (\xi_0^2 - \al{q}_0^2)^{\frac32}}.
\]

\begin{figure}
    \centering
    \subfigure[$\bar{W}(\bx;\bx_0)$ as $r_c\to0$.]{\includegraphics[width=0.45\textwidth]{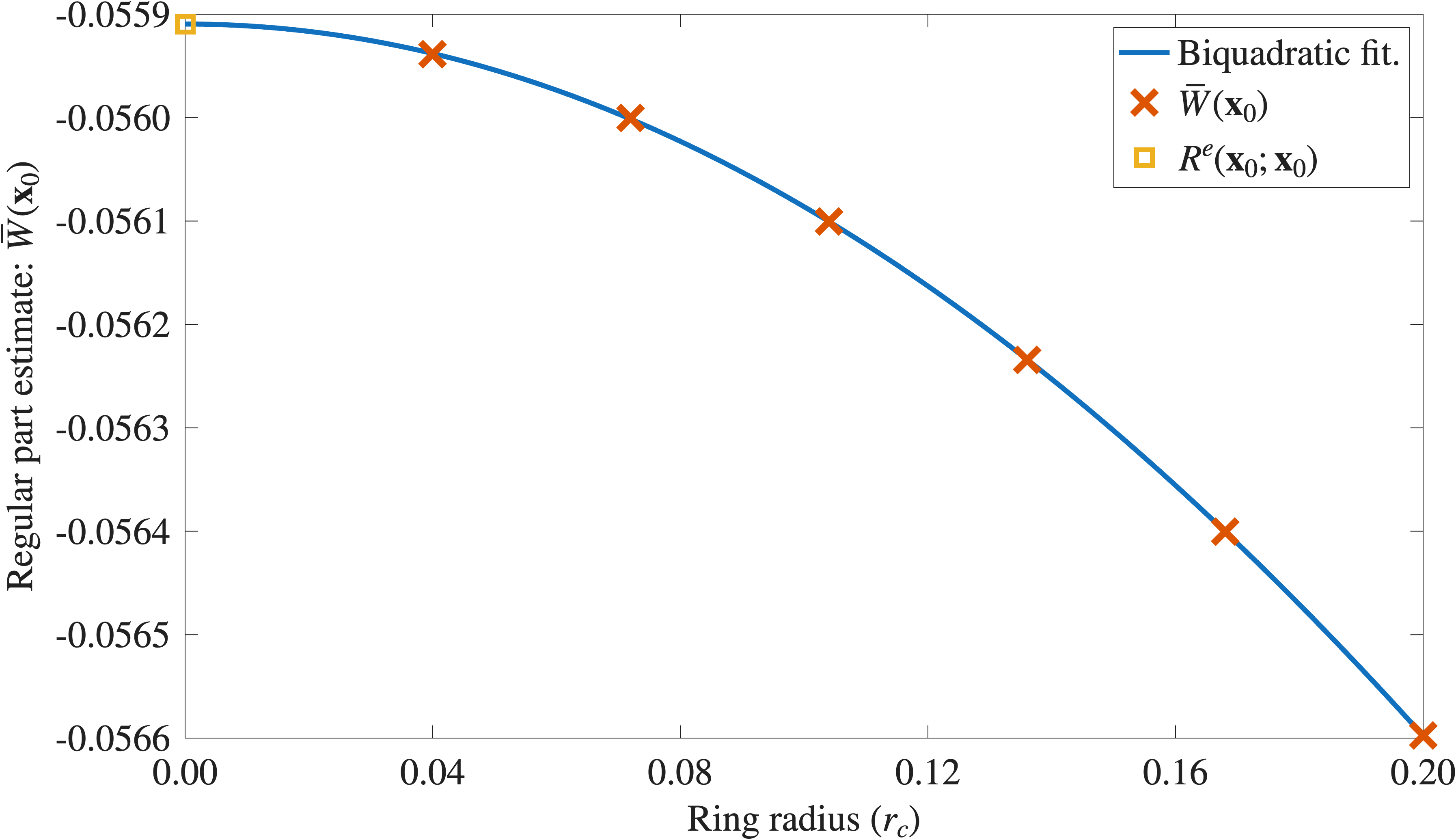}\label{fig:RingLocal_a}}\qquad
    \subfigure[$W(\bx;\bx_0) - \bar{W}$ on the ring $r_c = |\bx-\bx_0| =0.01$.]{\includegraphics[width=0.45\textwidth]{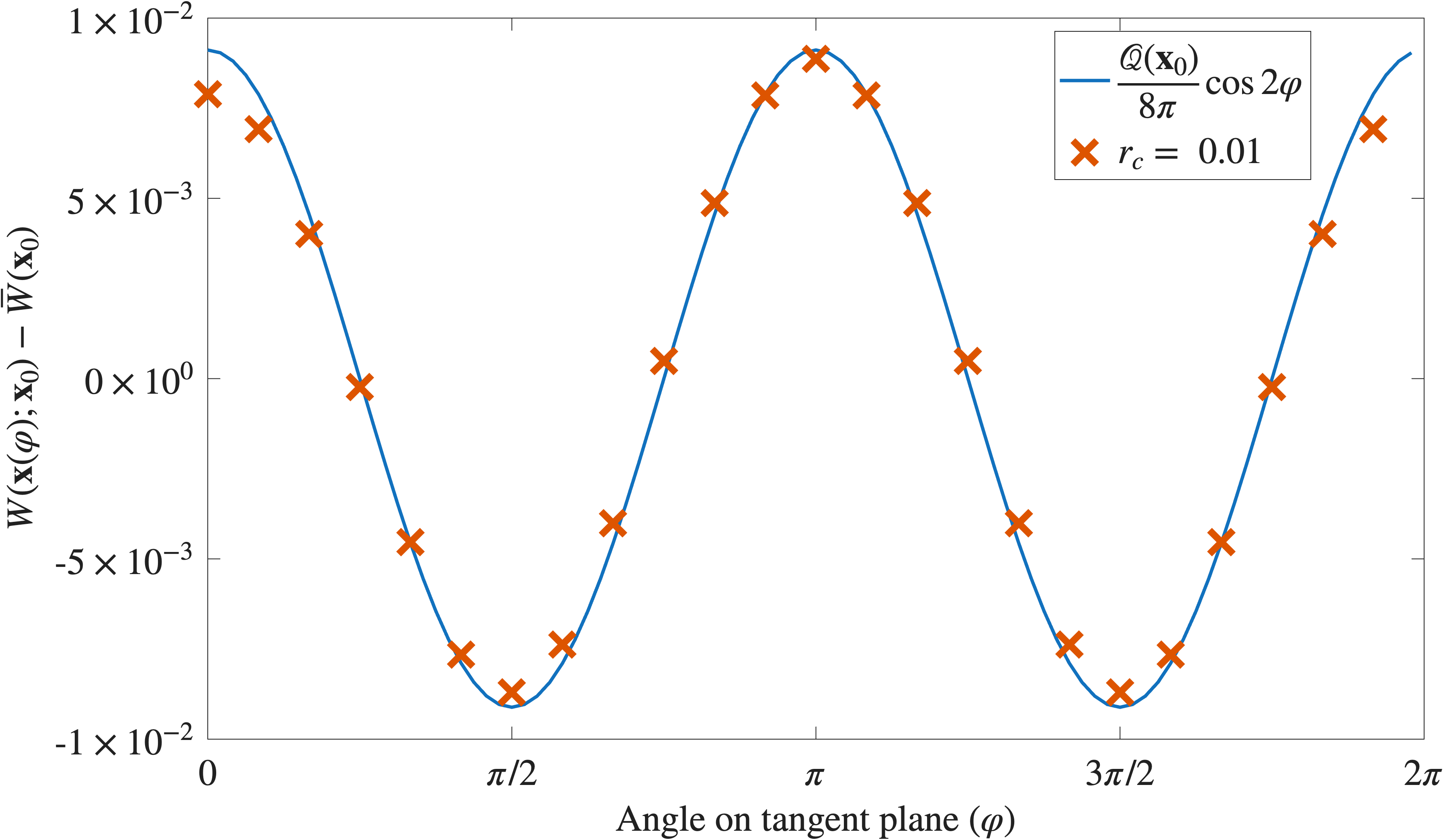}\label{fig:RingLocal_b}}
    \caption{Evaluation of the series $W(\bx;\bx_0)$ on the tangent plane at $\al{q}_0 = 0.7$ with $N = 2000$ terms. Left: the approximation of the average $\bar{W}$ as the ring radius $r_c$ shrinks to zero. We use a fitted curve (solid blue) to extrapolate the intercept as $R^e(\bx_0;\bx_0)$ (yellow square). Right: The quantity $W(\bx(\varphi);\bx_0) - \bar{W}$ at points (red squares) along a ring in the tangent plane centered at $\bx_0$ together with the predicted asymptotic behavior (solid blue) in \eqref{eqn:GreensLocal}. \label{fig:RingLocal}}
\end{figure}

In our numerical evaluation of the series \eqref{eqn:seriesS}, we calculate only ratios and logarithmic derivatives of the Legendre functions $P_n^m$ and $Q_n^m$ rather than the functions themselves. This avoids evaluating the factorial terms in the normalization factors $\gamma_{m,n}$ which leads to overflow at relatively moderate $n$ and $m$ values in standard double precision. Recursion relations adapted for this purpose were provided in \cite{prolate2} which we have implemented and provide reliable evaluation for thousands of terms. Our results are based on evaluations of $W(\bx;\bx_0)$ with truncation at $2000$ terms.

In Fig.~\ref{fig:ProlateError}, we display comparisons of estimates of $R^e(\bx_0;\bx_0)$ from the series solution approach and the boundary integral method. We take a prolate spheroid with parameters $a=1$, $ b= \sqrt{2}$ and a source point at $\bx_0 = (\xi_0,\al{q}_0,0)$ for a range of points $\al{q}_0 \in[-1,1]$ and $\xi_0 =b/\sqrt{b^2 - a^2} = \sqrt{2}$. We observe close agreement between the boundary integral method and the series solution. The relative error in the series approximation is around $10^{-4}$ which is limited by the number of terms in the truncation of the series and our extrapolation method. Consequently, the utility of this series solution is that it provides a non-trivial validation of both the numerical method and the derived singular behavior \eqref{eqn:GreensLocal} of the Green's function near the source. As we demonstrated in Fig.~\ref{eq:sphere_error}, our numerical method is capable of relative errors of $10^{-12}$ for the bulk problem and $10^{-9}$ for the surface problem.

\begin{figure}[htbp]
    \centering
    \subfigure[Agreement in $R^e(\bx_0;\bx_0)$.]{\includegraphics[width = 0.475\textwidth]{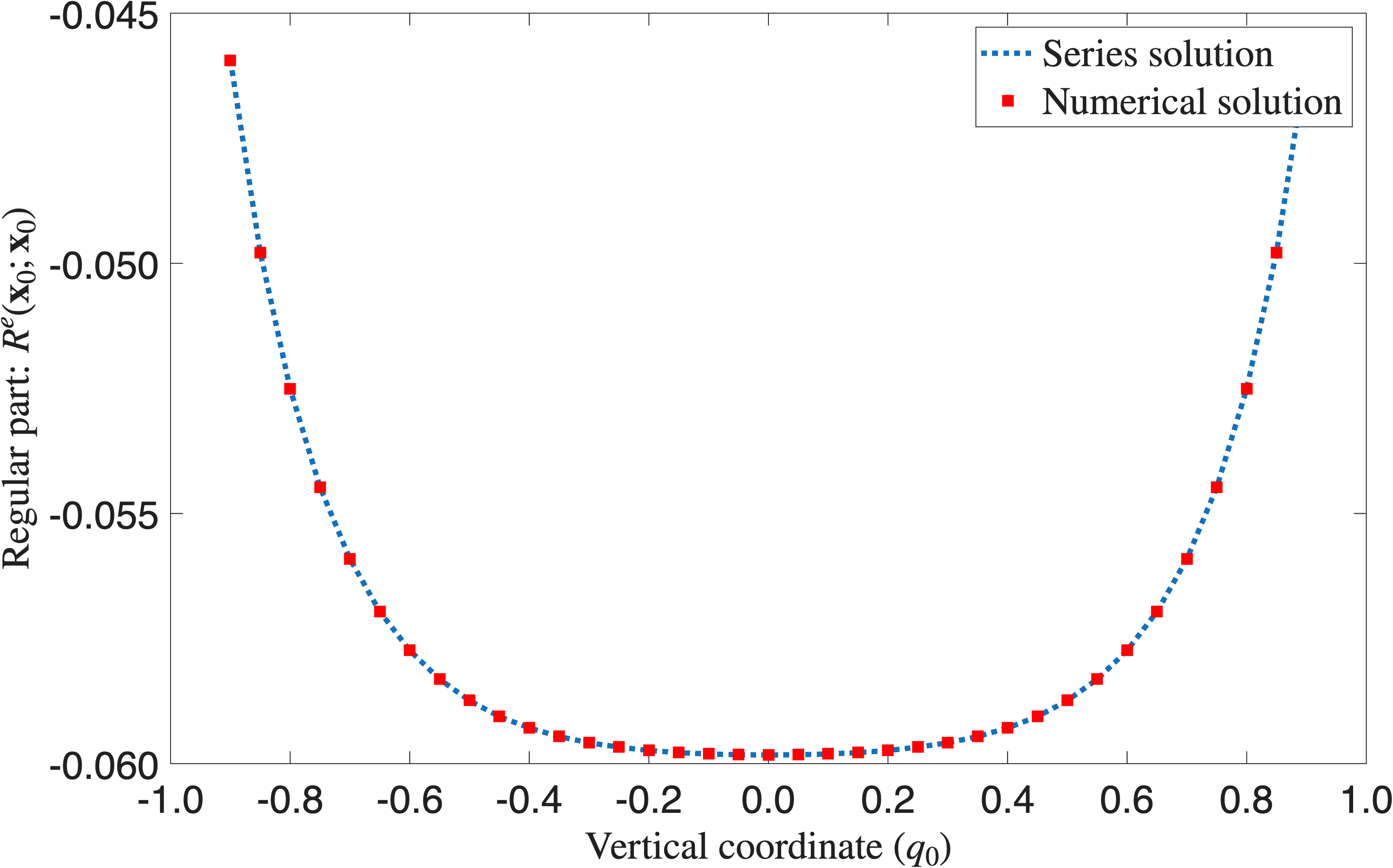}\label{fig:ProlateError_a}}\qquad
    \subfigure[Relative error in series approximation.]{\includegraphics[width = 0.45\textwidth]{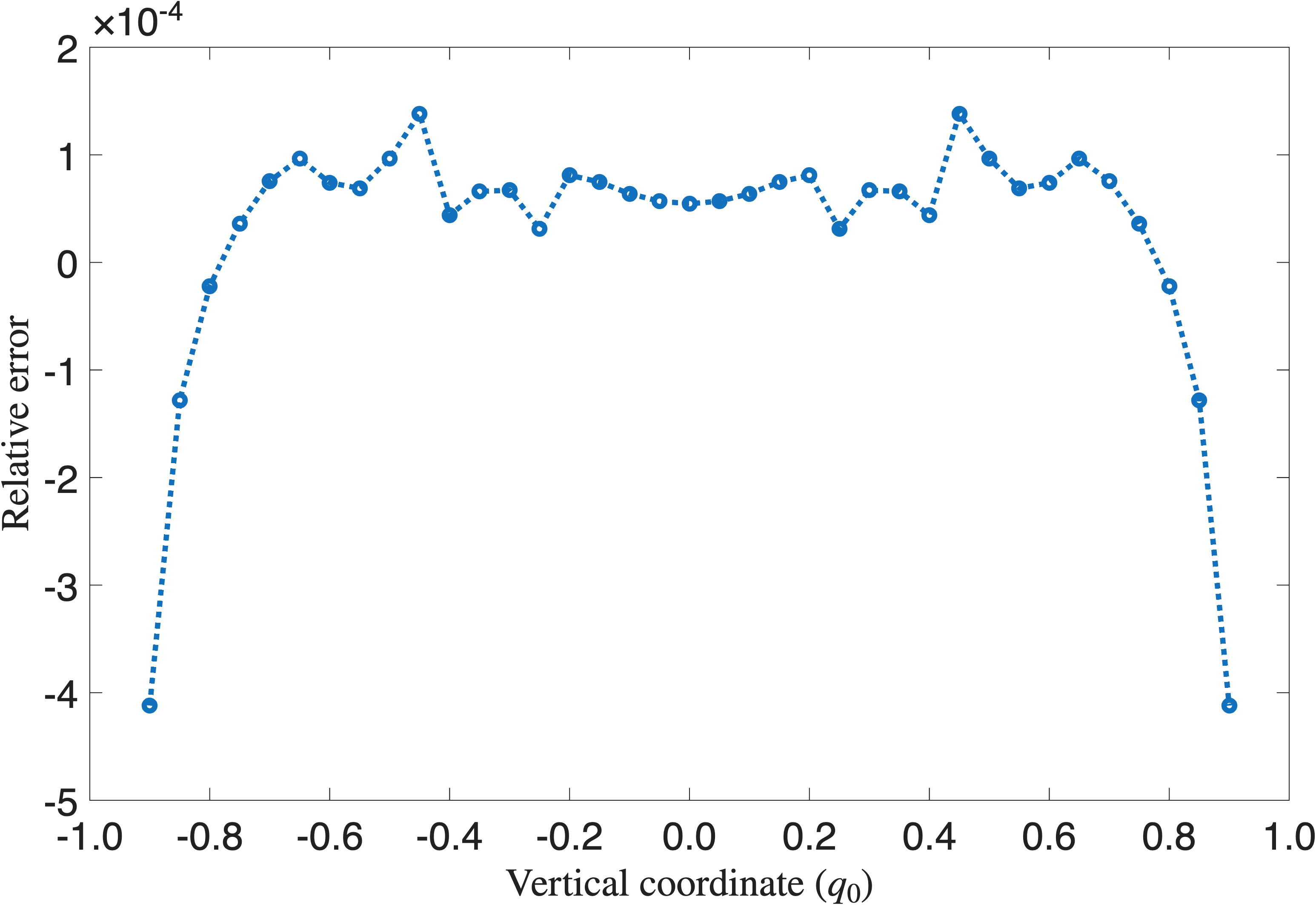}\label{fig:ProlateError_b}}
    \caption{Comparison between the regular part of the Green's function for the prolate spheroid case $a= 1$ and $b= \sqrt{2}$ for a range of source point in $\al{q}_0\in[-1,1]$. Left: A comparison of $R^e(\bx_0;\bx_0)$ from the series approximation (blue dotted) and our numerical solution (red squares). Right: The relative error in the series approximation is roughly $10^{-4}$. \al{As discussed in the text, this error is due to the slow convergence of the series solution and our extrapolation approach, not the boundary integral method.}  \label{fig:ProlateError}}
\end{figure}

{\color{black}
\subsection{Convergence study on the prolate spheroid domain}
In this subsection, we assess the convergence order of the method presented in Section~\ref{Sec:Numerics} for computing $R^i(\bx_0;\bx_0)$ and $R^e(\bx_0;\bx_0)$ in the ellipsoid geometry shown in Fig.~\ref{fig:ellipsoid_reg}. We choose a generic $\bx_0\in\partial\Omega$ and compute $R^{i,e}(\bx_0;\bx_0)$ using a sequence of meshes with typical diameter $h_0$ and of $6^{th},8^{th}$ and $10^{th}$ orders. We use the solution on a fine $10^{th}$ order mesh as the benchmark to calibrate the accuracy at coarser resolutions. Fig.~\ref{fig:convergence} shows that both values converge rapidly with faster convergence exhibited by the higher order meshes. In Fig.~\ref{fig:convergence}, we show convergence results with and without the inclusion of the Duffy patches near $\bx_0$ as described in Section~\ref{sec:rhs}. Crucially, we observe that the special Duffy treatment of the singularity in $\partial_n Q^\pm$ is essential to achieve high order convergence. Our results show that quantities $R^{i,e}(\bx_0;\bx_0)$ are approximated to accuracies of roughly $\bigoh(10^{-10})$.

\begin{figure}[htbp]
   \subfigure[$R^i(\bx_0;\bx_0)$ for $\bx_0\in\partial\Omega$.]{\includegraphics[width=0.475\textwidth]{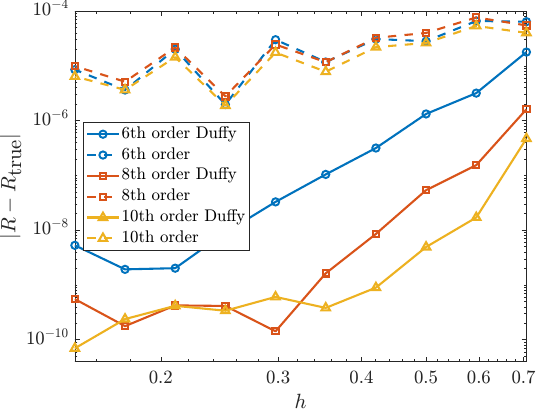}\label{fig:conva}}\qquad
    \subfigure[$R^e(\bx_0;\bx_0)$ for $\bx_0\in\partial\Omega$.]{\includegraphics[width=0.475\textwidth]{ell_conv_int.pdf}\label{fig:convb}}
   \caption{{\color{black}Estimated error in the interior (left) and exterior (right) regular parts of the Neumann function on a prolate spheroid with principal semi-axes $(a,b,c)=(1,1,\sqrt{2})$. The solid lines represent meshes with the singular Duffy patches near $\bx_0$ and the dashed ones indicate solutions generated without these special patches. \label{fig:convergence} }}
\end{figure}

}

\subsection{Applications to narrow capture problems}\label{sec:apps}

In this example, we demonstrate the ability of our method to inform on some open questions in narrow capture theory. For a bounded geometry $\Omega\subset\mathbb{R}^3$, the Narrow Capture Problem (NCP) 
\cite{SingerHolcman2006,TargetSearch2024,Gilbert_Cheviakov_2023} asks how \al{ the locations $\{\bx_j\}_{j=1}^N\in\Omega$ of $N$ small well separated traps should be arranged so as to minimize the average capture time $\tau$ of a diffusing molecule of diffusivity $D>0$?} For identical spherical traps of common radius $\varepsilon$, the quantity $\tau$ was shown to satisfy \cite{ChevWard2010} 
\begin{equation}\label{eqn:minP}
    \tau \sim \frac{|\Omega|}{D}\left[ \frac{1}{4\pi\varepsilon N} + \frac{p}{N^2}  \right]\, , \qquad p := p(\bx_1,\ldots,\bx_N) = \sum_{j=1}^N \Big[ R^i(\bx_j;\bx_j) + \sum_{ \substack{k=1 \\ j\neq k}}^N G^i(\bx_k; \bx_j)\Big],
\end{equation}
as $\varepsilon\to 0$. The optimizing trap locations $\{\bx_j\}_{j=1}^N$ therefore correspond to the minima of the discrete energy $p(\bx_1,\ldots,\bx_N)$. Minimization of similar discrete energies is a challenging computational problem on account of the multiplicity of critical points as $N$ increases. A number of optimization strategies combine local (e.g. Newton methods) and global search methods (e.g.~particle swarm optimization) \cite{TargetSearch2024,Gilbert_Cheviakov_2023,Sarafa2021}. In particular, it is essential to determine $p$ and its gradient to high precision so that the optimization algorithm can discern between configurations that have close values in the objective function. The recent survey article \cite{TargetSearch2024} highlighted the lack of reliable methods to compute the Neumann function as a major obstacle towards progress on this problem.

\begin{figure}[htbp]
    \centering
    \includegraphics[width=0.95\textwidth]{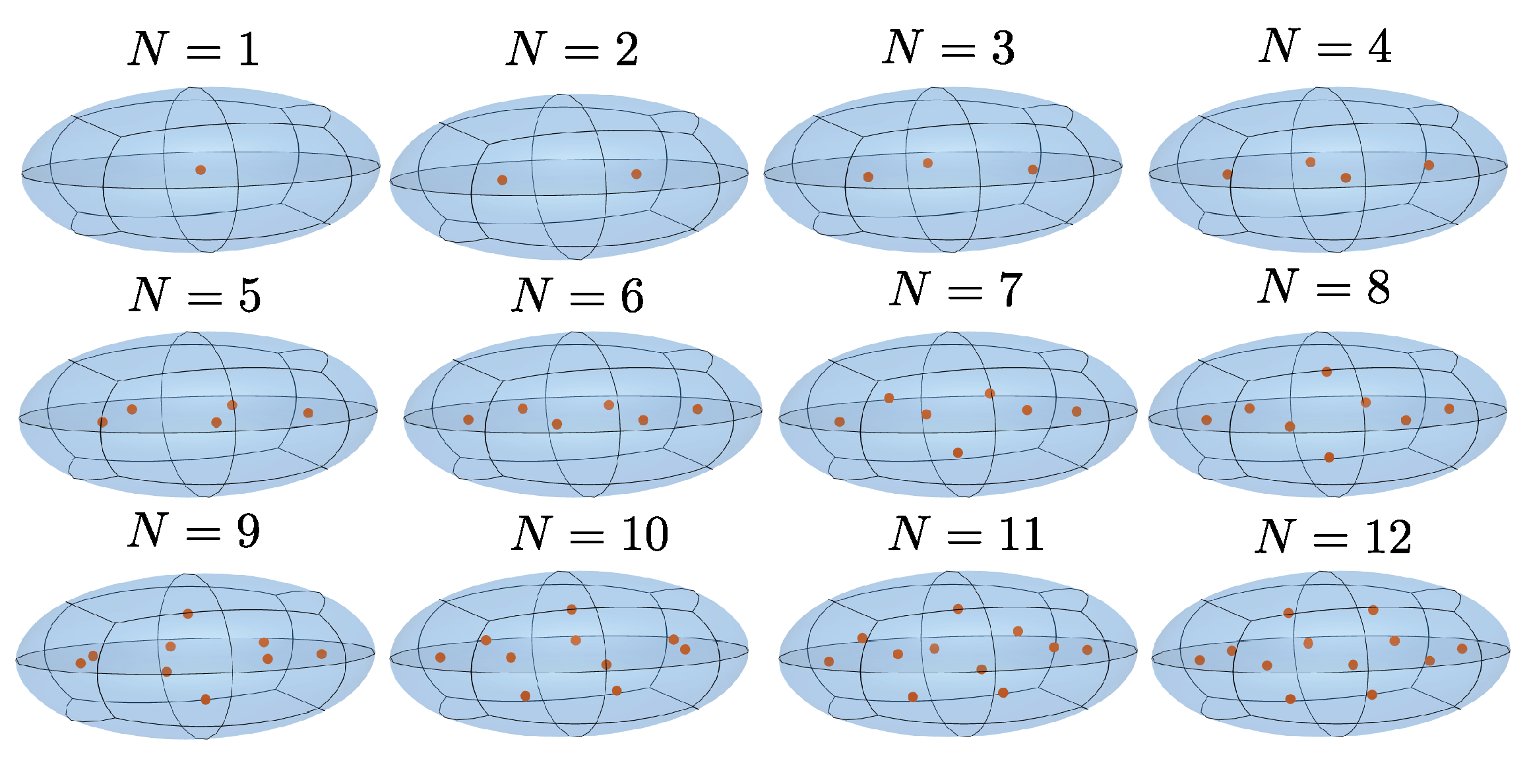}
    \caption{Locally optimal trap configurations for $N=1,\ldots,12$ in the ellipsoid \eqref{eqn:ellipsoidal}. For $1\leq N \leq 6$, the optimizing configurations are coplanar on $z=0$. \label{fig:EllipsoidOptimal}}
\end{figure}

In this section we demonstrate the effectiveness of our method at predicting optimizing configurations when coupled with established optimization routines. We present results on minimizers of \eqref{eqn:minP} obtained with the \textsc{Matlab} optimization routine {\tt fmincon} called with the {\tt interior-point} algorithm. A notable aspect of our boundary integral method is that it provides a high precision calculation of the gradient of the objective function without the need for differencing.

The first example is a demonstration for the ellipsoidal domain
\begin{equation}\label{eqn:ellipsoidal}
\frac{x^2}{a^2} + \frac{y^2}{b^2} + \frac{z^2}{c^2} = 1, \qquad (a,b,c) = \left(\frac32,1,\frac23\right),
\end{equation}
for $N=1,\ldots,12$ traps. We observe in Fig.~\ref{fig:EllipsoidOptimal} a geometric bifurcation where for $1\leq N \leq 6$, the minimizing configurations are coplanar on $z=0$ and not coplanar for $N\geq7$. A similar bifurcation has been observed in the two dimensional case of an ellipse where the minimizing set transitions from colinearity \cite{Sarafa2021,chakraborty2025fastintegralmethodsneumann,TargetSearch2024}.

\begin{figure}[htbp]
    \centering
    \includegraphics[width=0.95\textwidth]{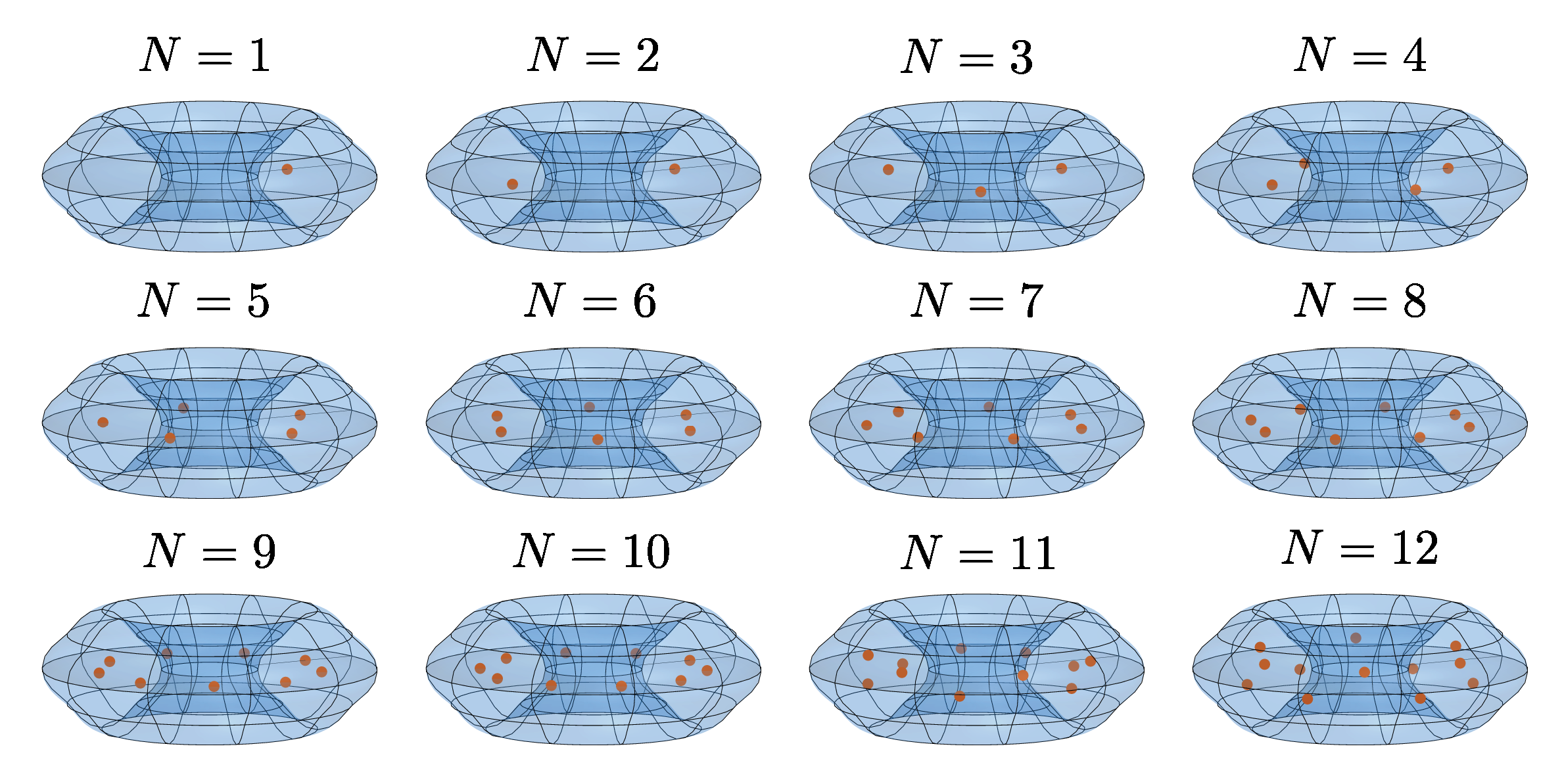}
    \caption{Minimizing locations of \eqref{eqn:minP} for $N=1,\ldots,12$ trap locations (red dots) in the torus defined in equation \eqref{eqn:torus_ex}. For $1\leq N\leq 10$, the optimizing configuration is coplanar with $z=0$. \label{fig:TorusOptimal}}
\end{figure}

In a second example, we compute some minimizing configurations for the toroidal domain specified by
\bsub\label{eqn:torus_ex}
\begin{align}
 \label{eqn:torus_ex_a}   x(\theta,\phi) &=  \, \big(r_{\text{maj}}+ \cos\theta \, (r_{\text{min}} + r_{\text{wave}}\cos n\theta )\big)\cos \phi,\\
  \label{eqn:torus_ex_b}  y(\theta,\phi) &=  \, \big(r_{\text{maj}}+ \cos\theta \, (r_{\text{min}} + r_{\text{wave}}\cos n\theta )\big) \sin \phi,\\
   \label{eqn:torus_ex_c} z(\theta,\phi) &=  \, \big(r_{\text{min}} + r_{\text{wave}} \cos n\theta \big) \sin \theta,
\end{align}
\esub
for $(\theta,\phi)\in(0,2\pi]^2$. We take the parameter values
\[
(r_{\text{maj}},r_{\text{min}},r_{\text{wave}},n) = (1,0.5,0.05,4),
\]
which represents a toroidal geometry with cross-sectional profile given by a near-circular disk. For the toroidal domain, we see in Fig.~\ref{fig:TorusOptimal} that for configurations $1\leq N\leq 10$, the optimizing configuration is coplanar with $z=0$ while for $N\geq 11$, the minimizing configuration undergoes a geometric bifurcation into the region $z\neq0$.

For both the ellipsoid and toroidal geometries, we show in Fig.~\ref{fig:ScalingP} \al{the calculated minimum energy $p$ as a function of $N$} for $N=1,\ldots,12$ trap configurations and again observe a function that decreases smoothly through this transition.

For a general domain $\Omega$, it would be desirable to derive an explicit formula for the minimizer of $p$ in the limit as $N\to\infty$ which can be verified with this method. This question relates to problems in Coulomb gases \cite{serfaty2024lecturescoulombrieszgases} and also appears in the concentration sets of nonlinear PDEs in  singularly perturbed limits \cite{ChenWard2011,Davila2012,delPino2005,WWK09,TXKTW2017}. Interestingly, despite the bifurcation from planar (red dots) to non-planar (yellow dots) trap configurations, the curve of minimizing energy appears to be smooth through this transition.

\begin{figure}[htbp]
    \centering
    \subfigure[Energy scaling for the ellipsoid domain.]{\includegraphics[width=0.45\textwidth]{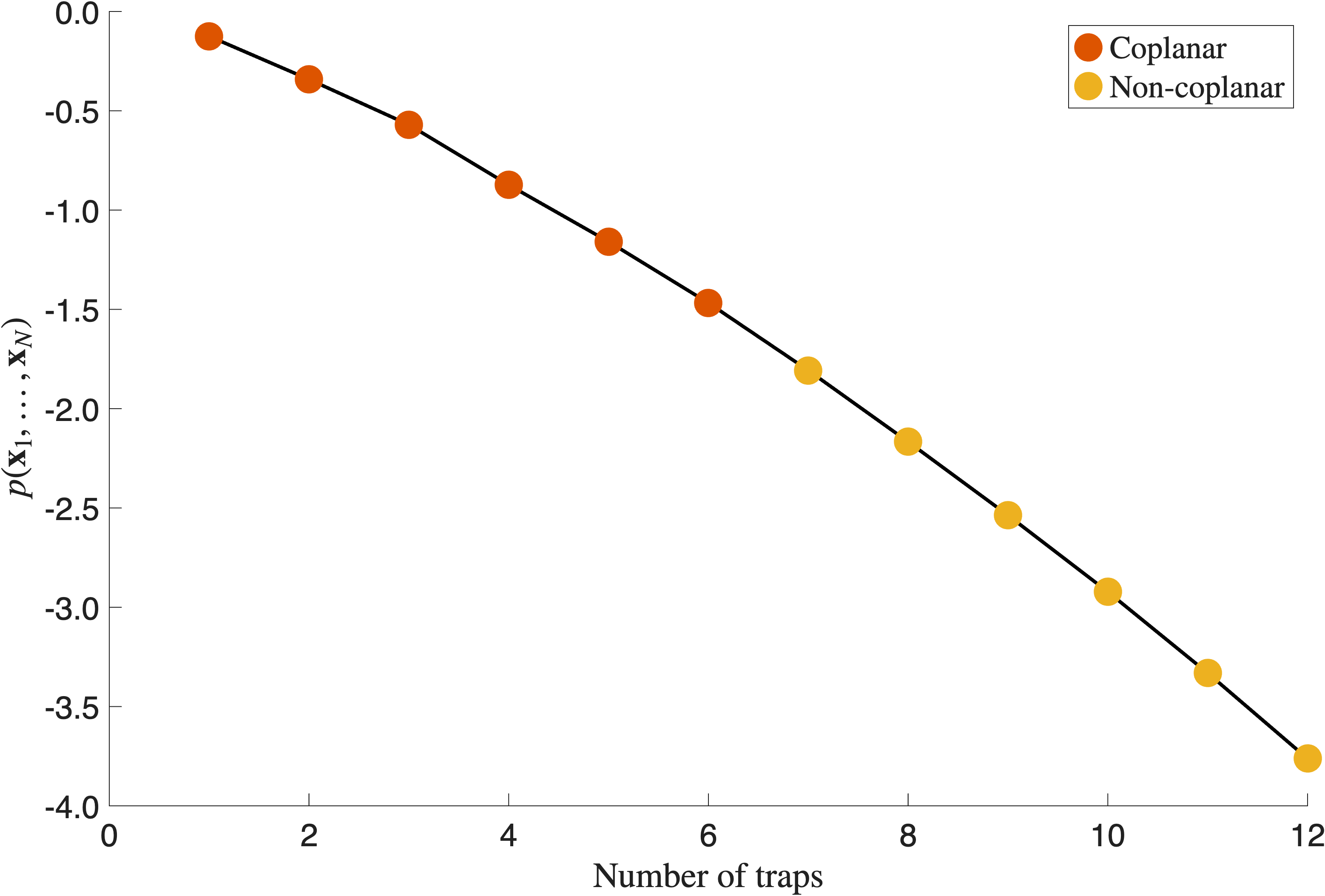} \label{fig:ScalingP_a}}
    \qquad
    \subfigure[Energy scaling for the toroidal domain.]{\includegraphics[width=0.45\textwidth]{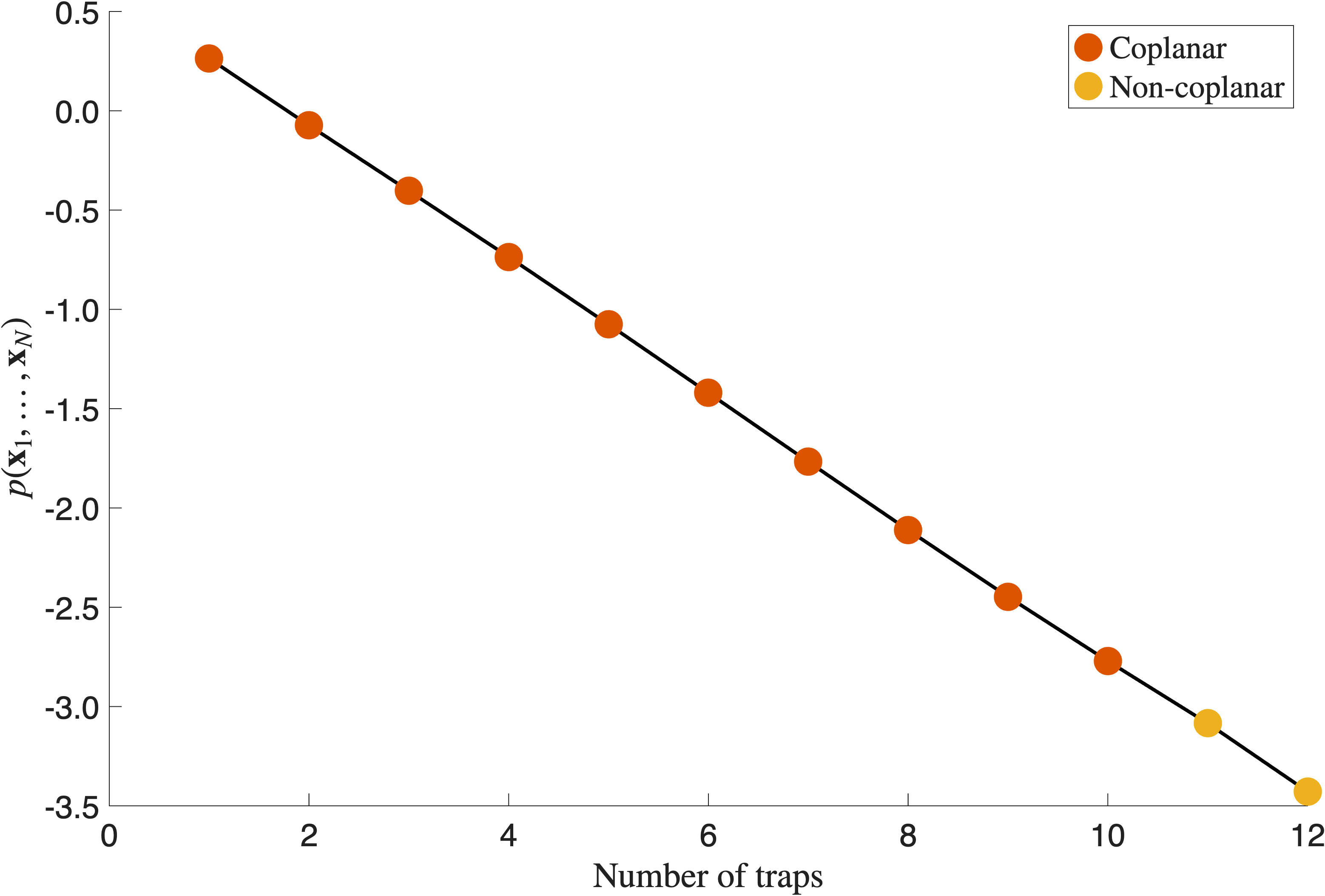} \label{fig:ScalingP_b}}
    \caption{Minimizing values of the discrete energy \eqref{eqn:minP} for $N= 1,\ldots,12$ for the ellipsoidal (left) domain \eqref{eqn:ellipsoidal} and the toroidal (right) domain \eqref{eqn:torus_ex}. The geometric bifurcation from planar (red dots) to non planar (yellow dots) optimizing configurations occurs at $N\geq7$ for the ellipsoid and $N\geq11$ for the torus. \label{fig:ScalingP}}
\end{figure}

\subsection{The regular part on some general geometries}\label{sec:reg_gen}

In this section, we present results of our method for the computation of the interior $R^i(\bx;\bx)$ and exterior $R^e(\bx;\bx)$ surface regular parts for $\bx\in\partial\Omega$ which are key to addressing the narrow escape problem. The study of \cite{NURSULTANOV2021202} determined that the quantity $R^i(\bx;\bx)$ describes how the location $\bx\in\partial\Omega$ of a boundary window modulates the expected exit time of diffusing particles from an enclosed geometry. In particular, lower values of $R^i(\bx;\bx)$ are associated with faster escape times. We remark that the study \cite{NURSULTANOV2021202} adopted the normalization condition $\int_{\partial \Omega} G^i(\bx;\by)\mathrm{d}A(\bx)=0$ which will modulate the value of $R^i(\bx;\bx)$ while retaining the singularity structure \eqref{eqn:Sing_G}. To demonstrate the potential of this method for the determination of optimizing locations of boundary windows, we calculate these quantities for three geometries. 

The biconcave disk, which approximates the geometry of blood-cells \cite{PozrikidisBloodCell}, is given by the surface 
\bsub\label{blood_cell}
\begin{equation}\label{blood_cell_a}
x(\theta,\phi) =  \alpha \sin \theta \cos \phi, \qquad y(\theta,\phi) =  \alpha \sin \theta \sin \phi \qquad z(\theta,\phi) = \frac{\alpha}{2} \left( b + c \sin^2 \theta - d \sin^4 \theta \right) \cos \theta,
\end{equation}
where $\phi\in[0,2\pi)$ and $\theta\in[0,\pi)$. The shape parameters are given by \cite{PozrikidisBloodCell}
\begin{equation}\label{blood_cell_b}
    \alpha = 1.38581994, \qquad b = 0.207, \qquad  c = 2.003, \qquad d = 1.123.
\end{equation}
\esub
In Fig.~\ref{fig:bloodcell_reg}, we plot solutions of the interior ($R^i(\bx;\bx)$) and exterior ($R^e(\bx;\bx)$) surface regular parts for the geometry \eqref{blood_cell}. In the case of the prolate spheroid with principal semi-axes $(a,b,c) = (1,1,\sqrt{2})$, we show in Fig.~\ref{fig:ellipsoid_reg} the numerical solution of $R^i(\bx;\bx)$ and $R^e(\bx;\bx)$ on the surface. Finally, in Fig.~\ref{fig:blob_reg} we plot the regular parts for a surface given by the unit sphere perturbed by a spherical harmonic, $0.3\cdot  \Re Y_6^4$. This surface has been used to model non-spherical perturbations of real cells \cite{Endres2025,Cavanagh2020}.

\begin{figure}[t!]
    \subfigure[$R^i(\bx;\bx)$ for $\bx\in\partial\Omega$.]{\includegraphics[width=0.475\textwidth]{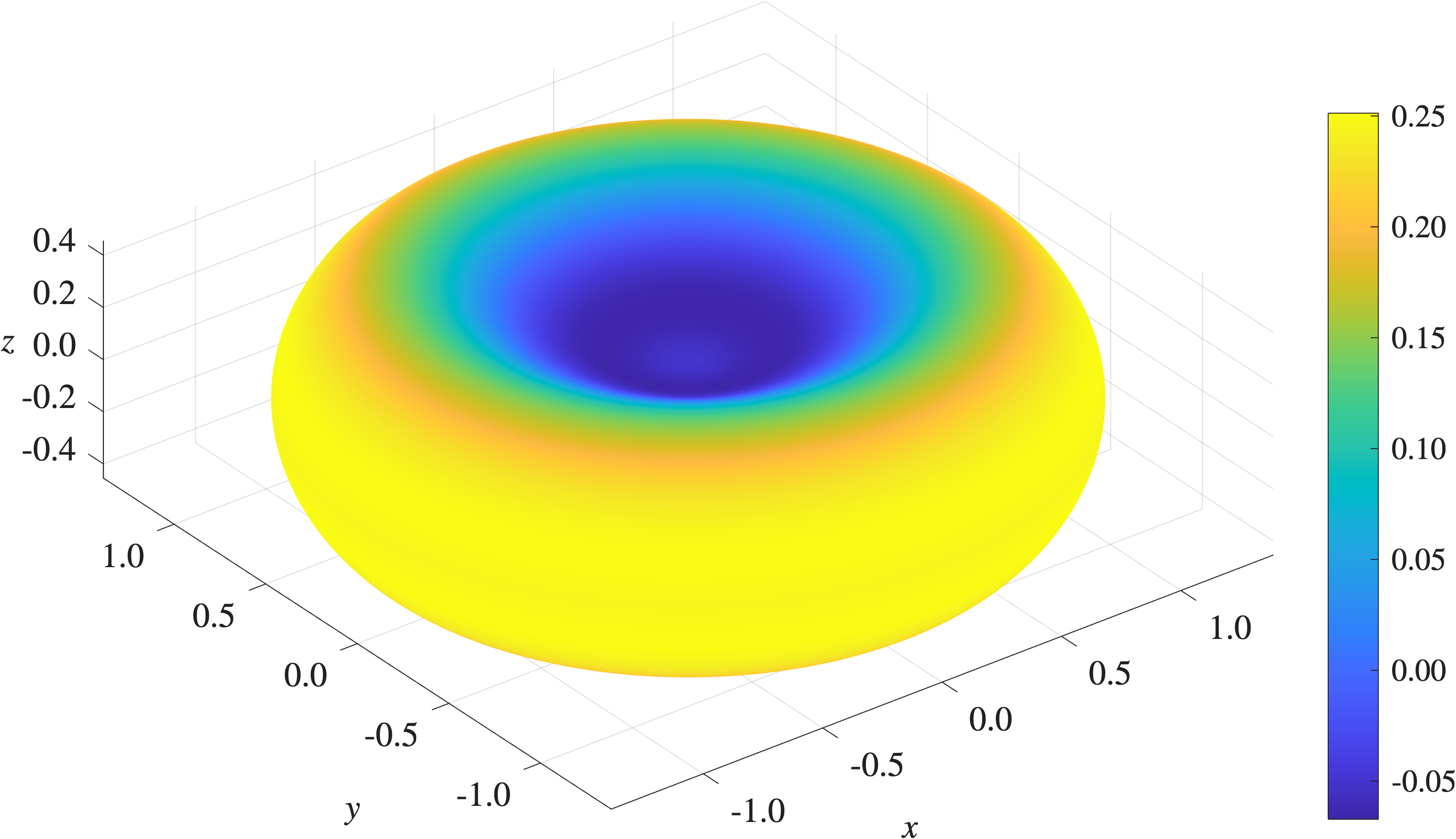}\label{fig:bloodcell_reg_a}}\qquad
    \subfigure[$R^e(\bx;\bx)$ for $\bx\in\partial\Omega$.]{\includegraphics[width=0.475\textwidth]{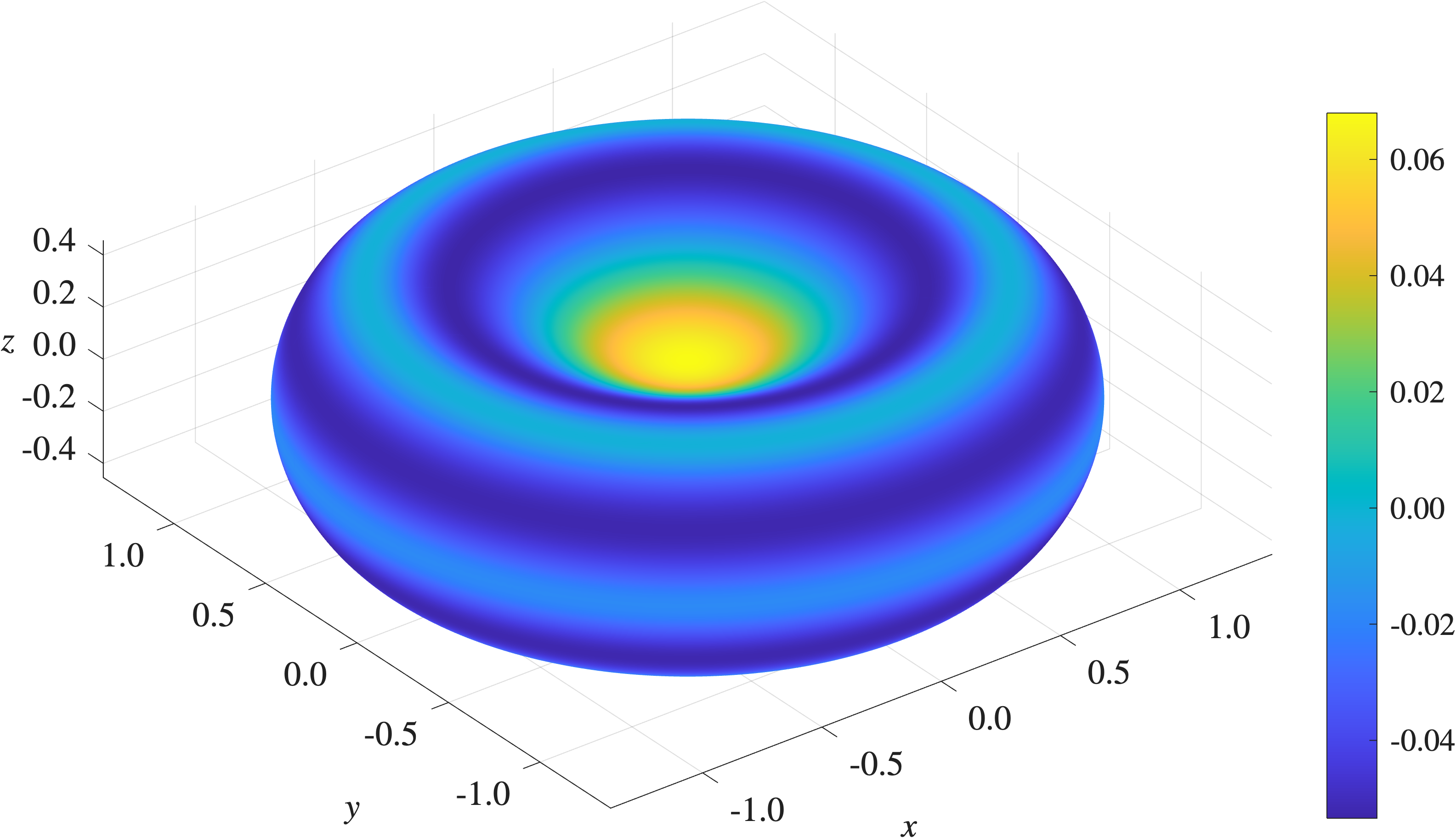} \label{fig:bloodcell_reg_b}}
   \caption{Numerically computed values of the interior (left) and exterior (right) regular parts for the blood cell shape domain \eqref{blood_cell}.
    \label{fig:bloodcell_reg}}
\end{figure}

\begin{figure}[htbp]
   \subfigure[$R^i(\bx;\bx)$ for $\bx\in\partial\Omega$.]{\includegraphics[width=0.475\textwidth]{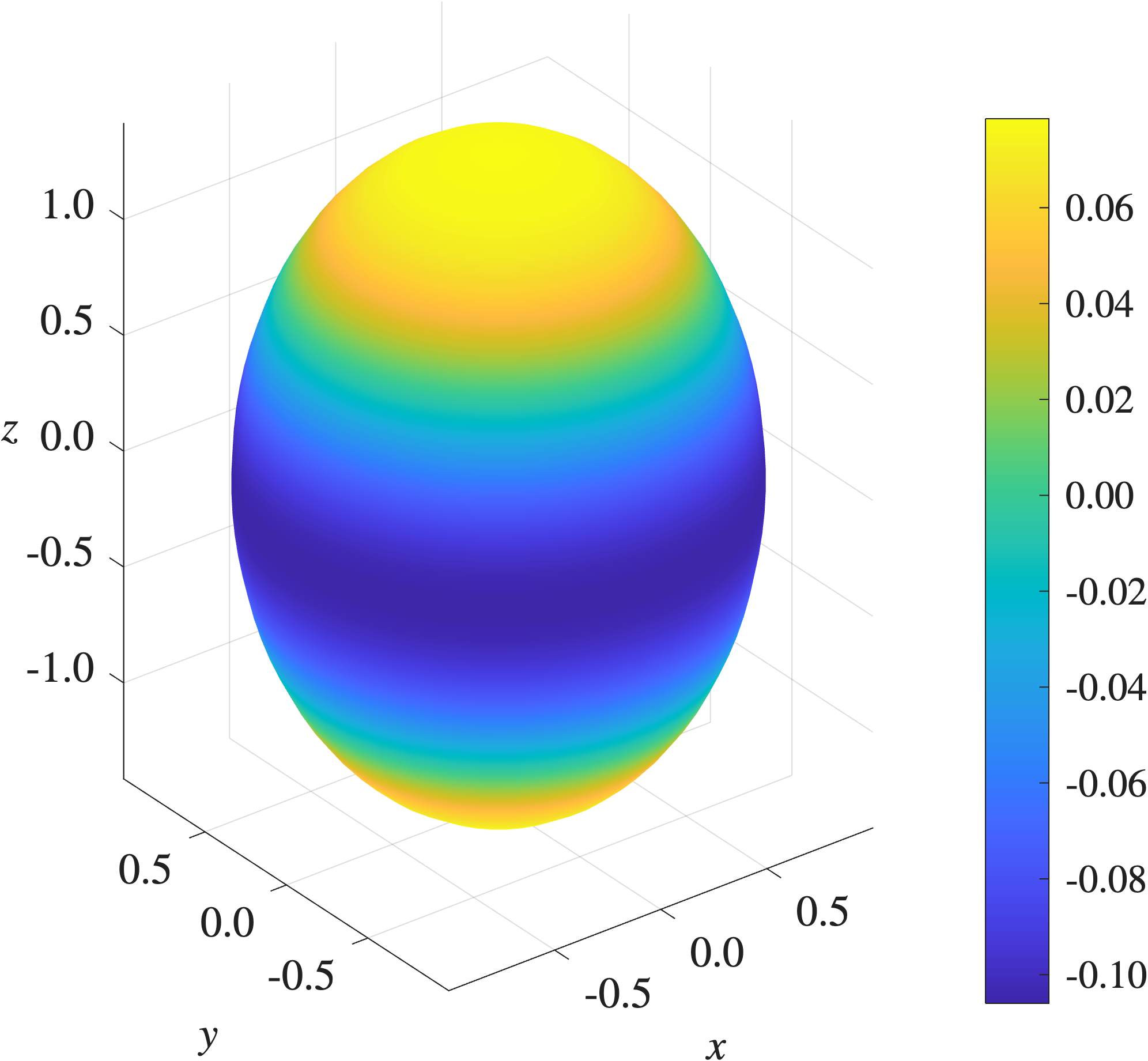}\label{fig:ellipsoid_reg_a}}\qquad
    \subfigure[$R^e(\bx;\bx)$ for $\bx\in\partial\Omega$.]{\includegraphics[width=0.475\textwidth]{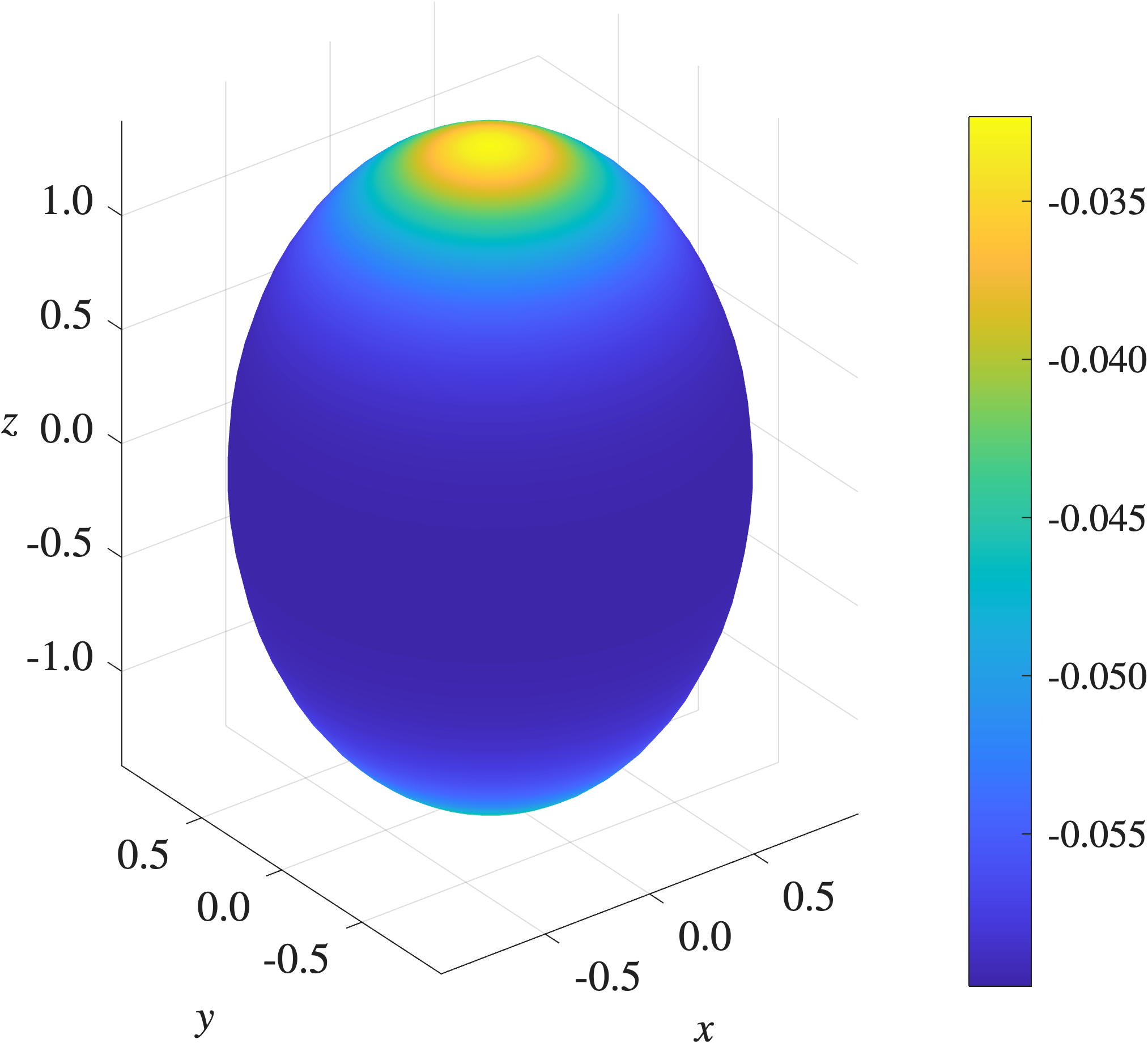}\label{fig:ellipsoid_reg_b}}
   \caption{Numerically computed values of the interior (left) and exterior (right) regular parts of the Neumann function on a prolate spheroid with principal semi-axes $(a,b,c)=(1,1,\sqrt{2})$. \label{fig:ellipsoid_reg} }
\end{figure}

In the case of the blood cell shaped domain, we see in Fig.~\ref{fig:bloodcell_reg_a} that the lowest value of $R^i(\bx;\bx)$ is attained at the dimple where the surface curvature is at a critical point. A similar optimizing preference was observed in \cite{chakraborty2025fastintegralmethodsneumann} for the planar case. In Fig.~\ref{fig:ellipsoid_reg_a} we see that the lowest value of $R^i(\bx;\bx)$ is attained on the equator. We remark that at the poles and the equator, the mean curvatures are
\[
\Hc_{\text{pole}} = -\frac{c}{a^2} = -\sqrt{2}, \qquad \Hc_{\text{eq}} = -\frac{c^2+ a^2}{2ac^2} = -\frac{3}{4},\]
so that $\Hc_{\text{eq}}>\Hc_{\text{pole}}$. We observe a simple correspondence between the maximum and minimum values of $R^i(\bx;\bx)$ and $\Hc(\bx)$ for $\bx\in\partial\Omega$ in the case of a prolate spheroid. Specifically, the value of $R^i(\bx;\bx)$ is minimized at the equator of the prolate spheroid - a point at which the curvature is maximized. For arbitrary geometries, we expect that $R^{i}(\bx;\bx)$ depends on both local and global geometric features in a non-trivial way. It remains an open problem to describe the minimizing set of boundary windows for a general geometry, however, the numerical tool developed here can provide highly accurate estimates of the key quantities involved in this optimization problem.

\begin{figure}[htbp]
    \centering
    \subfigure[$R^i(\bx;\bx)$ for $\bx\in\partial\Omega$.]{\includegraphics[width=0.475\textwidth]{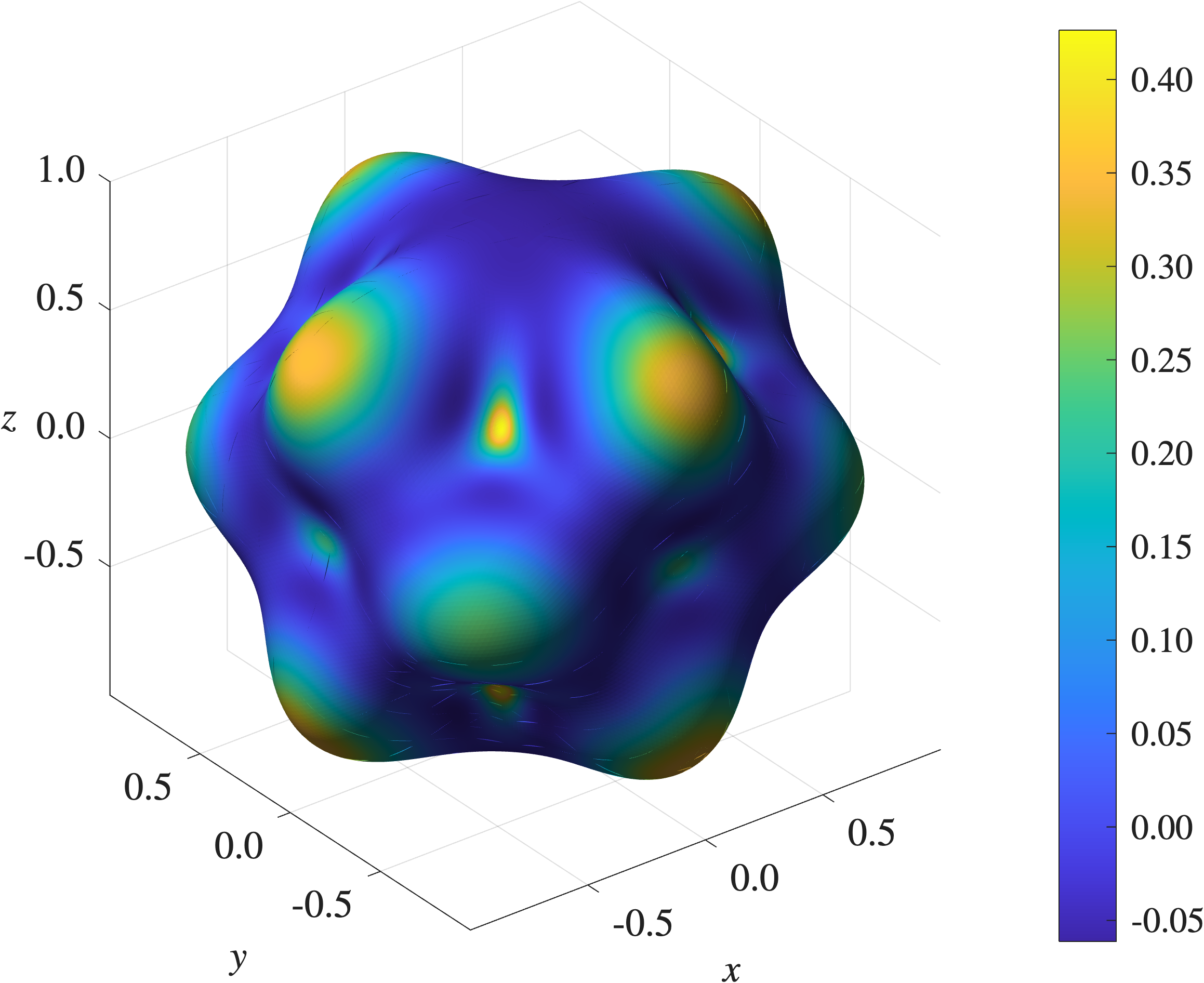}\label{fig:blob_reg_a}}\qquad
    \subfigure[$R^e(\bx;\bx)$ for $\bx\in\partial\Omega$.]{\includegraphics[width=0.475\textwidth]{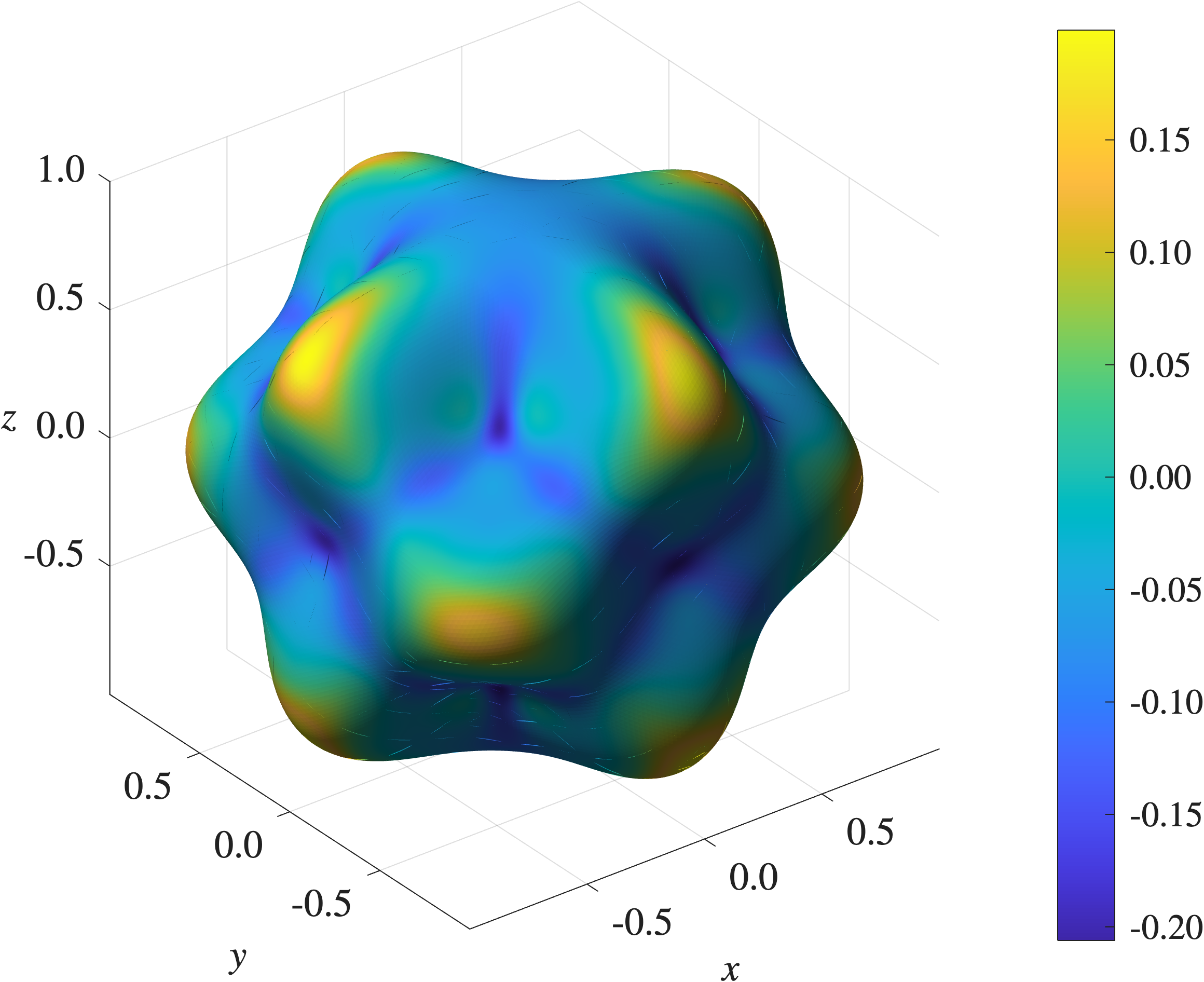}\label{fig:blob_reg_b}}
   \caption{Numerically computed values of the interior (left) and exterior (right) regular parts of the Neumann function on a sphere that has been perturbed by $Y_6^4$. \label{fig:blob_reg} }
\end{figure}

\section{Discussion}\label{sec:Discussion}

In this paper we have considered the three-dimensional Neumann Green's function and provided numerical methodologies to evaluate this quantity for general geometries. In the particular case of a surface source located on a curved geometry, we derived a three term expansion for the singularity structure that depends on local geometric features through the sum and difference of the principal curvatures. With this explicit information, we implemented a high order boundary integral method to recover the remaining regular part of the Green's function. For problems posed either interior or exterior to $\Omega$, with a bulk or surface source, we provide a high-order numerical routine to evaluate the Green's function, its regular part and their gradients. The method was validated on a number of test cases including the sphere, prolate spheroid and a family of constructed geometries. 

As an application of our method to an open problem in diffusion theory, we computed optimizing trap configurations for the narrow capture problem in ellipsoid and toroidal domains. We observed a geometric bifurcation where the optimal placement transitions away from a co-planar structure as $N$ increases. We anticipate that this numerical method will be a useful tool in studying related problems in capture theory, particularly in the open problem of determining the optimal configurations of boundary traps and identifying the limiting configurations as $N\to\infty$.

Singular perturbation methods are a powerful tool for studying problems with strongly localized phenomena, such as spot dynamics in reaction-diffusion systems or narrow escape in diffusion theory \cite{TXKTW2017,NURSULTANOV2021202,Bressloff2024}. In these settings the localized dynamics are globally coupled through the appropriate Neumann Green's function, and for sites localized on the boundary the influence of surface geometry is encoded in the diagonal regular part $R^{i,e}(\bx;\bx)$ for $\bx\in\partial\Omega$. In Section~\ref{sec:reg_gen} we computed this quantity across several representative geometries, demonstrating that the method resolves the geometric dependence relevant to surface-mediated reaction dynamics. The high-order approaches developed here can therefore broaden the range of geometries in which singular perturbation techniques are tractable.

An important avenue for extension of this work is to other operators such as the modified Helmholtz operator which arises when the Laplace transform is applied to the heat equation. This approach is one route \cite{Lindsay2023,BL2025,CL2022,CHERRY2025,LindsayTzou2016,lacroix2025lightningmethodheatequation} to reproduce the time-dependent statistics \cite{Grebenkov_2019} of the dynamics arrival rate at reactive sites. These quantities are essential to determine extreme statistics for first passage time problems \cite{Linn_2022,Morgan2023}.

\section*{Conflict of interest statement} The authors declare no conflicts of interest.

\section*{Acknowledgments}
AEL acknowledges support from the NSF under grant DMS-2052636. JGH was supported in part by a Sloan Research Fellowship. TG was supported by an AMS-Simons Travel Grant. This research was supported in part by grants from the NSF (DMS-2235451) and Simons Foundation (MPS-NITMB-00005320) to the NSF-Simons National Institute for Theory and Mathematics in Biology (NITMB).

\bibliographystyle{siamplain}
\bibliography{numerical_ref}

\end{document}